\documentclass[reqno, a4paper,12pt]{amsart}

\usepackage[capposition=bottom]{floatrow}
\usepackage{euscript,eufrak,verbatim}
\usepackage{graphicx}
\usepackage[usenames]{color}
\usepackage[colorlinks,linkcolor=red,anchorcolor=blue,citecolor=blue]{hyperref}
\usepackage{amsmath, mathtools}
\usepackage{mathtools}
\usepackage{amsthm}
\usepackage[all]{xy}
\usepackage{amssymb} 
\usepackage{bm}
\usepackage{booktabs,siunitx}

\usepackage[all]{xy}

\usepackage{mathrsfs}
\usepackage{amscd}

\usepackage{enumitem}
\usepackage[numbers]{natbib}

\usepackage{mathptmx}
\usepackage[T1]{fontenc}

\usepackage{lipsum}
\usepackage{accents}
\usepackage{titlesec}
\usepackage{titletoc}
\usepackage{calrsfs}

\makeatletter
\numberwithin{equation}{section}

\theoremstyle{cuplain}
\newtheorem{main theorem}{Main Theorem}
\newtheorem{theorem}{Theorem}[section]
\newtheorem{lemma}[theorem]{Lemma}

\newtheorem{proposition}[theorem]{Proposition}
\newtheorem{claim}[theorem]{Claim}

\newtheorem*{main theorem*}{Main Theorem}
\newtheorem*{theorem*}{``Theorem''}
\theoremstyle{definition}
\newtheorem{definition}[theorem]{Definition}
\newtheorem{remark}[theorem]{Remark}

\newtheorem*{example*}{Example}
\newtheorem*{remark*}{Remark}

\newtheoremstyle{break}
{\topsep}{\topsep}%
{\itshape}{}%
{\bfseries}{}%
{\newline}{}%
\theoremstyle{break}

\newtheoremstyle{break}
{\topsep}{\topsep}%
{\normalshape}{}%
{\bfseries}{}%
{\newline}{}%

\numberwithin{equation}{section}
\newcommand{\vep}{\varepsilon}

\newcommand{\setcond}{\;\middle|\;}

\DeclareFontFamily{U}{stix2bb}{\skewchar\font127 }
\DeclareFontShape{U}{stix2bb}{m}{n} {<-> stix2-mathbb}{}
\DeclareMathAlphabet{\mathbb}{U}{stix2bb}{m}{n}
\DeclareMathAlphabet{\pazocal}{OMS}{zplm}{m}{n}

\newcommand{\RN}[1]{%
\textup{\uppercase\expandafter{\romannumeral#1}}%
}

\newcommand{\SL}{\mathrm{SL}}
\newcommand{\GL}{\mathrm{GL}}

\newcommand{\supp}{\mathrm{supp}}
\newcommand{\RP}{\mathbb{RP}}
\newcommand{\closure}[1]{\overline{#1\hspace{-0.75pt}}\hspace{0.75pt}}
\newcommand{\Log}{\mathrm{Log}}

\titleformat{\section}
{\normalfont\bfseries\Large} % font for section titles
{\thesection} % section number
{1em} % space between number and title
{\bfseries} % makes the title bold
\titleformat{\subsection}
{\normalfont\bfseries\normalsize}
{\thesubsection} 
{1em} 
{\bfseries}

\titlecontents{section}
[2.0em]
{}
{\contentslabel{2.0em}}
{}
{\titlerule*[0.5pc]{.}\contentspage}

\titlecontents{subsection}
[4.7em]
{}
{\contentslabel{3.0em}}
{}
{\titlerule*[0.5pc]{.}\contentspage}

\begin{document}

%%%title decoration%%%%%%%%%%%%%%%%

\newcommand\titlelowercase[1]{\texorpdfstring{\lowercase{#1}}{#1}}

%defines the size of author name
\font\mathptmx=cmr12 at 12pt

%%%%%%%%%%%%%%%%%%%%%%%%%%%%%%%

% Use \protect{\\[5pt]} to break lines within title

\title[\fontsize{13}{12}\mathptmx {\it{L\titlelowercase{yapunov exponents for uniformly hyperbolic random matrix products}}}]{\LARGE L\titlelowercase{yapunov exponents for} \protect{\\[9pt]}\titlelowercase{uniformly hyperbolic random matrix products}}

%available font sizes for title
%\tiny\scriptsize\footnotesize\small\normalsize\footnotesize\large\Large\LARGE\huge\Huge

\author[\fontsize{13}{12}\mathptmx {\it{N\titlelowercase{ima} A\titlelowercase{libabaei}}}]{\fontsize{13}{12}\mathptmx N\titlelowercase{ima} A\titlelowercase{libabaei}}

\subjclass{37H15, 15B48, 28A80}

\keywords{Lyapunov exponent, random matrix products, uniform hyperbolicity}

\maketitle

\begin{abstract}
We consider a finite family of invertible $2 \times 2$ real matrices and a transitive Markov shift on the index set. Let $\lambda$ be the top Lyapunov exponent for random matrix products driven by the Markov shift. We prove that, if the matrices are projectively uniformly hyperbolic with respect to the Markov shift, then $\lambda$ admits an explicit representation in terms of an infinite matrix. This rapidly convergent representation yields a polynomial-time algorithm for approximating $\lambda$: only $O\big( (\log(1/\varepsilon))^3 \big)$ arithmetic operations are needed to achieve error $\varepsilon$. Furthermore, $\lambda$ depends real analytically on the matrix entries and the transition probabilities near a projectively uniformly hyperbolic system, and each Taylor coefficient can be approximated in polynomial time.
\end{abstract}

\tableofcontents

\section{Introduction}

\subsection{Preliminaries and main results}

In this article we study the top Lyapunov exponent of random matrix products along a Markov shift, and obtain an explicit representation for it in terms of an infinite matrix. This yields a method for its computation that is highly effective in practice: it has complexity that is only polynomial in $\log(1/\vep)$ to achieve error $\vep$. Moreover, the representation is robust enough that, under real-analytic perturbations of the underlying data, it yields effective computation of each Taylor coefficient of the Lyapunov exponent.

Let $\{A_i\}_{i \in X}$ be a finite family of invertible $2\times 2$ real matrices, let $\Sigma \subset X^{\mathbb N_0}$ be a transitive topological Markov shift, and let $\mu$ be a Markov measure on $\Sigma$ associated to a transition matrix on the alphabet. By Furstenberg--Kesten \cite{Furstenberg--Kesten}, the following limit exists and is constant for $\mu$-almost every $x = (x_n)_{n=0}^\infty \in \Sigma$.
\[
\lambda \;=\; \lim_{n \to \infty} \frac1n \log \| A_{x_{n-1}} \cdots A_{x_0} \|.
\]
This number is called the \emph{(top) Lyapunov exponent}.

Although its existence is classical, explicit calculation is notoriously difficult even for products of finitely many $2\times 2$ matrices. As Kingman \cite[p.897]{Kingman} wrote, “Pride of place among the unsolved problems of subadditive ergodic theory must go to the calculation of [the Lyapunov exponent]”. Nearly two decades later, Peres returned to what he called “the excruciating problem of the subject” and posed the challenge: “Devise reasonably general and effective algorithms for explicit calculation (or at least approximation) of Lyapunov exponents.'' \cite[Question 3]{Peres92}.

A major breakthrough was obtained by Pollicott \cite{Pollicott10} for i.i.d.\ products of (entrywise) \emph{positive} matrices. By using transfer operators, he expressed the Lyapunov exponent in terms of the associated Fredholm determinant. For $2\times 2$ positive matrices, Pollicott's method yields an algorithm with subexponential complexity:
\[
O\left(\exp\big(\alpha \sqrt{\log(1/\varepsilon)}\,\big)\right)
\]
arithmetic operations to achieve error $\varepsilon$, for some $\alpha>0$. The bounds were later refined by Jurga--Morris \cite{Jurga-Morris} and Pollicott \cite{Pollicott20}, with the same asymptotic complexity order.

One natural geometric extension of the positivity condition is \emph{projective uniform hyperbolicity}. Roughly speaking, this means that the cocycle admits a splitting into dominant/weak directions that vary continuously (Definition \ref{def: uniform hyperbolicity}). By the work of Avila--Bochi--Yoccoz \cite{Avila--Bochi--Yoccoz} and its recent continuation by Duarte--Dur\~aes--Graxinha--Klein \cite{DDGK25}, projective uniform hyperbolicity is equivalent to the existence of an invariant \emph{multicone} in projective dynamics (Theorem \ref{thm: multicone criterion}). If all matrices are strictly positive, then the projective positive cone gives a one-component multicone.

\medskip

The main purpose of this paper is to show that, under projective uniform hyperbolicity, the Lyapunov exponent admits an explicit and constructive representation. Let $\mathbb N = \{1,2,3,\ldots\}$ and $\mathbb N_0 = \mathbb N \cup \{0\}$. Also, let $\ell^\infty(\mathbb N_0)$ be the set of bounded sequences of complex numbers, indexed by $\mathbb N_0$.

\begin{theorem} \label{thm: explicit representation of the Lyapunov exponent in the introduction}
Suppose that $\{A_i\}_{i \in X}$ is projectively uniformly hyperbolic with respect to $\Sigma$, and that a multicone is given. Then, there are a finite set $C$, an infinite matrix operator \hspace{0.4pt}$\mathrm T$ acting on a subspace of \hspace{0.7pt}$\ell^\infty(\mathbb N_0)^C$, and a vector $\boldsymbol v \in \ell^\infty(\mathbb N_0)^C$ such that
\[
\lambda \;=\; \sum_{n = 0}^\infty \big[ \mathrm T^n \boldsymbol v \big]_{\boldsymbol 0},
\]
where $[(w_r)_{r\in C}]_{\boldsymbol 0}$ is a weighted average $\sum_{r \in C} \pi_r (w_r)_0$\, for some probability vector $(\pi_r)_{r \in C}$. All the data $C$, $\mathrm T$, $\boldsymbol v$, and $(\pi_r)_{r\in C}$ are explicitly defined and computable.
\end{theorem}

By truncating the infinite matrix $\mathrm T$, one can calculate $\lambda$ to arbitrary precision in \emph{polynomial} time.

\begin{theorem} \label{thm: polynomial computation of the Lyapunov exponent in introduction}
If $\{A_i\}_{i \in X}$ is projectively uniformly hyperbolic with respect to $\Sigma$, we can approximate \hspace{1pt}$\lambda$ within error $\vep > 0$ with
\[
O\Big( (\log(1/\vep))^3 \Big)
\]
arithmetic operations.
\end{theorem}

By combining with Mairesse's theory \cite{Mairesse} on traffic equations for nearest-neighbor random walks, we can even calculate the Lyapunov exponent of a certain system that is \emph{not} projectively uniformly hyperbolic. See Subsection \ref{subsection: example 1}. We remark that in this paper we do not address the separate problem of constructing multicones for a given system.

\medskip

The question of how $\lambda$ depends on the underlying data has its own long history. In a seminal paper, Ruelle \cite{Ruelle79} established analyticity of the top Lyapunov exponent with respect to the matrix entries for fixed base measure, under strict cone invariance. In a different direction, Peres \cite{Peres92} proved that for random positive matrices with Markovian dependence, the Lyapunov exponent depends real-analytically on the transition probabilities. Malheiro--Viana \cite{Malheiro--Viana} proved that, under a Markov law, the Lyapunov exponent is continuous when both matrix entries and transition probabilities vary.

Using the representation above, we obtain a joint real-analyticity statement near a projectively uniformly hyperbolic point in the Markov setting: both the matrix entries and the transition probabilities are allowed to vary.

\begin{theorem} \label{thm: real analyticity of Lyapunov exponent}
Let $\Sigma$ be a transitive topological Markov shift. Suppose that the matrices $A(t) = \{A_i(t)\}_{i\in X}$ and the transition probabilities $P(t) = \big(P_{ij}(t)\big)_{i,j\in X}$ vary real-analytically with $t$, without changing the underlying $\Sigma$. If the matrices are projectively uniformly hyperbolic with respect to $\Sigma$ at $t=t_0$, then the Lyapunov exponent
\[
t \;\; \longmapsto \;\; \lambda\big( A(t), P(t) \big)
\]
is real analytic around $t_0$. Furthermore, there is an explicit complex neighborhood $\Omega\ni t_0$ to which $\lambda$ extends holomorphically. Consequently, its Taylor expansion at $t_0$ converges uniformly on any closed disk $D\subset \Omega$ centered at $t_0$.
\end{theorem}

The proof readily extends to several parameters.

Moreover, the analytic theory obtained here is not merely qualitative. The infinite-matrix representation exposes not only the value of the Lyapunov exponent, but also its full local Taylor expansion, and it does so with effective control on the computational cost.

\begin{theorem} \label{thm: introduction: every derivative can be computed in polynomial time}
Under the setting of Theorem \ref{thm: real analyticity of Lyapunov exponent}, for every fixed integer $q\geq 1$,
\[
\left.\frac{d^q}{dt^q}\right|_{t=t_0} \lambda\big( A(t), P(t) \big)
\]
can be approximated to error less than $\vep>0$ with $O\big( (\log(1/\vep))^3 \big)$ arithmetic operations.
\end{theorem}

We remark that the previous results by Ruelle and Peres generally do not allow an explicit computation of derivatives.

\subsection{Roadmap of the paper}

We now describe the roadmap of this paper. In Section \ref{section: Preparation} we introduce the basic notions and metrics we will be working with. Projective uniform hyperbolicity is known to be equivalent to the existence of a \emph{multicone}, a family of subsets of the real projective line capturing the dynamics. This was proved for $\SL_2(\mathbb R)$ matrices in \cite{Avila--Bochi--Yoccoz}, and generalized to $\GL_2(\mathbb R)$ matrices in \cite{DDGK25}.

Theorem \ref{thm: explicit representation of the Lyapunov exponent in the introduction} is proven in Section \ref{section: Explicit representation of the Lyapunov exponent}. Using the multicone equivalence, we lift the Markov shift so that it remembers the location within the multicone (Subsection \ref{subsection: Branch-state extension}). This lifted shift is also Markov, and yields the same Lyapunov exponent. After passing to local charts on the lift, the matrices become positive (Subsection \ref{subsection: Local charts and Markov operator}).

We then apply the Malheiro--Viana formalism \cite{Malheiro--Viana} to the lift (Subsection \ref{subsection: Malheiro--Viana formalism for the lifted cocycle}). It is classically known that, under an i.i.d.\ law, the Lyapunov exponent is equal to the supremum of integrals against stationary measures \cite{Furstenberg--Kifer}. In \cite{Malheiro--Viana}, the direct analogue of this variational formula is proved in the Markov shift setting, where stationary measures are replaced by \emph{stationary measure vectors}. We prove that the maximizing measure vector for the variational formula corresponds to the unique fixed point of a Markov operator $\mathscr H$ acting on measure vectors.

In Subsection \ref{subsection: Kernel expansion}, we carry out \emph{Kernel expansion}, a method first introduced in \cite{Alibabaei26}. We consider the integrand in the variational formula for $\lambda$, and expand it using \emph{Kernel functions} that integrate to $0$ against stationary measure vectors. (See \cite[Section $2$]{Alibabaei26-2} for a heuristic explanation.) This allows us to rewrite the integral using an explicitly constructed infinite matrix $\mathrm T$, and the convergence to $\lambda$ follows by the contraction of the Markov operator $\mathscr H$. This completes the proof of Theorem \ref{thm: explicit representation of the Lyapunov exponent in the introduction}.

\smallskip

Theorem \ref{thm: polynomial computation of the Lyapunov exponent in introduction} is proven in Section \ref{section: Evaluation and Polynomial-time Convergence}. The infinite matrix $\mathrm T$ is truncated to a finite matrix $\mathrm T_{m\times m}$. Following the idea in \cite{Alibabaei26-2}, we obtain integral representations of $\mathrm T$ and $\mathrm T_{m\times m}$ (Subsection \ref{subsection: Integral representations}). The integrals preserve the structure well enough that they allow us to control the difference of iterates $\mathrm T^n - \mathrm T_{m\times m}^n$ (Subsection \ref{subsection: Integral representations for iterates}). We introduce associated integral operators in Subsection \ref{subsection: Integral operators and telescoping}, and this leads to a bound that is exponentially small in truncation parameters.

\smallskip

Section \ref{section: Analytic dependence on the underlying data} is devoted to the proof of real analyticity of $\lambda$ (Theorem \ref{thm: real analyticity of Lyapunov exponent}).  The proof proceeds by considering the holomorphic extension of the relevant functions, and the explicit domain of holomorphic extension is given.

\smallskip

Polynomial-time calculation of derivatives (Theorem \ref{thm: introduction: every derivative can be computed in polynomial time}) is proven in Section \ref{section: Calculating the derivatives}. For the first derivative, an error bound that does not depend on the domain of holomorphic extension is given in Subsections \ref{First-order derivative I: gathering bounds} and \ref{First-order derivative II: integral operators and final bounds}. An error bound for derivatives of arbitrary order is proven using Cauchy's integral formula in Subsection \ref{subsection: Higher-order derivatives}. Finally, Appendix \ref{section: appendix} records the proof that one can reduce to the aperiodic Markov case (Proposition \ref{prop: aperiodic reduction}).

\medskip

In the rest of this section we present examples that demonstrate the core ideas of the calculations. Natural applications such as random walks on surface groups will be treated in future work.

\subsection{Example 1} \label{subsection: example 1}

Let $X = \{x, y, \bar{x}, \bar{y}\}$, and denote by $\lambda_{\mathrm{orig}}$ the Lyapunov exponent of the uniform i.i.d.\ products of the following matrices:
\[
A_x = \begin{pmatrix} 1 & 0 \\ 1 & 4 \end{pmatrix},
\quad
A_y = \begin{pmatrix} 4 & 3 \\ 1 & 1 \end{pmatrix},
\quad
A_{\bar{x}} = A_x^{-1} = \begin{pmatrix} 1 & 0 \\ -\frac14 & \frac14 \end{pmatrix},
\quad
A_{\bar{y}} = A_y^{-1} = \begin{pmatrix} 1 & -3 \\ -1 & 4 \end{pmatrix}.
\]
Using our method, the rigorously certified value, displayed here to 128 decimal places, is
\begin{align*}
\lambda_{\mathrm{orig}}
\;=\;
0.&4426362721184141506898808096914817304882428124927663040078064914 \\
   &4441947676929179213894134514076767635962327236955401026667152037\cdots.
%Guaranteed digits 0.44263627211841415068988080969148173048824281249276630400780649144441947676929179213894134514076767635962327236955401026667152037
%Presumably correct 0.44263627211841415068988080969148173048824281249276630400780649144441947676929179213894134514076767635962327236955401026667152037199137599961725118827395
\end{align*}
We emphasize that this example was not chosen to have an exceptionally strong contraction. The value of $\rho$ (which governs the convergence rate for the algorithm) is about $0.78$, which is quite moderate. Thus the 128 certified digits above are not a consequence of an artificially small contraction factor, but rather of the explicit infinite-matrix representation and its error estimation.
\medskip

Now, the original system is not projectively uniformly hyperbolic, as $A_x A_{\bar{x}} = I$ and $x\to \bar{x}$ is admissible. However, Mairesse's theory \cite{Mairesse} on Markovian harmonic measure allows us to apply our method, as follows. Denote by $\lambda_{\mathrm{red}}$ the Lyapunov exponent of the random products of $\{A_i\}_{i\in X}$ over reduced words; from any current letter you move with probability $1/3$ to each of the three letters different from its inverse. Then, we have
\[
\lambda_{\mathrm{orig}}
\;=\;
\frac12 \lambda_{\mathrm{red}}.
\]

Mairesse proved that the original i.i.d.\ random walk on $X$ can be encoded by a Markov process on infinite reduced words. This Markov process is characterized by a unique solution of the \emph{traffic equations} \cite[Theorem 4.5]{Mairesse}, and the \emph{drift} describes the average amount of cancellation in typical products. By solving the traffic equations in the present example, we obtain the above relation, where the factor $1/2$ is the drift.

\medskip

We now explain in detail how one computes $\lambda_{\mathrm{red}}$. The transition matrix $P$ for the reduced-word system is
\[
P=
\frac{1}{3}
\begin{pmatrix}
1 & 1 & 0 & 1 \\
1 & 1 & 1 & 0 \\
0 & 1 & 1 & 1 \\
1 & 0 & 1 & 1
\end{pmatrix},
\]
with rows and columns ordered as $(x, y, \bar{x}, \bar{y})$, and the transition probability from $i$ to $j$ is $P_{ij}$.

We employ the following slope coordinate for the real projective space $\RP^1$:
\[
\RP^1 \;\to\; \mathbb R \cup \{\infty\},
\qquad
\left[ \begin{pmatrix} v_1 \\ v_2 \end{pmatrix} \right] \; \longmapsto \; \frac{v_1}{v_2}.
\]
We define an open interval $M_i\subset \RP^1$ for each $i\in X$, written in the coordinate above.
\[
M_x=\left(-\frac{5}{12},\;\frac{31}{30}\right),
\hspace{16pt}
M_y=\left(\frac{121}{60},\;10\right),
\hspace{16pt}
M_{\bar{x}}=\left(-13,\;-\frac{13}{10}\right),
\hspace{16pt}
M_{\bar{y}}=\left(-1,\;-\frac{19}{30}\right).
\]
If one is currently at $x$, the only letters that can follow are $x, y, \bar{y}$. Writing $[A]$ for the projective action of $A\in \GL_2(\mathbb R)$, we have $\overline{[A_j] (M_x)} \subset M_j$ for $j\in \{x, y, \bar{y}\}$. Indeed,
\[
[A_x](M_x) = \left(-\frac{5}{43},\;\frac{31}{151}\right),
\qquad
[A_y](M_x) = \left(\frac{16}{7},\;\frac{214}{61}\right),
\qquad
[A_{\bar{y}}](M_x) = \left(-\frac{41}{53},\;-\frac{59}{89}\right).
\]
Likewise, for every $i,j\in X$ with $P_{ij} > 0$, one checks that $\overline{[A_j](M_i)} \subset M_j$. Therefore, $\{M_i\}_{i\in X}$ forms a multicone.

Since each $M_i$ has a single connected component, the branch-state set has exactly $12$ states. Namely,
\begin{align*}
R
&= \left\{(i,j)\in X^2 \setcond j \ne \bar{i}\, \right\}.
%\\ &= \{ (x,x), (x,y), (x,\bar{y}), (\bar{x},y), (\bar{x},\bar{x}), (\bar{x},\bar{y}), (y,x), (y,y), (y,\bar{x}), (\bar{y},x), (\bar{y},\bar{x}), (\bar{y},\bar{y}) \}.
\end{align*}
The transition matrix $Q=(Q_{r,r'})_{r,r'\in R}$ on $R$ is defined as
\[
Q_{(i,j),\,(k,\ell)}
=
\begin{cases}
\frac13 & \text{if } j=k,\\
0 & \text{otherwise.}
\end{cases}
\]
We pick a recurrent class $C\subset R$ (Definition \ref{def: recurrent class}). In this example the unique recurrent class is $C=R$. (It is also aperiodic.) Now define the following matrices:
\[
L_x
=
\begin{pmatrix}
-\frac{5}{12} & \frac{93}{200}\\[4pt]
1 & \frac{9}{20}
\end{pmatrix},
\qquad
L_y
=
\begin{pmatrix}
\frac{121}{60} & \frac{5}{2}\\[4pt]
1 & \frac{1}{4}
\end{pmatrix},
\qquad
L_{\bar{x}}
=
\begin{pmatrix}
-13 & -\frac{39}{5}\\[4pt]
1 & 6
\end{pmatrix},
\qquad
L_{\bar{y}}
=
\begin{pmatrix}
-1 & -\frac{19}{20}\\[4pt]
1 & \frac{3}{2}
\end{pmatrix}.
\]
Each $\big[L_i\big]$ is the local chart that sends the non-negative cone to $\closure{M_i}$. For each $r=(i,j)\in C$, define
\[
B_r
\;=\;
B_{i,j}
\;=\;
\vep(r)\,L_j^{-1}A_j L_i,
\]
where $\vep(r)\in \{\pm1\}$ is chosen so that $B_r$ is a positive matrix. Then the local positive matrices are
\[
\begin{alignedat}{3}
B_{x,x}
&=
\begin{pmatrix}
\frac{1483}{522} & \frac{3751}{2900}\\[6pt]
\frac{3875}{2349} & \frac{1127}{522}
\end{pmatrix},
\qquad
&
B_{x,y}
&=
\begin{pmatrix}
\frac{270}{479} & \frac{1782}{2395}\\[6pt]
\frac{113}{1437} & \frac{16377}{23950}
\end{pmatrix},
\qquad
&
B_{x,\bar{y}}
&=
\begin{pmatrix}
\frac{223}{132} & \frac{237}{2200}\\[6pt]
\frac{20}{11} & \frac{9}{11}
\end{pmatrix},
\\[6pt]
B_{\bar{x},y}
&=
\begin{pmatrix}
\frac{4260}{479} & \frac{288}{479}\\[6pt]
\frac{5952}{479} & \frac{11484}{2395}
\end{pmatrix},
\qquad
&
B_{\bar{x},\bar{x}}
&=
\begin{pmatrix}
\frac{13}{18} & \frac{17}{60}\\[6pt]
\frac{25}{54} & \frac{19}{36}
\end{pmatrix},
\qquad
&
B_{\bar{x},\bar{y}}
&=
\begin{pmatrix}
\frac{157}{11} & \frac{849}{55}\\[6pt]
\frac{20}{11} & \frac{120}{11}
\end{pmatrix},
\\[6pt]
B_{y,x}
&=
\begin{pmatrix}
\frac{7561}{2610} & \frac{67}{87}\\[6pt]
\frac{16285}{2349} & \frac{4750}{783}
\end{pmatrix},
\qquad
&
B_{y,y}
&=
\begin{pmatrix}
\frac{1146}{479} & \frac{1005}{479}\\[6pt]
\frac{17939}{7185} & \frac{1249}{479}
\end{pmatrix},
\qquad
&
B_{y,\bar{x}}
&=
\begin{pmatrix}
\frac{1349}{9360} & \frac{283}{1872}\\[6pt]
\frac{103}{5616} & \frac{385}{5616}
\end{pmatrix},
\\[6pt]
B_{\bar{y},x}
&=
\begin{pmatrix}
\frac{82}{29} & \frac{3701}{870}\\[6pt]
\frac{100}{261} & \frac{1385}{783}
\end{pmatrix},
\qquad
&
B_{\bar{y},\bar{x}}
&=
\begin{pmatrix}
\frac{7}{234} & \frac{41}{3120}\\[6pt]
\frac{55}{702} & \frac{187}{1872}
\end{pmatrix},
\qquad
&
B_{\bar{y},\bar{y}}
&=
\begin{pmatrix}
\frac{25}{11} & \frac{629}{220}\\[6pt]
\frac{20}{11} & \frac{30}{11}
\end{pmatrix}.
\end{alignedat}
\]

Write $\lambda_+(B,Q)$ to denote the Lyapunov exponent of random products of the matrices $\{B_r\}_{r\in R}$ along the Markov shift on $C=R$ with transition probability $Q$. Then, the following holds. (Lemma \ref{lemma: Lyapunov exponent identity for the lift})
\[
\lambda_{\mathrm{red}} \;=\; \lambda_+(B,Q).
\]

\medskip

From here, we perform Kernel expansion. First, we identify the non-negative cone with the following $1$-simplex.
\[
\Delta = \left\{ v = \begin{pmatrix} v_1 \\ v_2 \end{pmatrix} \in \mathbb{R}^2 \setcond v_1, v_2 \geq 0, v_1 + v_2 = 1 \right\}.
\]
We define the chart
\[
\psi:[-1, 1] \to \Delta, \qquad \psi(x) = \frac{1}{2} \begin{pmatrix} 1+x \\ 1-x \end{pmatrix}.
\]
Then, the projective action $[B_r]: \Delta \to \Delta$ for each $r\in C$ corresponds to the M\"obius transformation $f_r: [-1,1]\to [-1,1]$, where
\[
f_r(x) =
\frac{(a_r-b_r-c_r+d_r)x + (a_r+b_r-c_r-d_r)}{(a_r-b_r+c_r-d_r)x + (a_r+b_r+c_r+d_r)}
\quad \text{for} \quad
B_r
=
\begin{pmatrix} a_r & b_r \\ c_r & d_r \end{pmatrix}.
\]
Also, define the ``transposed'' M\"obius transformation $f_r^\top$ by
\[
f_r^\top(x)
=
\frac{
\alpha x + \gamma}{
\beta x + \delta}
\quad \text{where} \quad
f_r(x)
=
\frac{
\alpha x + \beta}{
\gamma x + \delta}.
\]
Let $f_r'$ denote the usual differentiation of $f_r$. We define an infinite matrix $T_r=\big(b^{(r)}_{k,n}\big)_{k,n\in\mathbb N_0}$ by
\begin{align*}
& b^{(r)}_{k,0} = 0 \quad \text{for } k \geq 0, \hspace{45pt}
b^{(r)}_{0,n} = {\big( f_r^{\top}(0) \big)}^n \quad \text{for } n\geq1,\\
& b^{(r)}_{k,n} = \sum_{\ell=1}^{\min\{k,n\}} \binom{n}{\ell}\binom{k-1}{\ell-1} \, \, 
{\big( f_r^{\top}(0) \big)}^{n-\ell}\big(-f_r(0)\big)^{k-\ell} {\big( f_r'(0) \big)}^{\ell}\quad \text{for } k,n\geq1.
\end{align*}
Then, the operator $\mathrm T$ is defined as follows: for $r'\in C$ and $\boldsymbol u =(u_r)_{r\in C}\in \ell^\infty(\mathbb N_0)^C$ belonging to the domain of $\mathrm T$ introduced later,
\[
\big( \mathrm T \boldsymbol u \big)_{r'}
\;=\;
\frac13
\sum_{r\in I(r')}
T_r u_r,
\]
where $I(r')$ is the set of states $(i,j)\in C$ with second letter equal to the first letter of $r'$. Finally, set the seed vector $\boldsymbol v = (\boldsymbol v_{r'})_{r' \in C} \in \ell^\infty(\mathbb N_0)^C$ as
\begin{equation*}
\big( \boldsymbol v_{r'} \big)_n =
\begin{dcases}
\frac13 \sum_{r\in I(r')} \log \left( \frac{a_r+b_r+c_r+d_r}{2} \right) & (n = 0), \\
-\frac13 \sum_{r\in I(r')} \frac{\big( -f_r(0) \big)^n}{n} & (n \geq 1).
\end{dcases}
\end{equation*}
\medskip
The stationary probability vector is $(\pi_r)_{r\in C} = \big(\frac{1}{12}\big)_{r\in C}$. Then, by Theorem \ref{thm: infinite sum equals lambda},
\[
\lambda_{\mathrm{red}}
\;=\;
\lambda_+(B,Q)
\;=\;
\sum_{n = 0}^\infty \big[ \mathrm T^n \boldsymbol v \big]_{\boldsymbol 0}
\;=\;
\frac{1}{12} \sum_{r\in C} \; \sum_{n = 0}^\infty \Big( \big( \mathrm T^n \boldsymbol v \big)_r \Big)_0
\]

\smallskip

By direct calculation we see that $f_r([-1,1]) \subset [-\rho,\rho]$ for every $r\in C$, where
\[
\rho=\frac{279}{359} = 0.777158\cdots.
\]
This $\rho$ governs the convergence rate. We truncate each $T_r$ to its upper-left $m\times m$ matrix, and using the error bound from Theorem \ref{thm: precise bounds}, one computes $\lambda_{\mathrm{red}}$ with guaranteed precision as
\begin{align*}
\lambda_{\mathrm{red}}
\;=\;
0.&8852725442368283013797616193829634609764856249855326080156129828 \\
   &8883895353858358427788269028153535271924654473910802053334304074\cdots.
%Guaranteed digits 0.88527254423682830137976161938296346097648562498553260801561298288883895353858358427788269028153535271924654473910802053334304074
%Presumably correct digits 0.88527254423682830137976161938296346097648562498553260801561298288883895353858358427788269028153535271924654473910802053334304074398275199923450237654791
\end{align*}

\subsection{Example 2}

In this subsection we consider an analytic family of matrices and transition probabilities, and calculate the Taylor expansion of the Lyapunov exponent.

We use the representation of the free product $(\mathbb Z/3\mathbb Z) \ast (\mathbb Z/2\mathbb Z)$ considered in \cite[Subsection 5.5]{Bochi--Potrie--Sambarino}. Let
\[
A_x(t) \;=\; D(t)^{-1} \, R_{\pi/3} \, D(t),
\qquad
A_y(t) \;=\; R_{\pi/2},
\]
where
\[
D(t) = \begin{pmatrix} t & 0 \\ 0 & t^{-1} \end{pmatrix},
\qquad
R_{\theta} = \begin{pmatrix} \cos{\theta} & -\sin{\theta} \\ \sin{\theta} & \cos{\theta} \end{pmatrix}.
\]
We consider the random products over reduced words of $X=\{x,\bar{x},y\}$, where
\[
A_x(t)
=
\begin{pmatrix}
\frac12 & - \frac{\sqrt{3}}{2t^2} \\[6pt]
\frac{\sqrt{3}t^2}{2} & \frac12
\end{pmatrix},
\qquad
A_{\bar{x}}(t)
=
A_x(t)^{-1}
=
\begin{pmatrix}
\frac12 & \frac{\sqrt{3}}{2t^2} \\[6pt]
-\frac{\sqrt{3}t^2}{2} & \frac12
\end{pmatrix},
\qquad
A_y(t)
=
\begin{pmatrix}
0 & -1 \\
1 & 0
\end{pmatrix}.
\]
We consider $t\in (\frac52,\,\frac72)$. Let $t_0 = 3$, $p(t) = \frac12 + t-t_0$, and define the transition matrix as
\[
P(t)=
\begin{pmatrix}
0 & 0 & 1 \\
0 & 0 & 1 \\
p(t) & 1-p(t) & 0
\end{pmatrix},
\]
with rows and columns ordered as $(x, \bar{x}, y)$, and the transition probability from $i$ to $j$ is $P_{ij}$. 

Let $\lambda(t)$ be the Lyapunov exponent. Then, the Taylor expansion is
\[
\lambda(t)
\;=\;
\sum_{n=0}^\infty a_n (t-t_0)^n,
\]
where the rigorously certified values of the first $11$ coefficients are
\begin{align*}
a_0 =
1.02668753502574322381306743078885698085584733242608611503819069056\cdots,
\end{align*}
\begin{align*}
a_1 &=
0.33334463276727235751442614547568014132159285983670947707698164046\cdots,
\end{align*}
\begin{align*}
a_2 &=
-0.0639419036457356072034216310930854348429261266864059723287519982\cdots,
\end{align*}
\begin{align*}
a_3 &=
0.02371366414237382344579888220429926092282058042663105456465261186\cdots,
\end{align*}
\begin{align*}
a_4 &=
-0.0128210355730842625187071826547727740167197573983063450581179771\cdots,
\end{align*}
\begin{align*}
a_5 &= 
0.00758376125675217409196978464888737758815497314787095627447351185\cdots,
\end{align*}
\begin{align*}
a_6 &=
-0.0044084215651143939474934420144888917028299007460083793701014980\cdots,
\end{align*}
\begin{align*}
a_7 &= 
0.00247785191099582781991277973432519071852103088719365535035708164\cdots,
\end{align*}
\begin{align*}
a_8 &=
-0.0013568710710027835740828389326212150195134363688067371840763738\cdots,
\end{align*}
\begin{align*}
a_9 &=
0.00073213986765680152283128293804416409308284665797792103205611317\cdots,
\end{align*}
\begin{align*}
a_{10} &=
-0.0003933322663856287259434519883711520360640978875502183841794051\cdots.
%\\
%a_{11} =
%0.0002122115819985248092596834597173064213513019482190432250351260784\cdots
%\\
%a_{12} =
%-0.000115757086148789859717129934558657403092359636400249873936208297\cdots
\end{align*}

\smallskip

The calculation is as follows. The matrix $P(t)$ has period $2$, so we apply ``aperiodic reduction'' (Proposition \ref{prop: aperiodic reduction}). Namely, we perform cyclic decomposition of $X$ and consider the $2$-step dynamics. The resulting Markov shift has alphabet $\widetilde X = \{xy, \bar{x}y\}$ with the transition probability
\[
\widetilde P(t)
=
\begin{pmatrix}
p(t) & 1-p(t) \\
p(t) & 1-p(t)
\end{pmatrix}.
\]
This system yields the doubled Lyapunov exponent (Proposition \ref{prop: aperiodic reduction}). In the slope coordinate, let
\[
I = \left[ - \frac{\sqrt{3}}{9}, \; \frac{\sqrt{3}}{9} \right],
\qquad
M_{xy} = M_{\bar{x}y} = \RP^1 \setminus I.
\]
Then, $\{M_i\}_{i\in \widetilde X}$ is a multicone for the $2$-step dynamics. The lifted branch-state set is
\[
R
=
\left\{ (u,v)\in \widetilde X^2 \setcond \widetilde P_{uv} > 0 \right\}
=
\{ (xy, xy), (xy, \bar{x}y), (\bar{x}y, xy), (\bar{x}y,\bar{x}y) \}.
\]
The transition matrix $Q=(Q_{r,r'})_{r,r'\in R}$ on $R$ is defined as
\[
Q(t)_{(i,j),\,(k,\ell)}
=
\begin{cases}
\big( \widetilde P(t) \big)_{j,\ell} & \text{if } j=k,\\
0 & \text{otherwise}
\end{cases}
\quad \text{\it{id est}} \quad
Q(t)
=
\begin{pmatrix}
p(t) & 1-p(t) & 0 & 0 \\
0 & 0 & p(t) & 1-p(t) \\
p(t) & 1-p(t) & 0 & 0 \\
0 & 0 & p(t) & 1-p(t)
\end{pmatrix}.
\]
Here, the rows and columns are ordered as $\big( (xy, xy), (xy, \bar{x}y), (\bar{x}y, xy), (\bar{x}y,\bar{x}y) \big)$. Then $R$ is itself a recurrent class, so let $C=R$. We define the local chart as
\[
L_{xy} = L_{\bar{x} y}
=
\begin{pmatrix}
\frac{\sqrt{3}}{9} & \frac{\sqrt{3}}{9} \\[4pt]
1 & -1
\end{pmatrix}.
\]
Then, the local positive matrices are
\begin{gather*}
B_{(xy, \; xy)} = B_{(\bar{x}y, \; xy)}
=
\frac{\sqrt{3}}{36t^2}
\begin{pmatrix}
9t^4 + 26t^2 + 9 & 9t^4 - 28t^2 - 9 \\
9t^4 + 28t^2 - 9 & 9t^4 - 26t^2 + 9
\end{pmatrix},
\\
B_{(xy, \; \bar{x}y)} = B_{(\bar{x}y, \; \bar{x}y)}
=
\frac{\sqrt{3}}{36t^2}
\begin{pmatrix}
9t^4 - 26t^2 + 9 & 9t^4 + 28t^2 - 9 \\
9t^4 - 28t^2 - 9 & 9t^4 + 26t^2 + 9
\end{pmatrix}.
\end{gather*}
The associated M\"obius transformations $f_r[t]: [-1,1]\to[-1,1]$ are
\[
f_{(xy, \; xy)}[t](u) = f_{(\bar{x}y, \; xy)}[t](u)
=
\frac{9u-t^2}{9t^2(3u+t^2)},
\qquad
f_{(xy, \; \bar{x}y)}[t](u) = f_{(\bar{x}y, \; \bar{x}y)}[t](u)
=
\frac{9u+t^2}{9t^2(-3u+t^2)}.
\]

Denote by $\pi(t)$ the stationary probability vector of $Q(t)$. Let $U = \{z\in \mathbb C \;:\; |z-3|\leq 1/5\}$ and $\Omega$ be an open neighborhood of $U$. As in Remark \ref{rmk: condition on domain of holomorphic extension}, we complexify the relevant data and denote them by $Q(z)$, $\pi(z)$, and $f_r[z]$ for each $r\in C$. Then, by direct calculation we see that
\[
\max_{r\in C} \max_{z\in U} \sum_{r'\in C} |Q_{r,\,r'}(z)| = \frac{\sqrt{29}}{5},
\qquad
\max_{r\in C} \max_{z\in U} \max_{x\in \closure{\mathbb D_1}} |f_r[z](x)| = \frac{10525}{213444} = 0.0493\cdots.
\]
Letting $\bar{\rho} = 1/20$, we can confirm that the conditions in Remark \ref{rmk: condition on domain of holomorphic extension} are satisfied with small enough $\Omega \supset U$. Furthermore, the constants in Proposition \ref{prop: precise bounds for higher order derivatives} for this example are
\[
\overline Q  = \frac{\sqrt{29}}{5},
\qquad
\overline D = \frac{75}{196},
\qquad
\overline M_\ell = \frac{300}{769},
\qquad
\overline M_{\Sigma\pi} = \frac{29}{25},
\qquad
\overline m_\pi = \frac{9}{100}.
\]

\medskip

Then we can use the bounds in Proposition \ref{prop: precise bounds for higher order derivatives} to obtain rigorously certified values of the derivatives of the Lyapunov exponent $\lambda(t)$ of the original system, and hence its Taylor coefficients $a_n$. The coefficients are calculated using equation \eqref{eq: inductive identity of finite approximation of general derivatives of lambda} and certified by interval arithmetic.

\section{Preparation} \label{section: Preparation}

\subsection{Markov shifts and projective uniform hyperbolicity}

Let $X$ be a finite set. On the one-sided full shift $X^{\mathbb{N}_0}$, let $\sigma: X^{\mathbb{N}_0} \to X^{\mathbb{N}_0}$ be the shift map
\[
\sigma(x_0,x_1,x_2,\dots)=(x_1,x_2,\dots).
\]
Consider a \emph{row-stochastic} matrix $P=(P_{ij})_{i,j\in X}$, that is,
\[
P_{ij} \geq 0
\quad\text{for all } i,j\in X,
\qquad
\sum_{j\in X} P_{ij}=1
\quad\text{for all } i\in X.
\]
Define the subshift $\Sigma$ by
\[
\Sigma = \left\{ x=(x_n)_{n\geq 0}\in X^{\mathbb{N}_0} \setcond P_{x_n x_{n+1}}>0 \text{ for every } n\geq 0\right\}.
\]

\begin{definition}
The pair $(\Sigma, P)$ is called a \emph{Markov shift}, and $\Sigma$ is referred to as a \emph{topological Markov shift}. In addition, they are said to be \emph{transitive} if $P$ is \emph{irreducible}, meaning that for every $i,j\in X$ there exists $n\geq 1$ such that $(P^n)_{ij}>0$.
\end{definition}

From now on, assume that $P$ is irreducible. Then, $P$ determines a probability law on $\Sigma$: first, by Perron-Frobenius, there exists a unique \emph{stationary probability vector} $p=(p_i)_{i\in X}$ such that
\begin{equation*}
p_i>0 \quad\text{for all } i\in X,
\qquad
\sum_{i\in X} p_i=1,
\qquad
\sum_{i\in X} p_i P_{ij} = p_j
\quad\text{for all } j\in X.
\end{equation*}
For $i_0,\dots,i_n\in X$, the corresponding cylinder set is defined as
\[
[i_0,\dots,i_n]
=
\left\{x = (x_n)_n \in \Sigma \setcond x_k=i_k \text{ for all $0 \leq k \leq n$} \right\}.
\]
Then, there is a unique probability measure $\mu$ (called the \emph{Markov measure}) on $\Sigma$ such that
\begin{equation*}
\mu([i_0,\dots,i_n]) = p_{i_0} P_{i_0 i_1} \cdots P_{i_{n-1} i_n}
\qquad
\text{for every $n\geq 0$ and every $i_0,\dots,i_n\in X$.}
\end{equation*}

Let $\{A_i\}_{i\in X} \subset \GL_2(\mathbb R)$, and define a locally constant function
\begin{equation} \label{eq: definition of locally constant function A}
A: \Sigma \to \GL_2(\mathbb{R}), \qquad A(x) = A_{x_0}
\end{equation}
for any $x = (x_n)_n \in \Sigma$. Also write
\[
A^n(x) = A_{x_{n-1}} \cdots A_{x_0}.
\]
By \cite{Furstenberg--Kesten}, the following limit exists and is constant for $\mu$-almost every $x = (x_n)_{n=0}^\infty \in \Sigma$.
\[
\lambda_+(A,P) = \lim_{n\to\infty}\frac{1}{n}\log \|A^n(x)\|.
\]
This number is called the \emph{(top) Lyapunov exponent}. The limit is independent of the choice of norm.

\medskip

We consider the linear cocycle
\[
\boldsymbol{A}: \Sigma \times \mathbb R^2 \to \Sigma \times \mathbb R^2, \qquad (x, v) \longmapsto \big( \sigma(x), \; A(x)v \big),
\]
and the associated projective cocycle
\[
\mathbb{P}( \boldsymbol{A} ): \Sigma \times \RP^1 \to \Sigma \times \RP^1, \qquad \big( x, [v] \big) \longmapsto \big( \sigma(x), \; [A(x)v] \big).
\]
Here and in what follows, $\RP^1$ denotes the real projective line of $\mathbb R^2$, the projective class of $v \in \mathbb R^2 \setminus \{0\}$ is denoted by $[v]$, and for $G \in \GL_2(\mathbb R)$ we write
$[G] : \RP^1 \to \RP^1$ for the induced projective action.

\medskip

Based on \cite[Definition 3.1]{DDGK25}, we introduce the notion of projective uniform hyperbolicity. Let $\widehat{\Sigma}$ be the natural bi-infinite extension of $\Sigma$:
\[
\widehat{\Sigma}
=
\left\{ \hat x=(x_n)_{n\in \mathbb Z}\in X^{\mathbb Z} \setcond P_{x_n x_{n+1}}>0 \text{ for every } n\in \mathbb Z\right\}.
\]
and let
\[
\widehat{\boldsymbol A} : \widehat{\Sigma} \times \mathbb R^2 \to \widehat{\Sigma} \times \mathbb R^2
\qquad
(\widehat{\mathbf x},w)
\longmapsto
\big( \widehat\sigma(\widehat{\mathbf x}),\,A_{x_0}w \big).
\]
be the induced cocycle, where $\widehat \sigma$ is the left-shift map on $\widehat\Sigma$. Also, write $\widehat A^n(\hat x) = A_{x_{n-1}}\cdots A_{x_0}$.
\begin{definition} \label{def: uniform hyperbolicity}
We say that $\boldsymbol A$ is \emph{projectively uniformly hyperbolic} with respect to $\Sigma$ if
$\widehat{\boldsymbol A}$ is projectively uniformly hyperbolic in the sense of \cite[Definition 3.1]{DDGK25}; that is, there exists a continuous $\widehat{\boldsymbol A}$-invariant decomposition
\[
\mathbb R^2 = E^d(\hat x) \oplus E^w(\hat x),
\qquad \hat x \in \widehat{\Sigma},
\]
into one-dimensional subspaces such that there exists $n \geq 1$ for which
\[
\big\| \widehat A^n(\hat x)\vert_{E^w(\hat x)} \big\|
<
\big\| \widehat A^n(\hat x)\vert_{E^d(\hat x)} \big\|
\qquad
\text{for every } \hat x \in \widehat{\Sigma}.
\]
\end{definition}

\begin{remark}
This property is also known as \emph{dominated splittings}. The dominant/weak directions $E^d$ and $E^w$ are not necessarily unstable/stable directions since $\boldsymbol A$ may not exhibit exponential growth/decay on these directions.
\end{remark}

The authors of \cite{DDGK25} extended the multicone equivalence of \cite{Avila--Bochi--Yoccoz}
from $\SL_2(\mathbb R)$ to $\GL_2(\mathbb R)$.

\begin{theorem}[{\cite[Theorem 2.1]{DDGK25}}] \label{thm: multicone criterion}
The cocycle $\boldsymbol A$ is projectively uniformly hyperbolic with respect to $\Sigma$ if and only if there exists a family
$\{M_i\}_{i\in X}$ such that:
\begin{enumerate}
\item $M_i \subset \RP^1$ is non-empty, proper, and a finite union of open intervals for each $i \in X$.
\item If $P_{ij} > 0$, then $\overline{[A_j](M_i)} \subset M_j$.
\end{enumerate}
\end{theorem}

\begin{remark} \label{remark: remark for the multicone criterion}
\, \hfill
\begin{enumerate}
\item Such a family $\{M_i\}_{i \in X}$ is called the \emph{multicone} for the cocycle $\boldsymbol{A}$ with respect to $\Sigma$. 
\item The finiteness of the connected components of $M_i$ for each $i \in X$ is not stated explicitly in \cite[Theorem 2.1]{DDGK25}. However, in their proof they invoke \cite[Theorem 2.3]{Avila--Bochi--Yoccoz} after lifting the matrices to $\SL_2(\mathbb R)$, and the finiteness follows from \cite[Theorem 2.3]{Avila--Bochi--Yoccoz}.
\end{enumerate}
\end{remark}

\begin{definition}
Let $P$ be an irreducible row-stochastic matrix, and let $i\in X$. The \emph{period} of $P$ is
\begin{equation*}
d = \gcd \left\{ n\geq 1 \setcond (P^n)_{ii}>0\right\},
\end{equation*}
where the right-hand side is independent of the choice of $i\in X$ by the irreducibility of $P$. When $d=1$, the matrix $P$ is said to be \emph{aperiodic}.
\end{definition}

The Malheiro--Viana formalism used later is proved under the assumption that the transition matrix is aperiodic. By the standard cyclic decomposition, we may reduce to the aperiodic case.

\begin{proposition} \label{prop: aperiodic reduction}
Suppose $(\Sigma, P)$ is a transitive Markov shift with period $d$. Then there exist a finite set $\widetilde X \subset X^d$, a topological Markov shift $\widetilde \Sigma \subset \widetilde X^{\mathbb N_0}$, a row-stochastic matrix $\widetilde P$ on $\widetilde X$, and a map $\widetilde A : \widetilde \Sigma \to \GL_2(\mathbb R)$ such that the following hold.
\begin{enumerate}
\item The matrix $\widetilde P$ is irreducible and aperiodic, and the shift $\widetilde \Sigma$ is the topological Markov shift associated to $\widetilde P$.
\item For every $\widetilde x = (\widetilde x_n)_n \in \widetilde \Sigma$ with $\widetilde x_0 = (i_0,\dots,i_{d-1}) \in \widetilde X$, we have
\[
\widetilde A(\widetilde x)
\;=\;
A_{i_{d-1}} \cdots A_{i_0}.
\]
\item Let $\widetilde \mu$ be the Markov measure on $\widetilde \Sigma$ associated to $\widetilde P$. Then for $\widetilde \mu$-a.e.\ $\widetilde x \in \widetilde \Sigma$,
\[
\lim_{n\to\infty} \frac{1}{n} \log \big\| \widetilde A^n(\widetilde x) \big\|
\;=\;
d\,\lambda_+(A,P).
\]
\end{enumerate}
\end{proposition}

The proof is included in Appendix \ref{section: appendix}.

\subsection{Metrics and contraction}

In this subsection we introduce the hyperbolic distance and Wasserstein metric built from it, and then prove the contraction property of some M\"obius maps. Let $\mathbb{D}_t = \left\{ z \in \mathbb{C} \, : \, |z| < t \right\}$ for $t > 0$.

\begin{definition}
Take the principal branch of $\Log$ with $\Log(1) = 0$. For $z, w \in \mathbb D_1$ define
\[
d_{\mathrm{hyp}} ( z, w ) = 2 \mathrm{artanh} \left( \left| \frac{ z - w }{ 1- \overline{w}z } \right| \right).
\qquad \left( \mathrm{artanh}(z) = \frac{1}{2} \Log \frac{1+z}{1-z}. \right)
\]
This $d_{\mathrm{hyp}}$ defines a distance on $\mathbb D_1$, and hence on $(-1,1)$. It is called the \textbf{hyperbolic distance}.
\end{definition}
Denote by $\mathrm{len}(\gamma)$ the hyperbolic length of a piecewise $C^1$-path $\gamma:[0,1]\to \mathbb D_1$. That is,
\begin{equation} \label{eq: definition of hyperbolic length}
\mathrm{len}(\gamma)
=
\int_0^1 \frac{2|\gamma'(t)|}{1 - |\gamma(t)|^2} dt.
\end{equation}
Then, for any $z,w\in \mathbb D_1$, we have
\[
d_{\mathrm{hyp}}(z, w) = \min \left\{ \mathrm{len}(\gamma) \setcond \text{$\gamma: [0,1]\to\mathbb D_1$ is piecewise $C^1$ and } \gamma(0) = z, \; \gamma(1) = w \right\}.
\]
Also, for real $x,y\in (-1,1)$, we have
\[
d_{\mathrm{hyp}}(x,y) = 2 \Big| \mathrm{artanh}(x) - \mathrm{artanh}(y) \Big|.
\]

\begin{definition}
For a Polish space $Y$, denote by $\mathscr{P}(Y)$ the set of Borel probability measures on $Y$. For $\nu_1, \nu_2 \in \mathscr{P}(-1,1)$, define the set of couplings as
\[
\Pi( \nu_1, \, \nu_2) = \left\{ \eta \in \mathscr{P} \Big( (-1, 1) \times (-1, 1) \Big) \setcond (\pi_1)_* \eta = \nu_1, \; (\pi_2)_* \eta = \nu_2 \right\}.
\]
Here, $\pi_j$ is the projection onto the $j$-th coordinate for $j = 1,2$, and $(\pi_j)_* \eta$ is the push-forward measure of $\eta$ by $\pi_j$. Define
\[
\mathscr{P}_1(-1,1) = \left\{ \nu \in \mathscr{P}(-1,1) \setcond \int_{(-1,1)} d_{\mathrm{hyp}}(x,0) d\nu(x) < \infty \right\}.
\]
We define the \textbf{Wasserstein-1 metric} $W_1^{d_{\mathrm{hyp}}}$ on $\mathscr{P}_1(-1,1)$ by
\[
W_1^{d_{\mathrm{hyp}}}(\nu_1, \nu_2) = \inf_{\eta \in \Pi(\nu_1, \nu_2)} \int_{(-1, 1) \times (-1, 1)} d_{\mathrm{hyp}}(x, y) d\eta(x, y). \qquad \big( \nu_1, \nu_2 \in \mathscr{P}_1(-1,1). \big)
\]
\end{definition}

It is immediate that $W_1^{d_{\mathrm{hyp}}}(\nu_1, \nu_2)$ is finite for $\nu_1, \nu_2 \in \mathscr{P}_1(-1,1)$, and the completeness of $\left( \mathscr{P}_1(-1,1), W_1^{d_{\mathrm{hyp}}} \right)$ follows from the completeness of $\left( (-1,1), d_{\mathrm{hyp}} \right)$ \cite[Theorem 6.18]{Villani}. In this paper we write $W_1$ for $W_1^{d_{\mathrm{hyp}}}$.

\bigskip

Since M\"obius transformations map generalized circles to generalized circles, we have the following.
\begin{lemma}[{\cite[Lemma 2.5]{Alibabaei26-2}}] \label{lemma: real contraction implies complex contraction}
Let $f$ be a M\"obius transformation with real coefficients, and suppose that there is $\rho > 0$ with $f([-1,1]) \subset [-\rho, \rho]$. Then, we have $f(\closure{\mathbb D_1} ) \subset \closure{\mathbb D_\rho}$ and $f(\mathbb{D}_1) \subset \mathbb{D}_\rho$.
\end{lemma}

For the reader's convenience, we restate the proof in \cite{Alibabaei26-2}.

\begin{proof}
We denote by $\widehat{\mathbb{C}} = \mathbb C \cup \{\widehat \infty\}$ the Riemann sphere, and write $\widehat{\mathbb{R}} = \mathbb R \cup \{\widehat \infty\}$. Since the unique pole $f^{-1}(\widehat{\infty})$ of $f$ is in $\widehat{\mathbb{R}} \setminus [-1,1]$, and by the correspondence between generalized circles, $f( \partial \mathbb{D}_1)$ is a circle. Also, $f$ is conformal since it is holomorphic on a neighborhood of $\overline{\mathbb{D}_1}$ and has non-vanishing derivative. Then the circle $f( \partial \mathbb{D}_1)$ crosses $\mathbb{R}$ orthogonally at $f(\pm1)$, and the center of the circle is $c = \frac{1}{2} \big( f(-1) + f(1) \big) \in [-\rho,\rho]$, with radius $q = \frac{1}{2}\left| f(-1) - f(1) \right|$. Then, 
\[
\left| c + q e^{i\theta} \right| \;\leq\; |c| + q \;=\; \max \{ |f(-1)|, |f(1)| \} \;\leq\; \rho.
\]
Thus $f(z) \in \overline{\mathbb{D}_\rho}$ for any $z \in \partial \mathbb{D}_1$. By the maximum modulus principle, $|f|$ attains a maximum at the boundary of $\mathbb{D}_1$. We conclude that $f( \overline{\mathbb{D}_1} ) \subset \overline{\mathbb{D}_\rho}$. The claim $f(\mathbb{D}_1) \subset \mathbb{D}_\rho$ follows by the open mapping theorem.
\end{proof}

\begin{lemma}[{\cite[Lemma 2.7]{Alibabaei26-2}}] \label{lemma: small image implies contraction}
Let $f$ be a M\"obius transformation satisfying $f(\mathbb{D}_1) \subset \mathbb{D}_\rho$ for some $0<\rho<1$. Then for any $u,v \in \mathbb D_1$,
\[
d_{\mathrm{hyp}}\big(f(u),f(v)\big)
\;\leq\;
\rho \, d_{\mathrm{hyp}}(u,v).
\]
\end{lemma}

We prove this for completeness.

\begin{proof}
Define
\[
F:\mathbb D_1\to\mathbb C,
\qquad
F(z)=\frac{1}{\rho}f(z).
\]
Then $F$ is holomorphic on $\mathbb D_1$ and satisfies $F(\mathbb D_1)\subset \mathbb D_1$. Then, by the Schwarz--Pick theorem,
\[
d_{\mathrm{hyp}}\big(F(u),F(v)\big)
\;\leq\;
d_{\mathrm{hyp}}(u,v)
\qquad
\text{for every } u,v\in\mathbb D_1.
\]

Next, consider the map $m_\rho : \mathbb D_1 \to \mathbb D_1$ defined by $m_\rho(z)=\rho z$. We claim that
\[
d_{\mathrm{hyp}}\big(m_\rho(\xi),m_\rho(\eta)\big)
\;\leq\;
\rho\, d_{\mathrm{hyp}}(\xi,\eta)
\qquad
\text{for every } \xi,\eta\in\mathbb D_1.
\]
To see that, first note that for any $C^1$-path $\gamma: [0,1]\to \mathbb D_1$, $1-|\gamma(t)|^2\leq 1-\rho^2|\gamma(t)|^2$. Then, by the definition of hyperbolic length in equation \eqref{eq: definition of hyperbolic length},
\begin{align*}
\mathrm{len}(m_\rho\circ\gamma)
&=
\int_0^1 \frac{2\, |(m_\rho\circ\gamma)'(t)|}{1-|m_\rho(\gamma(t))|^2}\,dt
=
\int_0^1 \frac{2\rho\, |\gamma'(t)|}{1-\rho^2|\gamma(t)|^2}\,dt \\
&=
\int_0^1
\rho \left( \frac{1-|\gamma(t)|^2}{1-\rho^2|\gamma(t)|^2} \right)
\frac{2\,|\gamma'(t)|}{1-|\gamma(t)|^2}\,dt
\leq
\rho
\int_0^1 \frac{2\,|\gamma'(t)|}{1-|\gamma(t)|^2}\,dt
=
\rho\, \mathrm{len}(\gamma).
\end{align*}
Taking the infimum over all $\gamma$ joining $\xi$ to $\eta$ proves the claim.

Finally, since $f=m_\rho\circ F$, we obtain, for every $u,v\in\mathbb D_1$,
\[
d_{\mathrm{hyp}}\big(f(u),f(v)\big)
=
d_{\mathrm{hyp}}\big(m_\rho(F(u)),m_\rho(F(v))\big)
\leq
\rho\, d_{\mathrm{hyp}}\big(F(u),F(v)\big)
\leq
\rho\, d_{\mathrm{hyp}}(u,v). \qedhere
\]
\end{proof}

\section{Explicit representation of the Lyapunov exponent} \label{section: Explicit representation of the Lyapunov exponent}

In this section, we prove Theorem \ref{thm: explicit representation of the Lyapunov exponent in the introduction} while presenting the precise construction. Throughout this section, $(\Sigma, P)$ is a transitive Markov shift on alphabet $X$, and the cocycle $\boldsymbol A$ is assumed to be projectively uniformly hyperbolic with respect to $\Sigma$.

\subsection{Branch-state extension} \label{subsection: Branch-state extension}

We first lift $\Sigma$ to another Markov shift that carries enough local information. In the next subsection, the lifted cocycle will be shown to yield the same Lyapunov exponent.

By Theorem \ref{thm: multicone criterion}, the cocycle $\boldsymbol A$ admits a multicone $\{M_i\}_{i\in X}$. Let us write each $M_i$ as the finite union of disjoint connected components:
\[
M_i = \bigsqcup_{a=1}^{m(i)} M_{i,a} \quad \text{for each $i\in X$}.
\]

Fix $i,j\in X$ with $P_{ij}>0$, and fix $a \in \{1,\dots,m(i)\}$. Since $\big[A_j\big](M_{i,a})$ is connected and contained in $M_j$, there is a unique index $b=\beta(i,a,j) = \beta\big((i,a),j\big) \in \{1,\dots,m(j)\}$ such that
\[
\big[A_j\big]\big(\overline{M_{i,a}}\big)\subset M_{j,b}.
\]

Define the set of branch states by
\[
R
\;=\;
\left\{
\big( (i,a), (j,b) \big) \setcond \text{$i,j \in X$, \; $P_{ij}>0$, \; $1\leq a \leq m(i)$, \; $b=\beta(i,a,j)$}
\right\}.
\]
In words, an element of $R$ records the location $a$ within the current multicone $M_i$, and the target location $\beta(i,a,j)$ within the next multicone $M_j$, only if $i\to j$ is possible.

For $r= \big( (i,a), (j,b) \big) \in R$, we write
\[
s(r)\coloneqq(i,a),
\qquad
t(r)\coloneqq(j,b),
\qquad
\tau(r)\coloneqq j.
\]

The branch-state transition matrix $Q=(Q_{r,r'})_{r,r'\in R}$ is defined by
\[
Q_{r,r'}
=
\begin{cases}
P_{\tau(r),\,\tau(r')} & \text{if } t(r)=s(r'),\\
0 & \text{otherwise.}
\end{cases}
\]
The matrix $Q$ is row-stochastic. Indeed, for any $r$,
\[
\sum_{r' \in R} Q_{r, r'} = \sum_{k \in X: \; P_{\tau(r), k} > 0} P_{\tau(r), k}
= \sum_{k \in X} P_{\tau(r), k} = 1.
\]
Here, we used the fact that for every $k\in X$ with $P_{\tau(r),k}>0$, there is a unique $r'\in R$ such that $Q_{r,r'}>0$ and $\tau(r')=k$.

For $C \subset R$, define a matrix $Q|_C$ by $(Q|_C)_{r,r'} = Q_{r,r'}$ for each $r, r' \in C$.

\begin{definition} \label{def: recurrent class}
A set $C \subset R$ is said to be a \emph{recurrent class} if $Q|_C$ is irreducible and
\begin{equation*}
r\in C,\ Q_{r,r'}>0 \implies r'\in C.
\end{equation*}
\end{definition}

There is at least one recurrent class $C\subset R$. We fix such $C$, and denote by $\Sigma_C\subset C^{\mathbb N_0}$ the topological Markov shift associated to $Q|_C$,
\[
\Sigma_C
=
\left\{
(r_n)_{n\geq 0}\in C^{\mathbb N_0} \;\middle|\; Q_{r_n r_{n+1}}>0 \text{ for every } n\geq 0
\right\}.
\]
Let $(\pi_r)_{r\in C}$ be the unique stationary probability vector of $Q|_C$, and let $\mu_C$ be the Markov measure on $\Sigma_C$ associated to $Q|_C$ and $(\pi_r)_{r\in C}$. By the irreducibility of $Q|_C$, we have $\pi_r > 0$ for every $r \in C$.

\begin{proposition} \label{prop: branch-projection}
Define
\[
\vartheta: \Sigma_C\to \Sigma, \qquad
\vartheta \big( (r_n)_{n\geq 0} \big) = \big( \tau(r_n) \big)_{n\geq 0}
\]
Then $\vartheta$ is surjective, and $\vartheta_*\mu_C=\mu$.
\end{proposition}

\begin{proof}
First, let $r_0 \in C$, $j\in X$ and suppose $\tau(r_0) \to j$ is admissible, that is, $P_{\tau(r_0), j} > 0$. Let $b = \beta(t(r_0), j)$. Then $r_1 = \big( t(r_0), (j, b) \big) \in R$ satisfies $\tau(r_1) = j$ and $s(r_1) = t(r_0)$. Thus, $Q_{r_0, r_1} = P_{\tau(r_0), j} > 0$. By the definition of a recurrent class, we have $r_1 \in C$, and the path $r_0 \to r_1$ projects to $\tau(r_0) \to j$. Next, take any $k \in X$. By the irreducibility of $P$, there is $n > 0$ with $(P^n)_{\tau(r_0), k} > 0$, so there is an admissible path connecting $\tau(r_0)$ and $k$. Then applying the same argument as above to the path, some $r\in C$ satisfies $\tau(r) = k$. Therefore,
\begin{equation} \label{eq: projection of the first coordinate is surjective}
\left\{ \tau(r) \setcond r \in C \right\} = X.
\end{equation}

Take any $(x_n)_{n \geq 0} \in \Sigma$. By equation \eqref{eq: projection of the first coordinate is surjective}, there is at least one $r_0 \in C$ with $\tau(r_0) = x_0$. By induction, at each $n \geq 0$ we define $b_{n+1} = \beta( t(r_n), x_{n+1})$ and $r_{n+1} = \big( t(r_n), (x_{n+1}, b_{n+1}) \big)$. We obtain $(r_n)_{n \geq 0} \in \Sigma_C$ with $\vartheta \big( (r_n)_{n\geq 0} \big) = (x_n)_{n \geq 0}$, and thus $\vartheta$ is a surjection. Moreover, once $r_0$ is chosen, the lift $(r_n)_{n\geq 0}$ is uniquely determined.

Next, let $[i_0,\dots,i_n]$ be a cylinder in $\Sigma$. For each $\gamma \in C$ with $\tau(\gamma) = i_0$, there is a unique $Q|_C$-admissible path,
\[
\big( \gamma, r_1(\gamma), \dots, r_n(\gamma) \big) \in C^{n+1},
\]
such that $\tau\big(r_k(\gamma)\big)=i_k$ for every $1\leq k\leq n$. Hence,
\[
\vartheta^{-1}([i_0,\dots,i_n])
=
\bigsqcup_{\substack{\gamma \in C \\ \tau(\gamma) = i_0}}
\Big\{ (s_n)_{n \geq 0} \in \Sigma_C \;\Big|\; \text{$s_0 = \gamma$, $s_k = r_k(\gamma)$ for every $1 \leq k \leq n$} \Big\}.
\]
The right-hand side is a disjoint union of cylinder sets in $\Sigma_C$, and by the definition of $\mu_C$,
\begin{equation} \label{eq: cylinder formula for the lifted measure}
\mu_C\big(\vartheta^{-1}([i_0,\dots,i_n])\big)
= \hspace{-3pt}
\sum_{\substack{\gamma \in C \\ \tau(\gamma) = i_0}}
\pi_{\gamma} Q_{\gamma, r_1(\gamma)} Q_{r_1(\gamma), r_2(\gamma)} \cdots Q_{r_{n-1}(\gamma), r_n(\gamma)}
= \hspace{-3pt}
\sum_{\substack{\gamma \in C \\ \tau(\gamma) = i_0}}
\pi_{\gamma} P_{i_0 i_1}\cdots P_{i_{n-1}i_n}.
\end{equation}

We prove that
\begin{equation} \label{eq: recovery of stationary prob vector from the lift}
\sum_{\substack{\gamma \in C \\ \tau(\gamma) = i_0}} \pi_{\gamma} = p_{i_0}.
\end{equation}
Write the left-hand side as $\hat{p}_{i_0}$. Then, $\hat{p}_{i_0} > 0$ for any $i_0 \in X$, and
\[
\sum_{i_0 \in X} \hat{p}_{i_0}
= \sum_{i_0 \in X} \sum_{\substack{\gamma \in C \\ \tau(\gamma) = i_0}} \pi_{\gamma}
= \sum_{\gamma \in C} \pi_\gamma = 1.
\]
Moreover, for each $i_0 \in X$, by the definition of the stationary probability vector $(\pi_r)_{r \in C}$,
\[
\hat{p}_{i_0}
= \sum_{\substack{\gamma \in C \\ \tau(\gamma) = i_0}} \pi_{\gamma}
= \sum_{\substack{\gamma \in C \\ \tau(\gamma) = i_0}} \sum_{r \in C} \pi_r Q_{r, \gamma}
= \sum_{r \in C} \pi_r \sum_{\substack{\gamma \in C \\ \tau(\gamma) = i_0}} Q_{r, \gamma}
= \sum_{r \in C} \pi_r P_{\tau(r), i_0}.
\]
Here, for each $r \in C$ with $P_{\tau(r), i_0} > 0$, we used the existence and uniqueness of $\gamma \in C$ satisfying $\tau(\gamma) = i_0$ and $Q_{r, \gamma} > 0$. We have
\[
\hat{p}_{i_0}
= \sum_{r \in C} \pi_r P_{\tau(r), i_0}
= \sum_{j \in X} \left( \sum_{\substack{r \in C \\ \tau(r) = j}} \pi_r \right) P_{j, i_0}
= \sum_{j \in X} \hat{p}_j P_{j, i_0}.
\]
Therefore, $(\hat{p}_j)_{j \in X}$ is a stationary probability vector of $P$, and by uniqueness, it is in fact $(p_j)_{j \in X}$, proving equation \eqref{eq: recovery of stationary prob vector from the lift}.

Thus, combining with equation \eqref{eq: cylinder formula for the lifted measure},
\[
\mu_C\big(\vartheta^{-1}([i_0,\dots,i_n])\big)
=
p_{i_0} P_{i_0 i_1}\cdots P_{i_{n-1}i_n}
=
\mu\big([i_0,\dots,i_n]\big).
\]
This proves $\vartheta_*\mu_C=\mu$.
\end{proof}

\subsection{Local charts and Markov operator} \label{subsection: Local charts and Markov operator}

Recall that $\{M_{i,a}\}_{i \in X, \; 1 \leq a\leq m(i)}$ consists of the connected components of a multicone for $\boldsymbol A$. Let $J_{i,a} = \overline{M_{i,a}}$ for each $i$ and $a$. By the definition of multicones, and since each projective action $\left[ A_j \right]$ is a homeomorphism, we have
\begin{equation} \label{eq: contraction of Aj}
\big[A_j\big](J_{i,a})\subset J_{j,\beta(i,a,j)}^\circ
\end{equation}
whenever $\big( (i,a), (j,\beta(i,a,j)) \big) \in R$. Here, $J^\circ$ denotes the interior of $J$.

Let $\pazocal V \subset \RP^1$ be the projective non-negative cone, that is,
\[
\pazocal V = \left\{ [v] \in \RP^1 \setcond \text{$ v = (v_1, v_2)^\top$, $v_k \geq 0$ for $k = 1,2$ and $v\ne0$} \right\}.
\]
For each $(i,a)$, take any matrix $L_{i,a} \in \GL_2(\mathbb{R})$ so that its projective action $\phi_{i,a} = \big[L_{i,a}\big] : \RP^1 \to \RP^1$ sends $\pazocal V$ to $J_{i,a}$.
\[
\phi_{i,a}(\pazocal V) = J_{i,a}.
\]
For each $r \in C$, define $B_r \in \GL_2(\mathbb R)$ by
\[
B_r \;=\; \vep(r) \, {L_{t(r)}}^{-1} A_{\tau(r)} L_{s(r)},
\]
where $\vep(r) \in \{\pm1\}$ is chosen so that $B_r$ is a positive matrix: to see that this is possible, define $g_r: \RP^1 \to \RP^1$ by
\[
g_r \;=\; {\phi_{t(r)}}^{-1} \, \circ \, \big[A_{\tau(r)}\big] \, \circ \, \phi_{s(r)}.
\]
By equation \eqref{eq: contraction of Aj}, $g_r$ maps $\pazocal V$ into the positive cone,
\begin{equation} \label{eq: g maps non-negative cone to positive cone}
g_r( \pazocal V ) \subset \pazocal V^\circ.
\end{equation}
Therefore, the matrix ${L_{t(r)}}^{-1} A_{\tau(r)} L_{s(r)}$ has a uniform sign across its entries.

We define
\[
B: \Sigma_C \to \GL_2(\mathbb{R}), \qquad B\big( (r_n)_{n \geq 0} \big) = B_{r_0}.
\]
As before, for $\mathbf r = (r_n)_{n\geq 0}\in \Sigma_C$ and $n\geq 1$, we write
\[
B^n(\mathbf r) = B_{r_{n-1}} \cdots B_{r_0}.
\]

\begin{lemma} \label{lemma: Lyapunov exponent identity for the lift}
We have
\begin{equation} \label{eq: Lyapunov exponent identity for the lift}
\lambda_+\big(B, Q|_C\big)
\;=\;
\lambda_+(A,P).
\end{equation}
\end{lemma}

\begin{proof}
Let
\[
E = \left\{ x\in \Sigma \setcond \lim_{n\to\infty}\frac{1}{n}\log \|A^n(x)\| = \lambda_+(A,P) \right\}.
\]
Then $\mu(E)=1$. Since $\vartheta_*\mu_C=\mu$ by Proposition \ref{prop: branch-projection}, we have $\mu_C\big(\vartheta^{-1}(E)\big)=\mu(E)=1$.
Therefore it is enough to consider points in $\vartheta^{-1}(E)$. Fix $\mathbf r=(r_n)_{n\geq 0} \in \vartheta^{-1}(E)$, and let $x=\vartheta(\mathbf r) \in E$.

By definition of $B$, for any $n \geq 1$,
\[
\left\| B^n(\mathbf r) \right\|
=
\left\| B_{r_{n-1}}\cdots B_{r_0} \right\|
=
\left\| \big({L_{t(r_{n-1})}}^{\!\!-1}A_{\tau(r_{n-1})}L_{s(r_{n-1})}\big)
\cdots
\big({L_{t(r_0)}}^{\!\!-1}A_{\tau(r_0)}L_{s(r_0)}\big) \right\|.
\]
Because $s(r_k) = t(r_{k-1})$ for every $1\leq k\leq n-1$, the middle factors cancel, and we obtain
\[
\left\| B^n(\mathbf r) \right\|
=
\left\| {L_{t(r_{n-1})}}^{\!\!-1}
A_{\tau(r_{n-1})}\cdots A_{\tau(r_0)}
L_{s(r_0)} \right\|
=
\left\| {L_{t(r_{n-1})}}^{\!\!-1}A^n(x)L_{s(r_0)} \right\|.
\]
Since there are only finitely many $L_{i,a}$, the above identity implies equation \eqref{eq: Lyapunov exponent identity for the lift}.
\end{proof}

\bigskip

Now, we perform ``aperiodic reduction''. Denote by $d$ the period of $Q|_C$. If $d>1$, we apply Proposition \ref{prop: aperiodic reduction} to $(\Sigma_C, Q|_C)$ and reduce to the aperiodic case. For simplicity of notation, from this point on the following objects are replaced accordingly and denoted by the same symbols: the alphabet $C$, the matrix $Q|_C$, the subshift $\Sigma_C$, the positive matrices $B_r$ along with maps $g_r$, and the probability vector $(\pi_r)_{r\in C}$. Only in the final recovery of the Lyapunov exponent of the original system does one divide by $d$ (and we do so in Theorem \ref{thm: infinite sum equals lambda}). Since $Q|_C$ is irreducible and aperiodic, $Q|_C$ is primitive; i.e.\ there is $n \in \mathbb N$ such that $\big( (Q|_C)^n\big)_{r,r'} > 0$ for any $r,r' \in C$.

We identify the non-negative cone $\pazocal V$ with the following $1$-simplex
\[
\Delta = \left\{ v = \begin{pmatrix} v_1 \\ v_2 \end{pmatrix} \in \mathbb{R}^2 \setcond v_1, v_2 \geq 0, v_1 + v_2 = 1 \right\}.
\]
We define the chart
\[
\psi:[-1, 1] \to \Delta, \qquad \psi(x) = \frac{1}{2} \begin{pmatrix} 1+x \\ 1-x \end{pmatrix}.
\]
For each $r \in C$, the composition
\[
f_r \,=\, \psi^{-1} \circ g_r \circ \psi: \; [-1,1] \to [-1,1]
\]
is a M\"obius transformation. Indeed, we have
\[
f_r(x) =
\frac{(a_r-b_r-c_r+d_r)x + (a_r+b_r-c_r-d_r)}{(a_r-b_r+c_r-d_r)x + (a_r+b_r+c_r+d_r)}
\quad \text{where} \quad
B_r
=
\begin{pmatrix} a_r & b_r \\ c_r & d_r \end{pmatrix}.
\]
For later purposes, let us define the map $F: \mathrm{GL}_2(\mathbb{R}) \to \mathrm{GL}_2(\mathbb{R})$ by
\begin{equation} \label{eq: definition of F}
F
\begin{pmatrix}
a & b \\
c & d
\end{pmatrix}
\;=\;
\begin{pmatrix}
1 & -1 \\
1 & 1
\end{pmatrix}
\begin{pmatrix}
a & b \\
c & d
\end{pmatrix}
{
\begin{pmatrix}
1 & -1 \\
1 & 1
\end{pmatrix}
}^{-1}
=\;
\frac{1}{2}
\begin{pmatrix}
a-b-c+d & a+b-c-d \\
a-b+c-d & a+b+c+d
\end{pmatrix}.
\end{equation}

By equation \eqref{eq: g maps non-negative cone to positive cone}, there is $0 < \rho_r < 1$ such that $f_r([-1,1]) \subset [-\rho_r, \rho_r]$. Letting $\rho = \max_{r \in C} \rho_r$, for every $r \in C$,
\begin{equation} \label{eq: every f is a rho-contraction}
f_r([-1,1]) \subset [-\rho, \rho].
\end{equation}

\bigskip

Let $\mathscr P_1^C = \prod_{r\in C}\mathscr P_1(-1,1)$, and define a metric $\pazocal W$ on $\mathscr P_1^C$ by
\[
\pazocal W(\eta, \zeta)
=
\sum_{r\in C} \pi_r W_1(\eta_r,\zeta_r) \quad \text{for $\eta, \zeta \in \mathscr P_1^C$}.
\]
This is indeed a metric since $\pi_r > 0$ for every $r \in C$. The space $\big(\mathscr P_1^C, \pazocal W\big)$ is complete due to the completeness of each $\big( \mathscr P_1(-1,1), W_1 \big)$.

\begin{proposition} \label{prop: contraction of H}
The map $\mathscr H: \mathscr P_1^C \to \mathscr P_1^C$, defined by 
\[
(\mathscr H \eta)_{r'}
\;=\;
\sum_{r\in C} \frac{\pi_r Q_{r,r'}}{\pi_{r'}}\, (f_r)_*\eta_r,
\]
is a $\rho$-contraction with respect to $\pazocal W$. That is, for every $\eta, \zeta \in \mathscr P_1^C$,
\[
\pazocal W\big( \mathscr H \eta, \mathscr H \zeta \big)
\;\leq\;
\rho \, \pazocal W(\eta, \zeta).
\]
\end{proposition}

\begin{proof}
Take any $\eta \in \mathscr P_1^C$. Since $(\pi_r)_{r \in C}$ is the stationary vector for $Q|_C$, for any $r' \in C$, $\sum_{r\in C} \pi_r Q_{r,r'} = \pi_{r'}$. Thus $(\mathscr H \eta)_{r'}$ is a probability measure. Equation \eqref{eq: every f is a rho-contraction}, Lemma \ref{lemma: real contraction implies complex contraction}, and Lemma \ref{lemma: small image implies contraction} imply that each $f_r$ is a $\rho$-contraction for $d_{\mathrm{hyp}}$. Thus, for each $r \in C$ and $x \in (-1,1)$,
\[
d_{\mathrm{hyp}}\big( f_r(x),0 \big) \leq d_{\mathrm{hyp}}\big( f_r(x), f_r(0) \big) + d_{\mathrm{hyp}}\big( f_r(0),0 \big)
\leq \rho \, d_{\mathrm{hyp}}(x,0)+d_{\mathrm{hyp}}\big( f_r(0),0 \big).
\]
Therefore, for each $r' \in C$,
\[
\int d_{\mathrm{hyp}}(x,0) \, d(\mathscr{H}\eta)_{r'}
=
\sum_{r\in C} \frac{\pi_r Q_{r,r'}}{\pi_{r'}}\, \int d_{\mathrm{hyp}}\big( f_r(x),0 \big) \, d\eta_r < \infty
\]
This implies $\mathscr H(\mathscr P_1^C) \subset \mathscr P_1^C$.

Fix $\eta, \zeta \in \mathscr P_1^C$ and $\vep>0$. For each $r\in C$, take a coupling $\kappa_r \in \Pi(\eta_r, \zeta_r)$ so that
\[
\int d_{\mathrm{hyp}}(x,y)\,d\kappa_r(x,y)
\leq
W_1(\eta_r,\zeta_r)+\vep.
\]
Then using $\rho$-contraction of $f_r$, for every $r' \in C$,
\begin{align*}
\sum_{r \in C} \frac{\pi_rQ_{r,r'}}{\pi_{r'}} \hspace{-4pt} \int d_{\mathrm{hyp}}\big( f_r(x), f_r(y) \big)\,d\kappa_r(x,y)
&\leq
\sum_{r \in C} \frac{\pi_rQ_{r,r'}}{\pi_{r'}} \rho \int d_{\mathrm{hyp}}(x,y)\,d\kappa_r(x,y) \\
&\leq
\rho \sum_{r \in C} \frac{\pi_rQ_{r,r'}}{\pi_{r'}} W_1(\eta_r, \zeta_r) + \rho \vep.
\end{align*}
Since $\sum_{r \in C} \frac{\pi_rQ_{r,r'}}{\pi_{r'}} (f_r \times f_r)_* \kappa_r \in \Pi\big( (\mathscr H\eta)_{r'}, (\mathscr H\zeta)_{r'} \big)$, and $\vep > 0$ was arbitrary, this implies
\[
\pazocal W\big( \mathscr H \eta, \mathscr H \zeta \big)
=
\sum_{r' \in C} \pi_{r'} W_1\big( (\mathscr H \eta)_{r'},(\mathscr H \zeta)_{r'} \big)
\leq
\rho \pazocal W(\eta, \zeta). \qedhere
\]
\end{proof}

By applying Banach's fixed point theorem to $\mathscr H$ on a complete space $\big( \mathscr P_1^C, \pazocal W \big)$, the map $\mathscr H$ has a unique fixed point
\begin{equation} \label{eq: unique fixed point of the Markov operator}
\eta^C=(\eta_r^C)_{r\in C}.
\end{equation}

\subsection{Malheiro--Viana formalism for the lifted cocycle} \label{subsection: Malheiro--Viana formalism for the lifted cocycle}

In this subsection, we apply the Malheiro--Viana formalism to the lifted cocycle generated by $(B_r)_{r\in C}$. We then show that, after passing to the chart $\psi$, a maximizing $Q|_C$-stationary unit measure vector for the variational formula of $\lambda_+\big(B,Q|_C\big)$ corresponds exactly to the fixed point $\eta^C$ in \eqref{eq: unique fixed point of the Markov operator}.

We use the same symbol $\sigma$ to denote the shift map on $\Sigma_C$:
\[
\sigma:\Sigma_C \to \Sigma_C,
\qquad
\sigma\big((r_n)_{n\geq 0}\big)=(r_{n+1})_{n\geq 0}.
\]
Recall that $(\pi_r)_{r\in C}$ is the unique stationary probability vector of $Q|_C$, and $\mu_C$ is the Markov measure on $\Sigma_C$ associated to $Q|_C$.

We now apply the formalism of Malheiro--Viana \cite{Malheiro--Viana} to the lifted cocycle
generated by the matrices $(B_r)_{r\in C}$. We first recall the notions from the paper. Consider the linear cocycle
\[
\boldsymbol{B}: \Sigma_C \times \mathbb R^2 \to \Sigma_C \times \mathbb R^2,
\qquad
(\mathbf r,v)\longmapsto \big(\sigma(\mathbf r),\,B(\mathbf r)v\big).
\]
Its associated projective cocycle is
\[
\mathbb P(\boldsymbol{B}): \Sigma_C \times \mathbb{RP}^1 \to \Sigma_C \times \mathbb{RP}^1,
\qquad
(\mathbf r,[v]) \longmapsto \big( \sigma(\mathbf r), \, [B_{r_0}v] \big).
\]

For each $r\in C$, define the cylinder
\[
[r]
\coloneqq
\left\{
\mathbf s=(s_n)_{n\geq 0}\in \Sigma_C
\;\middle|\;
s_0=r
\right\}.
\]
A \emph{unit measure vector} (\emph{unit vector} in \cite[Section 3]{Malheiro--Viana}) on $C$ is a family
\[
\zeta=(\zeta_r)_{r\in C},
\]
where each $\zeta_r$ is a Borel probability measure on $\mathbb{RP}^1$. To such a unit measure vector they associate the skew-product measure $\mu_C\ltimes \zeta$ on $\Sigma_C\times \mathbb{RP}^1$, defined by
\[
\mu_C\ltimes \zeta
\;=\;
\sum_{r\in C}\big(\mu_C|_{[r]}\big)\times \zeta_r.
\]

Next, define the operator $\widehat{\mathscr H}$ on unit measure vectors by
\[
(\widehat{\mathscr H} \zeta)_{r'}
\;=\;
\sum_{r\in C}\frac{\pi_r Q_{r,r'}}{\pi_{r'}}\,
(g_r)_* \zeta_r
\qquad\text{for each } r'\in C.
\]
We say that a unit measure vector
$\zeta=(\zeta_r)_{r\in C}$ is \emph{$Q|_C$-stationary} if it is a fixed point of $\widehat{\mathscr H}$.

\begin{proposition}[{\cite[Proposition~3.1]{Malheiro--Viana}}] \label{prop MV characterization of stationary measure vectors}
A unit measure vector $\zeta$ is $Q|_C$-stationary if and only if the skew-product measure $\mu_C \ltimes \zeta$ is invariant under the projective cocycle
$\mathbb P(\boldsymbol{B})$.
\end{proposition}

The key variational statement is the following. It is the Markov-shift analogue of the
Furstenberg--Kifer integral formula \cite{Furstenberg--Kifer}, and applied here to the lifted cocycle.

\begin{proposition}[{\cite[Proposition~3.5]{Malheiro--Viana}}] \label{prop: MV variational principle}
We have
\[
\lambda_+\big(B,Q|_C\big)
=
\max\left\{
\sum_{r\in C}\pi_r
\int_{\mathbb{RP}^1}
\log \frac{\|B_r v\|_1}{\|v\|_1}
\, d\zeta_r([v])
\;\middle|\;
\zeta \text{ is a $Q|_C$-stationary unit measure vector}
\right\}.
\]
Here, $\|\cdot\|_1$ denotes the $\ell^1$ norm, the sum of absolute values of the entries.
\end{proposition}

\bigskip

The next proposition shows that, in the chart $\psi$, a maximizing $Q|_C$-stationary unit measure vector in Proposition \ref{prop: MV variational principle} corresponds to $\eta^C$ from \eqref{eq: unique fixed point of the Markov operator}.

\begin{proposition} \label{prop: integral formula of lambda for the lift}
We have
\begin{equation} \label{eq: integral formula of lambda for the lift}
\lambda_+\big(B,Q|_C\big)
=
\sum_{r\in C} \pi_r \int_{(-1,1)} \log {\|B_r \psi(x)\|}_1 \,d\eta_r^C(x).
\end{equation}
\end{proposition}

\begin{proof}
Let
\begin{gather*}
\widehat\Sigma_C
=
\left\{
(r_n)_{n\in\mathbb Z}\in C^{\mathbb Z} \;\middle|\; Q_{r_n r_{n+1}}>0 \text{ for every } n\in\mathbb Z
\right\}, \\
\widehat\sigma:\widehat\Sigma_C\to\widehat\Sigma_C,
\qquad
\widehat\sigma\big((r_n)_{n\in\mathbb Z}\big)=(r_{n+1})_{n\in\mathbb Z}, \\
\Pi_C:\widehat\Sigma_C\to\Sigma_C,
\qquad
\Pi_C\big((r_n)_{n\in\mathbb Z}\big)=(r_n)_{n\geq 0},
\end{gather*}
and let $\widehat\mu_C$ be the two-sided Markov measure on $\widehat\Sigma_C$ associated to $Q|_C$ and
$(\pi_r)_{r\in C}$. Consider the natural extension of the cocycle $\boldsymbol{B}$:
\[
\widehat{\boldsymbol{B}}:\widehat\Sigma_C\times \mathbb R^2\to\widehat\Sigma_C\times \mathbb R^2,
\qquad
(\widehat{\mathbf r},w)
\longmapsto
\big( \widehat\sigma(\widehat{\mathbf r}),\,B_{r_0}w \big).
\]

\medskip

The positive cone $\pazocal V^\circ$ is strictly invariant under $g_r$ for every $r \in C$ by equation \eqref{eq: g maps non-negative cone to positive cone}, that is, $\overline{g_r(\pazocal V^\circ)} \subset \pazocal V^\circ$. Then by Theorem \ref{thm: multicone criterion}, $\boldsymbol{B}$ is projectively uniformly hyperbolic on $\Sigma_C$. Therefore, there exists a continuous splitting into invariant directions
\[
\mathbb R^2 = E^d(\widehat{\mathbf r}) \oplus E^w(\widehat{\mathbf r}),
\qquad \text{for each \; } \widehat{\mathbf r}\in \widehat\Sigma_C,
\]
such that
\begin{equation} \label{eq: invariance of dominant bundle}
[B_{r_0}]\,E^d(\widehat{\mathbf r})
=
E^d\big(\widehat\sigma(\widehat{\mathbf r})\big),
\qquad
[B_{r_0}]\,E^w(\widehat{\mathbf r})
=
E^w\big(\widehat\sigma(\widehat{\mathbf r})\big).
\end{equation}

Now, just like in the proof of \cite[Proposition 3.5]{Malheiro--Viana}, define a probability measure $\widehat m$ on $\widehat\Sigma_C\times \RP^1$ by
\[
\widehat{m}(D \times G) = \widehat{\mu}_C \left( \left\{ \widehat{\mathbf r} \in D \setcond E^d(\widehat{\mathbf r}) \in G \right\} \right),
\]
for any measurable sets $D \subset \widehat{\Sigma}_C$ and $G \subset \RP^1$. Then $\widehat m$ is invariant under the projective cocycle
\[
\mathbb P(\,\widehat{\boldsymbol{B}}\,): \;\widehat\Sigma_C\times \RP^1\to\widehat\Sigma_C\times \RP^1,
\qquad
\big( \widehat{\mathbf r},[w] \big)
\longmapsto
\big( \widehat\sigma(\widehat{\mathbf r}), [B_{r_0}w] \big).
\]
Consider the one-sided projection $m = (\Pi_C\times\mathrm{id})_*\widehat m$. By the proof of \cite[Proposition~3.5]{Malheiro--Viana}, we have the following.

\begin{claim}
The probability measure $m$ is a skew product
\[
m = \mu_C \ltimes \zeta,
\]
where a unit measure vector $\zeta = (\zeta_r)_{r\in C}$ satisfies
\begin{equation} \label{eq: integral formula intermediate}
\lambda_+\big(B,Q|_C\big)
=
\sum_{r\in C}\pi_r\int_{\RP^1} \log \frac{\|B_r v\|_1}{\|v\|_1} \,d\zeta_r.
\end{equation}
\end{claim}

For completeness, we rephrase part of their proof in our notation.

\begin{proof}[Proof of the Claim]

Let
\[
\widehat{[r]}
\coloneqq
\left\{
\widehat{\mathbf s}=(s_n)_{n\in\mathbb Z}\in \widehat\Sigma_C \setcond s_0=r
\right\}
\qquad\text{for each } r\in C.
\]
Also let $\mathrm{pr}_1:\Sigma_C\times \mathbb{RP}^1 \to \Sigma_C$, and $\widehat{\mathrm{pr}}_1:\widehat\Sigma_C\times \mathbb{RP}^1 \to \widehat\Sigma_C$ be the coordinate projections. Then $m$ is invariant under $\mathbb P(\boldsymbol B)$, and it projects to $\mu_C$. This is by the $\mathbb P(\widehat{\boldsymbol B})$-invariance of $\widehat m$ and the identities $\mathbb P(\boldsymbol B)\circ (\Pi_C\times \mathrm{id}) = (\Pi_C\times \mathrm{id})\circ \mathbb P(\widehat{\boldsymbol B})$ and $\mathrm{pr}_1\circ (\Pi_C\times \mathrm{id}) = \Pi_C\circ \widehat{\mathrm{pr}}_1$.

For each $r\in C$, define the spaces of admissible pasts and futures by
\begin{gather*}
\Sigma_C^-(r)
=
\left\{
(u_n)_{n\leq -1}\in C^{\mathbb Z_{<0}} \setcond Q_{u_nu_{n+1}}>0 \text{ for all } n\leq -2,\;
Q_{u_{-1}r}>0
\right\}, \\
\Sigma_C^+(r)
=
\left\{
(v_n)_{n\geq 1}\in C^{\mathbb Z_{>0}} \setcond Q_{rv_1}>0,\; Q_{v_nv_{n+1}}>0 \text{ for all } n\geq 1
\right\}.
\end{gather*}
Because $\widehat\mu_C$ is the two-sided Markov measure associated to $Q|_C$, for each $r\in C$ there exist probability measures $\widehat\mu_r^-$ and $\widehat\mu_r^+$ on $\Sigma_C^-(r)$ and $\Sigma_C^+(r)$, respectively, such that
\[
\widehat\mu_C|_{\widehat{[r]}}
=
\pi_r \cdot (\widehat\mu_r^- \times \widehat\mu_r^+).
\]
Moreover, since the line $E^d(\widehat{\mathbf r})$ is determined by the negative coordinates of $\widehat{\mathbf r}$ alone, there exists a measurable map $e_r:\Sigma_C^-(r)\to \mathbb{RP}^1$ such that for $(\widehat\mu_r^- \times \widehat\mu_r^+)$-a.e.\ $\big( (u_n)_{n\leq -1}, (v_n)_{n\geq 1} \big)$,
\[
E^d\big((u_n)_{n\leq -1},\, r,\,(v_n)_{n\geq 1}\big)
=
e_r\big((u_n)_{n\leq -1}\big).
\]

Now let $U^+ \subset \Sigma_C^+(r)$ and $D \subset \mathbb{RP}^1$ be measurable. By the definition of $\widehat m$,
\begin{align*}
\widehat m \Big( ( \Sigma_C^-(r) \times \{r\} \times U^+) \times D \Big)
&=
\widehat\mu_C\Big(
\left\{ (u,r,v) \in \Sigma_C^-(r) \times \{r\} \times U^+ \setcond E^d(u,r,v) \in D \right\}
\Big) \\
&=
\pi_r \, \widehat\mu_r^+(U^+)
\int_{\Sigma_C^-(r)} \mathbf 1_D\big(e_r(u)\big)\, d\widehat\mu_r^-(u).
\end{align*}
On the other hand, $\widehat\mu_C\big(\Sigma_C^-(r) \times \{r\} \times U^+\big) = \pi_r\,\widehat\mu_r^+(U^+)$.
Then the quotient
\[
\frac{
\widehat m \Big( ( \Sigma_C^-(r) \times \{r\} \times U^+) \times D \Big)
}{
\widehat\mu_C\big(\Sigma_C^-(r)\times \{r\}\times U^+\big)
}
\;=\;
\int_{\Sigma_C^-(r)} \mathbf 1_D\big(e_r(u)\big)\, d\widehat\mu_r^-(u)
\]
is independent of the choice of $U^+$, provided $\widehat\mu_r^+(U^+)>0$.

We can therefore define a probability measure $\zeta_r$ on $\mathbb{RP}^1$, for each $r\in C$, by
\begin{equation} \label{eq: the definition of zeta r}
\zeta_r(D)
=
\int_{\Sigma_C^-(r)} \mathbf 1_D\big(e_r(u)\big)\, d\widehat\mu_r^-(u).
\end{equation}
It follows that $m = \mu_C \ltimes \zeta$, since for every measurable $E\subset \Sigma_C$ and $D\subset \mathbb{RP}^1$ we have
\[
m(E\times D)
=
\sum_{r\in C} m\big((E\cap [r])\times D\big)
=
\sum_{r\in C} \mu_C(E\cap [r])\, \zeta_r(D).
\]

Next define
\[
\Phi: \Sigma_C \times \mathbb{RP}^1 \to \mathbb R,
\qquad
\big( \mathbf r,[v] \big)
\longmapsto
\log \frac{\|B_{r_0}v\|_1}{\|v\|_1},
\]
and $\Psi: \widehat\Sigma_C\times \mathbb{RP}^1\to \mathbb R$ by $\Psi \big( \widehat{\mathbf r},[v] \big) = \Phi \big( \Pi_C(\widehat{\mathbf r}), [v] \big)$. Then, Birkhoff's theorem gives
\[
\int_{\Sigma_C\times \mathbb{RP}^1} \Phi\,dm
=
\int_{\widehat\Sigma_C\times \mathbb{RP}^1} \Psi\,d\widehat m
=
\int_{\widehat\Sigma_C\times \mathbb{RP}^1} \widehat\Psi\, d\widehat m,
\]
where the time average $\widehat\Psi$ is defined for $\widehat m$-almost every $\big( \widehat{\mathbf r},[v] \big)$ by
\[
\widehat\Psi \big(\widehat{\mathbf r},[v] \big)
=
\lim_{n\to\infty} \frac{1}{n} \sum_{j=0}^{n-1} \Psi\big( \mathbb P(\widehat{\boldsymbol B})^j \big( \widehat{\mathbf r}, [v] \big) \big)
=
\lim_{n\to\infty} \frac{1}{n} \log \frac{\|B^n(\Pi_C(\widehat{\mathbf r}))v\|_1}{\|v\|_1}.
\]

Since $\widehat m$ is concentrated on the graph of directions $E^d$, we have $\widehat\Psi(\widehat{\mathbf r},[v]) = \lambda_+(B,Q|_C)$ for $\widehat m$-a.e.\ $\big( \widehat{\mathbf r},[v] \big)$ by Oseledets' theorem. Therefore, using $m=\mu_C\ltimes \zeta$,
\[
\lambda_+\big(B,Q|_C\big)
=
\int_{\Sigma_C\times \mathbb{RP}^1}\Phi\,dm
=
\sum_{r\in C}\pi_r \int_{\mathbb{RP}^1} \log \frac{\|B_r v\|_1}{\|v\|_1}\, d\zeta_r([v]). \qedhere
\]
\end{proof}

The next aim is to prove that $\zeta_r$ is supported on $\pazocal V_{\rho} \coloneqq \psi([-\rho, \rho])$. For each $\widehat{\mathbf r}\in \widehat\Sigma_C$, define
\[
V_n(\widehat{\mathbf r})
\coloneqq
g_{r_{-1}}\circ g_{r_{-2}}\circ \cdots \circ g_{r_{-n}}(\pazocal V_{\rho}) \quad \text{for \; } n \geq 1,
\qquad
\xi(\widehat{\mathbf r})
\coloneqq
\bigcap_{n\geq 1}V_n(\widehat{\mathbf r}).
\]
Also set
\[
v_n(\widehat{\mathbf r})
\coloneqq
\psi^{-1}\big(V_n(\widehat{\mathbf r})\big)
=
f_{r_{-1}}\circ f_{r_{-2}}\circ \cdots \circ f_{r_{-n}}([-\rho, \rho]).
\]
Each $v_n(\widehat{\mathbf r})$ is a nonempty compact interval, and the family is nested:
\[
v_n(\widehat{\mathbf r}) \supset v_{n+1}(\widehat{\mathbf r})
\qquad\text{for every } n\geq 1.
\]
Moreover, because the family $(f_r)_{r\in C}$ is uniformly contracting in hyperbolic metric, $\bigcap_{n\geq 1}v_n(\widehat{\mathbf r})$ consists of a single point of $[-\rho, \rho]$, and hence $\xi(\widehat{\mathbf r})$ is a point in $\pazocal V_{\rho}$. We will think of $\xi$ as a map $\widehat\Sigma_C \to \pazocal V_\rho$.

Let us show that $\xi$ is invariant under the projective cocycle. Fix $\widehat{\mathbf r}=(r_n)_{n\in\mathbb Z}\in \widehat\Sigma_C$. Then,
\[
V_n\big( \widehat\sigma(\widehat{\mathbf r}) \big)
=
g_{r_0}\circ g_{r_{-1}}\circ\cdots\circ g_{r_{-n+1}}(\pazocal V_{\rho})
=
g_{r_0}\big( V_{n-1}(\widehat{\mathbf r}) \big)
\quad (n\geq 2).
\]
Since the compact sets $V_k(\widehat{\mathbf r})$ are nested and $g_{r_0}: \pazocal V_{\rho} \to \pazocal V_{\rho}$ is a continuous injection, we have
\begin{equation} \label{eq: invariance of xi}
\xi \big( \widehat\sigma(\widehat{\mathbf r}) \big)
=
\bigcap_{n\geq 1}V_n\big(\widehat\sigma(\widehat{\mathbf r})\big)
=
\bigcap_{n\geq 2}V_n\big(\widehat\sigma(\widehat{\mathbf r})\big)
=
\bigcap_{k\geq 1} g_{r_0}\big(V_k(\widehat{\mathbf r})\big)
=
g_{r_0}\left( \bigcap_{k\geq 1}V_k(\widehat{\mathbf r}) \right)
=
g_{r_0}\big(\xi(\widehat{\mathbf r})\big).
\end{equation}

\begin{claim}
We have
\[
E^d(\widehat{\mathbf r})=\xi(\widehat{\mathbf r}) \in \pazocal V_\rho
\qquad\text{for every } \widehat{\mathbf r}\in \widehat\Sigma_C.
\]
\end{claim}

\begin{proof}
First consider a periodic point $\widehat{\mathbf p}=(p_n)_{n\in\mathbb Z}\in \widehat\Sigma_C$ of period $\ell\geq 1$. Define
\[
\pazocal G_{\widehat{\mathbf p}}
\coloneqq
g_{p_{\ell-1}}\circ\cdots\circ g_{p_0}: \pazocal V_{\rho} \to \pazocal V_{\rho}.
\]
In the chart $\psi$, this map is
\[
f_{p_{\ell-1}} \circ \cdots \circ f_{p_0}: [-\rho, \rho] \to [-\rho, \rho].
\]
Since $f_{p_{\ell-1}} \circ \cdots \circ f_{p_0}$ is a $\rho^\ell$-contraction, $\pazocal G_{\widehat{\mathbf p}}$ has a unique attracting fixed point in $\pazocal V_{\rho}$. By applying equation \eqref{eq: invariance of xi} $\ell$ times, we see that this fixed point is $\xi(\widehat{\mathbf p})$.

On the other hand, by \eqref{eq: invariance of dominant bundle}, the line $E^d(\widehat{\mathbf p})$ is invariant under $\pazocal G_{\widehat{\mathbf p}}$:
\[
\pazocal G_{\widehat{\mathbf p}}E^d(\widehat{\mathbf p}) = [B^\ell(\widehat{\mathbf p})]\,E^d(\widehat{\mathbf p}) = E^d \big( \widehat\sigma^\ell(\widehat{\mathbf p}) \big) = E^d(\widehat{\mathbf p}).
\]
Furthermore, this is the unique dominant eigendirection of $\pazocal G_{\widehat{\mathbf p}}$ by projective uniform hyperbolicity. By uniqueness of the attracting fixed point of $\pazocal G_{\widehat{\mathbf p}}$, we conclude that
\[
E^d(\widehat{\mathbf p}) = \xi(\widehat{\mathbf p})
\qquad\text{for every periodic } \widehat{\mathbf p}\in \widehat\Sigma_C.
\]

Since $Q|_C$ is primitive, periodic points are dense in $\widehat\Sigma_C$. Also, the map $\xi: \widehat\Sigma_C \to \pazocal V_\rho$ is continuous; if the $-N, \ldots, -1$ coordinates of the two points $\widehat{\mathbf r}$, $\widehat{\mathbf s}$ coincide, then the points $\psi^{-1}\big( \xi(\widehat{\mathbf r}) \big)$, $\psi^{-1}\big( \xi(\widehat{\mathbf s}) \big)$ differ by at most $O(\rho^N)$ in the hyperbolic distance. Together with the continuity of the dominant bundle $\widehat{\mathbf r}\mapsto E^d(\widehat{\mathbf r})$, we conclude that
\[
E^d(\widehat{\mathbf r})=\xi(\widehat{\mathbf r}) \in \pazocal V_\rho
\qquad\text{for every } \widehat{\mathbf r}\in \widehat\Sigma_C. \qedhere
\]
\end{proof}

Thus, $\zeta_r$ has support on $\pazocal V_\rho$ by equation \eqref{eq: the definition of zeta r}, and can be regarded as a probability measure on $\pazocal V_\rho$. Because $m$ is $\mathbb P(\boldsymbol{B})$-invariant, $\zeta$ is $Q|_C$-stationary by Proposition \ref{prop MV characterization of stationary measure vectors}. Then,
\[
\zeta_{r'}
=
\sum_{r\in C} \frac{\pi_rQ_{r,r'}}{\pi_{r'}}\,(g_r)_*\zeta_r.
\]
Let $\widetilde{\zeta} = \big( (\psi^{-1})_* \zeta_r \big)_{r \in C}$. By the equation above, we have $\mathscr H \widetilde{\zeta} = \widetilde{\zeta}$. We also have $\widetilde{\zeta}\in \mathscr P_1^C$ since $\supp\big(\widetilde{\zeta}\big) \subset [-\rho,\rho]$. Then the uniqueness of the fixed point of $\mathscr H$ implies
\[
\widetilde{\zeta} = \eta^C.
\]

Since ${\|\psi(x)\|}_1 = 1$ for any $x \in (-1,1)$, equation \eqref{eq: integral formula intermediate} implies equation \eqref{eq: integral formula of lambda for the lift} under the chart $\psi^{-1}$.
\end{proof}

\subsection{Kernel expansion} \label{subsection: Kernel expansion}

In this subsection, we define the infinite matrix operator $\mathrm T$ acting on some $V_C \subset \ell^\infty(\mathbb N_0)^C$, and obtain an explicit representation of the integral formula in equation \eqref{eq: integral formula of lambda for the lift} in terms of $\mathrm T$. This is done by generalizing the \emph{Kernel expansion} technique in \cite{Alibabaei26-2} to Markov shift settings.

For a M\"obius transformation $g$, we define its ``transposed'' M\"obius transformation $g^\top$ by
\[
g^\top(x)
=
\frac{
ax + c}{
bx + d}
\quad \text{where} \quad
g(x)
=
\frac{
ax + b}{
cx + d}.
\]

\begin{lemma} \label{lemma: Mobius contraction implies transpose contraction}
For any M\"obius transformation $g$ satisfying $g\big(\closure{\mathbb D_1}\big) \subset \mathbb D_1$, we have $g^\top\big(\closure{\mathbb D_1}\big) \subset \mathbb D_1$.
\end{lemma}

\begin{proof}
Let us define $s: \widehat{\mathbb{C}} \to \widehat{\mathbb{C}}$ by $s(x) = -\frac{1}{x}$ for $x \in \mathbb{C}$ and $s(0) = \widehat{\infty}$, \,$s(\widehat{\infty}) = 0$. Then, direct calculation shows that
\begin{equation} \label{eq: transpose identity}
g^{\top}
= s \circ g^{-1} \circ s.
\end{equation}
Since $g: \widehat{\mathbb{C}} \to \widehat{\mathbb{C}}$ is a homeomorphism, we have
\[
g^\top(\closure{\mathbb D_1}) = s \Big( g^{-1} \big( \widehat{\mathbb{C}} \setminus \mathbb D_1 \big) \Big) \subset s \Big( \widehat{\mathbb{C}} \setminus \closure{\mathbb D_1} \Big) = \mathbb D_1. \qedhere
\]
\end{proof}

By this lemma, we have the following for every $r \in C$.
\begin{equation} \label{eq: bounding the transpose of fr}
f_r^\top([-1,1]) \subset (-1,1).
\end{equation}

For $x\in \closure{\mathbb D_1}$ and $c \in \mathbb{C}$, we define $v(x \, ; \, c) \in \ell^\infty(\mathbb{N}_0)$ by
\begin{equation*}
\left( v(x \, ; \, c) \right)_n =
\begin{dcases}
c & (n = 0), \\
- \frac{(-x)^n}{n} & (n \geq 1)
\end{dcases}.
\end{equation*}
Let $V \subset \ell^\infty(\mathbb{N}_0)$ be the algebraic linear span of vectors of the form $v(x \, ; \, c)$.
\[
V = \mathrm{Span}_{\mathbb{C}} \left\{ \, v(x\,;\,c) \setcond x\in \closure{\mathbb D_1}, \; c \in \mathbb{C}\; \right\}.
\]

We write $f_r'$ to denote the usual differentiation of $f_r$ for each $r\in C$. Define an infinite matrix $T_r=(b^{(r)}_{k,n})_{k,n\in\mathbb N_0}$ by
\begin{align*}
& b^{(r)}_{k,0} = 0 \quad \text{for } k \geq 0, \hspace{45pt}
b^{(r)}_{0,n} = {\big( f_r^{\top}(0) \big)}^n \quad \text{for } n\geq1,\\
& b^{(r)}_{k,n} = \sum_{\ell=1}^{\min\{k,n\}} \binom{n}{\ell}\binom{k-1}{\ell-1} \, \, 
{\big( f_r^{\top}(0) \big)}^{n-\ell}\big(-f_r(0)\big)^{k-\ell} {\big( f_r'(0) \big)}^{\ell}\quad \text{for } k,n\geq1.
\end{align*}

Recall the definition of $F: \GL_2(\mathbb R) \to \GL_2(\mathbb R)$ in equation \eqref{eq: definition of F}. For each $r \in C$, write
\[
F(B_r) = \begin{pmatrix} \alpha_r & \beta_r \\ \gamma_r & \delta_r \end{pmatrix}.
\]

\begin{lemma} \label{lemma: denominator is positive}
Suppose $A$ is an invertible non-negative $2 \times 2$ matrix. Let $F(A) = \begin{pmatrix} a & b \\ c & d \end{pmatrix}$. Then $cx + d > 0$ for any $x \in [-1,1]$.
\end{lemma}

\begin{proof}
Write $A=\begin{pmatrix} p & q \\ r & s \end{pmatrix}$ with $p, q, r, s \geq 0$.
By the definition of $F$, we have $|c|=|\tfrac12(p-q+r-s)| \leq \tfrac12(p+q+r+s) = d$, and equality implies that $A$ is singular. Hence $cx+d \geq d - |c| > 0$ for any $x \in [-1,1]$.
\end{proof}

By Lemma \ref{lemma: denominator is positive}, we have $\delta_r > 0$. Since $\left| f_r^\top(0) \right| = \left| \gamma_r /\delta_r \right| < 1$ by equation \eqref{eq: bounding the transpose of fr}, we can define the following map, which is holomorphic on a neighborhood of $\closure{\mathbb{D}_1}$.
\[
\ell_r(x) = \mathrm{Log} \big( \gamma_r \,x + \delta_r \big). \quad \big( x \in \closure{\mathbb{D}_1}, \, r\in C \big).
\]

\begin{lemma} \label{lemma: coboundary identity for Tr}
Let $r \in C$. For any generator $v(x\,;\,c) \in V$,
\[
T_r v(x\,;\,c) \;=\; v\big( f_r(x)\,;\,\ell_r(x) \big) \,-\, v\big( f_r(0)\,;\,\ell_r(0) \big).
\]
Therefore, we can define the linear operator $T_r: V \to V$ via coordinatewise matrix–vector multiplication.
\end{lemma}

\begin{proof}
For a holomorphic function $h(w) = \sum_{n=0}^\infty a_n w^n$ on $\mathbb D_1$, write $v(h) = (a_n)_{n=0}^\infty$ to denote its coefficient vector. We first prove that, if $h$ has bounded Taylor coefficients, then
\begin{equation} \label{eq: application of Tr is composition minus reference}
T_r v(h) = v\big( h\circ f_r^\top - h(0) \big).
\end{equation}

Notice that
\[
f_r^\top(w)
=
f_r^\top(0) + \frac{f_r'(0)w}{1+f_r(0)w}.
\]
By expanding its $n$-th power in powers of $w$ and using the definition of $b^{(r)}_{k,n}$, we have for every $n\geq1$,
\[
\big( f_r^\top(w) \big)^n
=
\sum_{k=0}^\infty b^{(r)}_{k,n}w^k.
\]
By Lemma \ref{lemma: Mobius contraction implies transpose contraction}, every $f_r^\top$ maps $\closure{\mathbb D_1}$ into $\mathbb D_1$. Since $C$ is finite, there are $R>1$ and $0<q<1$ such that every $f_r^\top$ is holomorphic on a neighborhood of $\closure{\mathbb D_R}$ and $\max_{z\in\closure{\mathbb D_R}} \big|f_r^\top(z)\big| \leq q$. Then, for every $k\geq0$ and $n\geq1$,
\[
\big| b^{(r)}_{k,n} \big|
=
\frac{1}{2\pi} \left| \oint_{|z|=R} \frac{\big(f_r^\top(z)\big)^n}{z^{k+1}}dz \right|
\leq
\frac{q^n}{R^k}.
\]

Suppose $h$ has bounded Taylor coefficients. Denote by $\|\cdot\|_\infty$ the $\ell^\infty$-norm on $\ell^\infty(\mathbb N_0)$. Then, for any $w\in\closure{\mathbb D_1}$,
\[
\sum_{k=0}^\infty \sum_{n=1}^\infty \big| b^{(r)}_{k,n} a_n w^k \big|
\;\leq\;
\|v(h)\|_\infty \sum_{k=0}^\infty \sum_{n=1}^\infty \frac{q^n}{R^k}
\;<\;
\infty.
\]
Therefore, we can rearrange the following double sum:
\begin{align*}
\sum_{k=0}^\infty \big( T_r v(h) \big)_k w^k
&=
\sum_{k=0}^\infty \left( \sum_{n=1}^\infty b^{(r)}_{k,n}a_n \right) w^k
=
\sum_{n=1}^\infty a_n \left( \sum_{k=0}^\infty b^{(r)}_{k,n}w^k \right) \\
&=
\sum_{n=1}^\infty a_n \big( f_r^\top(w) \big)^n
=
h\circ f_r^\top(w)-h(0).
\end{align*}
This proves $T_r v(h) = v\big( h\circ f_r^\top - h(0) \big)$.

Take any $x\in\closure{\mathbb D_1}$ and $c\in\mathbb C$, and consider $u:\mathbb D_1\to \mathbb C$ defined by $u(w) = c+\mathrm{Log}(1+xw)$. Its coefficient vector is $v(u)=v(x\,;\,c)$, and its Taylor coefficients are bounded. Now, notice the following identity, which follows by direct calculation.
\[
1+x f_r^\top(w)
=
\frac{\gamma_r x+\delta_r}{\delta_r}
\cdot
\frac{1+f_r(x)w}{1+f_r(0)w}.
\]
Therefore,
\[
u\circ f_r^\top(w)-u(0)
=
\ell_r(x)-\ell_r(0)
+
\mathrm{Log}\big(1+f_r(x)w\big)
-
\mathrm{Log}\big(1+f_r(0)w\big).
\]

By taking the coefficient vector and using equation \eqref{eq: application of Tr is composition minus reference}, we obtain
\begin{equation} \label{eq: Tv}
T_r v(x\,;\,c) \;=\; v\big( f_r(x)\,;\,\ell_r(x) \big) \,-\, v\big( f_r(0)\,;\,\ell_r(0) \big).
\end{equation}
Since $f_r([-1,1]) \subset [-\rho, \rho]$, we have $f_r(\closure{\mathbb D_1}) \subset \closure{\mathbb D_\rho}$ by Lemma \ref{lemma: real contraction implies complex contraction}. Hence $|f_r(x)| \leq \rho <1$. Thus, we have $T_r(V) \subset V$.
\end{proof}

Let
\[
V_C \;\coloneqq\; \bigoplus_{r\in C} V.
\]
The infinite matrix operator $\mathrm T: V_C \to V_C$ is defined by, for $\boldsymbol u =(u_r)_{r\in C}\in V_C$ and $r' \in C$,
\[
\big( \mathrm T \, \boldsymbol u \big)_{r'}
\;\coloneqq\;
\sum_{r\in C}
\frac{\pi_rQ_{r,r'}}{\pi_{r'}}\,T_r u_r.
\]
Also, for $\boldsymbol w =(w_r)_{r\in C}\in \ell^\infty(\mathbb N_0)^C$, the weighted average of the $0$-th coordinates is defined as
\[
[\boldsymbol w]_{\boldsymbol 0} \;\coloneqq\; \sum_{r\in C}\pi_r (\boldsymbol w)_{r,\,0}.
\]
Here, and in what follows, we denote by $(\boldsymbol w)_{r,k}$ the $k$-th entry of $w_r\in \ell^\infty(\mathbb N_0)$.

Define the seed vector $\boldsymbol v = (\boldsymbol v_s)_{s \in C} \in V_C$ by
\[
\boldsymbol v_s
\;\coloneqq\;
\sum_{r\in C} \frac{\pi_r Q_{r,s}}{\pi_s}\, v\big( f_r(0)\,;\,\ell_r(0) \big).
\]

For a finite path $\mathbf r = (r_0, \dots, r_k)\in C^{k+1}$, write the probability of its occurrence by
\[
p(\mathbf r)
\coloneqq
\pi_{r_0} Q_{r_0 r_1} \cdots Q_{r_{k-1} r_k}.
\]
Also define
\[
f_{\mathbf r}^{(0)}
\;\coloneqq\;
\mathrm{id},
\qquad
f_{\mathbf r}^{(m)}
\;\coloneqq\;
f_{r_{m-1}} \circ \cdots \circ f_{r_0}
\quad \text{for } 1 \leq m \leq k+1.
\]

For each $n \geq 1$, define $S_n \in V_C$ by
\[
(S_n)_{r'}
\;\coloneqq\;
\frac{1}{\pi_{r'}} \sum_{\mathbf r=(r_0, \dots, r_{n-1}) \in C^n} p(\mathbf r) Q_{r_{n-1}, r'}\,
v\!\left( f_{\mathbf r}^{(n)}(0) \,;\, \ell_{r_{n-1}} \big( f_{\mathbf r}^{(n-1)}(0) \big) \right).
\]

\begin{proposition} \label{prop: Sn identity}
For every $n \geq 1$,
\[
\sum_{j=0}^{n-1} \mathrm T^j \boldsymbol v = S_n.
\]
\end{proposition}

\begin{proof}
For $n=1$, we have
\begin{equation*}
(S_1)_{r'}
=
\frac{1}{\pi_{r'}} \sum_{r_0\in C} \pi_{r_0} Q_{r_0, r'}\, v\big( f_{r_0}(0) \,;\, \ell_{r_0}(0) \big)
=
\boldsymbol v_{r'}
\end{equation*}
for every $r' \in C$, by the definition of $\boldsymbol v$. Hence $S_1=\boldsymbol v$.

Now assume that the identity holds for some $n \geq 1$. Then
\[
\sum_{j=0}^{n} \mathrm T^j \boldsymbol v
=
\boldsymbol v + \mathrm T S_n.
\]
Therefore it is enough to prove that $\mathrm T S_n = S_{n+1}-\boldsymbol v$. Fix $r' \in C$. By the definition of $\mathrm T$,
\[
(\mathrm T S_n)_{r'}
=
\sum_{t \in C} \frac{\pi_{t} Q_{t, r'}}{\pi_{r'}}\, T_t \big( (S_n)_t \big).
\]
Using the definition of $(S_n)_t$ and the linearity of $T_t$, (while renaming $t$ to $r_n$,)
\begin{align*}
(\mathrm T S_n)_{r'}
&= \frac{1}{\pi_{r'}} \sum_{t \in C} \pi_t Q_{t, r'} T_t \left( \frac{1}{\pi_t} \sum_{\mathbf s = (r_0,\dots,r_{n-1}) \in C^n} p(\mathbf s) Q_{r_{n-1}, \,t} \; v\!\left( f_{\mathbf s}^{(n)}(0) \,;\, \ell_{r_{n-1}} \big(f_{\mathbf s}^{(n-1)}(0)\big) \right) \right) \\
&=
\frac{1}{\pi_{r'}}
\sum_{\mathbf r = (r_0,\dots,r_n) \in C^{n+1}}
p(\mathbf r) Q_{r_n, r'} \, T_{r_n} v\!\left( f_{\mathbf r}^{(n)}(0) \,;\, \ell_{r_{n-1}} \big(f_{\mathbf r}^{(n-1)}(0)\big) \right).
\end{align*}
By Lemma \ref{lemma: coboundary identity for Tr}, for any $\mathbf r = (r_0,\dots,r_n) \in C^{n+1}$,
\begin{align*}
T_{r_n}
v\Big( f_{\mathbf r}^{(n)}(0) \,;\, \ell_{r_{n-1}}\big( f_{\mathbf r}^{(n-1)}(0) \big) \Big)
&\;=\;
v\Big( f_{\mathbf r}^{(n+1)}(0) \,;\, \ell_{r_n}\big( f_{\mathbf r}^{(n)}(0) \big) \Big)
- v\Big( f_{r_n}(0) \,;\, \ell_{r_n}(0) \Big).
\end{align*}
Hence $(\mathrm T S_n)_{r'}$ is the difference of two terms. The first term is,
\[
\frac{1}{\pi_{r'}}
\sum_{\mathbf r = (r_0,\dots,r_n) \in C^{n+1}}
p(\mathbf r) Q_{r_n, r'} \,
v\!\left( f_{\mathbf r}^{(n+1)}(0) \,;\, \ell_{r_n}\big( f_{\mathbf r}^{(n)}(0) \big) \right)
=
(S_{n+1})_{r'}.
\]
The second term computes as, by using the stationarity of the vector $(\pi_r)_{r\in C}$,
\begin{align*}
&- \frac{1}{\pi_{r'}} \sum_{\mathbf r = (r_0,\dots,r_n)\in C^{n+1}} p(\mathbf r)Q_{r_n,r'} \; v\big(f_{r_n}(0) \,;\, \ell_{r_n}(0)\big)  \\
&=
- \frac{1}{\pi_{r'}} \sum_{r_n \in C}
\left( \sum_{\mathbf s = (r_0,\dots,r_{n-1})\in C^n} p(\mathbf s) Q_{r_{n-1}, r_n}\right)
Q_{r_n,r'} \;v\big(f_{r_n}(0) \,;\, \ell_{r_n}(0)\big) \\
&=
- \frac{1}{\pi_{r'}} \sum_{r_n \in C}
\Big( \left(\pi_{r_0}\right)_{r_0\in C} {(Q|_C)}^n\Big)_{r_n}
Q_{r_n,r'} \;v\big(f_{r_n}(0) \,;\, \ell_{r_n}(0)\big) \\
&=
- \frac{1}{\pi_{r'}} \sum_{r_n \in C}
\pi_{r_n}
Q_{r_n,r'}\; v\big(f_{r_n}(0) \,;\, \ell_{r_n}(0)\big)
=
-\boldsymbol v_{r'}.
\end{align*}
Thus $(\mathrm T S_n)_{r'} = (S_{n+1})_{r'}-\boldsymbol v_{r'}$, and $\mathrm T S_n = S_{n+1}-\boldsymbol v$.
\end{proof}

The identity obtained in the proposition above states that, in fact, the finite sum of iterates $\mathrm T^j \boldsymbol v$ is exactly the ``discrete approximation'' of $\eta^C$. Hence, taking the limit we recover the integral.

\begin{proposition} \label{prop: infinite sum equals the integral}
We have
\[
\sum_{n=0}^\infty \big[ \mathrm T^n \boldsymbol v\big]_{\boldsymbol 0}
=
\sum_{r\in C} \pi_r \int_{(-1,1)} \log {\|B_r \psi(x)\|}_1 \,d\eta_r^C(x).
\]
\end{proposition}

\begin{proof}
Let
\[
\delta^0
\coloneqq
(\delta_0)_{r\in C},
\qquad
\boldsymbol \nu^{(n)}
\coloneqq
\mathscr H^n(\delta^0)
\quad \text{for } n\geq 0,
\]
where $\delta_a$ denotes the Dirac mass at $a \in (-1,1)$.

Write $\boldsymbol \nu^{(n)} = \big( \nu^{(n)}_r \big)_{r \in C}$. We first prove that for every $n\geq 1$ and every $r\in C$,
\begin{equation} \label{eq: explicit formula for nu-n}
\nu_r^{(n)}
=
\frac{1}{\pi_r} \sum_{\mathbf r = (r_0,\dots,r_{n-1})\in C^n} p(\mathbf r) Q_{r_{n-1}, r}\, \delta_{f_{\mathbf r}^{(n)}(0)}.
\end{equation}
For $n=1$, this is immediate from the definition of $\delta^0$ and $\mathscr H$. Assume that the formula holds for some $n\geq 1$. Then, for any $r' \in C$,
\begin{align*}
\nu_{r'}^{(n+1)}
&=
(\mathscr H \boldsymbol \nu^{(n)})_{r'}
=
\sum_{r\in C}
\frac{\pi_r Q_{r,r'}}{\pi_{r'}}\,
(f_r)_* \nu_r^{(n)} \\
&=
\sum_{r\in C}
\frac{\pi_r Q_{r,r'}}{\pi_{r'}}\,
(f_r)_*
\left(
\frac{1}{\pi_r} \sum_{\mathbf s = (r_0,\dots,r_{n-1})\in C^n} p(\mathbf s) Q_{r_{n-1}, r}\, \delta_{f_{\mathbf s}^{(n)}(0)}
\right) \\
&=
\frac{1}{\pi_{r'}} \sum_{\mathbf r = (r_0,\dots,r_n)\in C^{n+1}} p(\mathbf r) Q_{r_n, r'}\, \delta_{f_{\mathbf r}^{(n+1)}(0)}.
\end{align*}
This proves equation \eqref{eq: explicit formula for nu-n} by induction.

Let $n \geq 2$. By the definition of $S_n$ and equation \eqref{eq: explicit formula for nu-n},

\begin{align*}
[S_n]_{\boldsymbol 0}
&=
\sum_{r_n\in C}\pi_{r_n} \big(S_n\big)_{r_n,\,0}
=
\sum_{r_n \in C} \; \sum_{\mathbf s = (r_0,\dots,r_{n-1})\in C^n} p(\mathbf s) Q_{r_{n-1}, r_n}\, \ell_{r_{n-1}}\big(f_{\mathbf s}^{(n-1)}(0) \big) \\
&=
\sum_{\mathbf s = (r_0,\dots,r_{n-1})\in C^n} p(\mathbf s)
\left(
\sum_{r_n \in C} Q_{r_{n-1}, r_n}
\right)
\ell_{r_{n-1}}\big(f_{\mathbf s}^{(n-1)}(0) \big) \\
&=
\sum_{r \in C} \; \sum_{\mathbf r = (r_0,\dots,r_{n-2})\in C^{n-1}} p(\mathbf r) Q_{r_{n-2}, r} \; \ell_r\big(f_{\mathbf r}^{(n-1)}(0) \big)
\;=\;
\sum_{r \in C} \pi_r \int_{(-1,1)} \ell_r(x) d\nu^{(n-1)}_r(x).
\end{align*}
The same identity holds for $n=1$. Therefore, by Proposition \ref{prop: Sn identity}, for every $n\geq 1$,
\begin{equation} \label{eq: partial sum formula for Lambda-C}
\left[
\sum_{j=0}^{n-1} \mathrm T^j \boldsymbol v
\right]_{\boldsymbol 0}
=
\sum_{r\in C}
\pi_r \int_{(-1,1)}\!\ell_r(x)\,d\nu_r^{(n-1)}\!(x).
\end{equation}

Since $\eta^C=(\eta_r^C)_{r\in C}$ is the unique fixed point of $\mathscr H$, Proposition \ref{prop: contraction of H} implies that
\[
\boldsymbol \nu^{(n)}
=
\mathscr H^n(\delta^0)
\;\xrightarrow[]{n \to \infty}\;
\eta^C
\qquad \text{in } (\mathscr P_1^C,\pazocal W).
\]
In particular, $\nu_r^{(n)}$ converges weakly to $\eta_r^C$ for each $r\in C$ by \cite[Theorem 6.9]{Villani}. Since each $\ell_r$ is bounded and continuous on $[-1,1]$, we conclude that
\[
\sum_{j=0}^\infty \big[ \mathrm T^j \boldsymbol v \big]_{\boldsymbol 0}
=
\lim_{n\to\infty}
\sum_{r\in C}
\pi_r \int_{(-1,1)} \ell_r(x)\,d\nu_r^{(n-1)}(x)
=
\sum_{r\in C}
\pi_r \int_{(-1,1)} \ell_r(x)\,d\eta_r^C(x).
\]
Finally, we use the following identity and conclude the proof.
\[
\ell_r(x)
\;=\;
\Log(\gamma_r x + \delta_r)
\;=\;
\log {\|B_r \psi(x)\|}_1.
\]
\end{proof}

\begin{remark}
The argument above also proved that for any $r \in C$,
\begin{align*}
\sum_{j=0}^{n-1} \big( \mathrm T^j \boldsymbol v \big)_{r,\,0}
&\;=\;
\sum_{t \in C} \frac{\pi_t Q_{t, r} }{\pi_r} \int_{(-1,1)} \log {\|B_t \psi(x)\|}_1 \,d\nu^{(n-1)}_t(x) \\
&\;\xrightarrow[]{n \to \infty}\;
\sum_{t \in C} \frac{\pi_t Q_{t, r} }{\pi_r} \int_{(-1,1)} \log {\|B_t \psi(x)\|}_1 \,d\eta_t^C(x),
\end{align*}
and for any $k \geq 1$,
\begin{align*}
\sum_{j=0}^{n-1} \big( \mathrm T^j \boldsymbol v \big)_{r,\,k}
\;=\;
\frac{(-1)^{k-1}}{k} \int_{(-1,1)} x^k \,d\nu^{(n)}_r(x)
\;\xrightarrow[]{n \to \infty}\;
\frac{(-1)^{k-1}}{k} \int_{(-1,1)} x^k \,d\eta_r^C(x).
\end{align*}
\end{remark}

Recall that $d$ is the original period of $Q|_C$, before acceleration.
\begin{theorem} \label{thm: infinite sum equals lambda}
We have
\begin{equation} \label{eq: precise infinite series representation of the Lyapunov exponent}
\lambda_+(A, P)
\;=\;
\sum_{n=0}^\infty \big[ \mathrm T^n \widehat{\boldsymbol v}]_{\boldsymbol 0}
\quad \text{where} \quad
\widehat{\boldsymbol v} = \frac1d \boldsymbol v.
\end{equation}
\end{theorem}

\begin{proof}
This follows from Proposition \ref{prop: aperiodic reduction}, Lemma \ref{lemma: Lyapunov exponent identity for the lift}, Proposition \ref{prop: integral formula of lambda for the lift}, and Proposition \ref{prop: infinite sum equals the integral}.
\end{proof}

With Theorem \ref{thm: infinite sum equals lambda}, we have completed the proof of Theorem \ref{thm: explicit representation of the Lyapunov exponent in the introduction}.

\section{Evaluation and Polynomial-time Convergence} \label{section: Evaluation and Polynomial-time Convergence}

\noindent\textbf{Overview of this section.}
In this section we prove Theorem \ref{thm: polynomial computation of the Lyapunov exponent in introduction} from the introduction. Let $(\Sigma, P)$ be a transitive Markov shift, $A$ be the locally constant function on $\Sigma$ defined in equation \eqref{eq: definition of locally constant function A}, and write $\lambda = \lambda_+(A, P)$. All other notation is retained from the previous section, including the aperiodic system $(\Sigma_C, Q|_C)$ obtained after the aperiodic reduction. The approximation of $\lambda$ is done by truncating the infinite series in equation \eqref{eq: precise infinite series representation of the Lyapunov exponent} to its first $n$ terms, and then replacing the infinite matrix $\mathrm T$ by its finite truncation $\mathrm T_{m\times m}$. This decomposes the error into two terms of different nature,
\[
\left|
\lambda - \sum_{\ell=0}^{n-1} \big[ \mathrm T_{m\times m}^\ell \widehat{\boldsymbol v} \big]_{\boldsymbol 0} 
\right|
\;\leq\;
\left|
\sum_{\ell = n}^\infty \big[ \mathrm T^\ell \widehat{\boldsymbol v} \big]_{\boldsymbol 0}
\right|
\;+\;
\left|
\sum_{\ell=0}^{n-1} \Big[ \mathrm T^\ell \widehat{\boldsymbol v} \;-\; \mathrm T_{m\times m}^\ell \widehat{\boldsymbol v} \Big]_{\boldsymbol 0}
\right|.
\]
Subsection \ref{subsection: Cutting the tail of the series} treats the first term. This is a straightforward application of the contraction property of the Markov operator $\mathscr H$ proved earlier. The second term requires more delicate analysis: to control how the truncation error accumulates through iteration, we introduce integral representations for $\mathrm T$ and $\mathrm T_{m\times m}$ in Subsection \ref{subsection: Integral representations}. These representations are extended to the iterates $\mathrm T^n$ and $\mathrm T_{m\times m}^n$ in Subsection \ref{subsection: Integral representations for iterates}. Finally, Subsection \ref{subsection: Integral operators and telescoping} defines the corresponding holomorphic integral operators, proves their oscillation bounds, and combines them with a telescoping identity to obtain an explicit estimate for the second term.

\medskip

For each $r \in C$ and $m \in \mathbb N$, let $T_{r,m}$ denote the infinite matrix obtained by keeping the upper-left $m \times m$ submatrix of $T_r$ and setting all other entries to zero. More precisely, writing $T_{r,m} = \left(c^{(r)}_{k,n}\right)_{k,n\in \mathbb N_0}$,
\[
c^{(r)}_{k,n}
\coloneqq
\begin{dcases}
b^{(r)}_{k,n} & \text{if } 0 \leq k, n \leq m-1, \\
0 & \text{otherwise.}
\end{dcases}
\]
The map $T_{r,m}: \ell^\infty(\mathbb N_0) \to \ell^\infty(\mathbb N_0)$ is defined by coordinatewise matrix-vector multiplication. Then, we define $\mathrm T_{m \times m} : \ell^\infty(\mathbb N_0)^C \to \ell^\infty(\mathbb N_0)^C$ by
\[
\big( \mathrm T_{m \times m} \boldsymbol u \big)_{r'}
\;\coloneqq\;
\sum_{r\in C} \frac{\pi_r Q_{r,r'}}{\pi_{r'}} \, T_{r,m}u_r
\qquad \text{for } \boldsymbol u=(u_r)_{r\in C}\in \ell^\infty(\mathbb N_0)^C,\;\; r' \in C.
\]
Theorem \ref{thm: precise bounds} presents a self-contained statement of the final bound.

\medskip

This section follows the same strategy as \cite[Section 4]{Alibabaei26-2}, extended here to the Markov-shift setting.

\subsection{Cutting the tail of the series} \label{subsection: Cutting the tail of the series}

We will work with $\boldsymbol v$, and at the very end we divide by $d$ to obtain a bound in terms of $\widehat{\boldsymbol v}$.
\begin{proposition} \label{prop: cutting the tail of the Neumann series}
Let $0 < \rho < 1$ be as in equation \eqref{eq: every f is a rho-contraction}. Define
\[
E_C
=
\frac{1}{1-\rho} \left(
\max_{r\in C}
\frac{\left| f_r^{\top}(0) \right|}{1+\sqrt{1-{\left| f_r^{\top}(0) \right|}^2}}
\right) \left(
\sum_{r\in C}\pi_r\, d_{\mathrm{hyp}}\big( f_r(0),0 \big)
\right).
\]
Then, for any $n \in \mathbb N$,
\[
\Bigg|
\sum_{k=0}^\infty \big[ \mathrm T^k \boldsymbol v \big]_{\boldsymbol 0}
-\sum_{k=0}^{n-1} \big[ \mathrm T^k \boldsymbol v \big]_{\boldsymbol 0}
\Bigg|
\leq
E_C \rho^{n-1}.
\]
\end{proposition}

Before giving the proof, we first record the following lemma.
\begin{lemma} \label{lemma: hyperbolic lipschitz constant}
Let $g \in C^1(-1,1)$. Define
\[
L(g) = \sup_{x \in (-1,1)} \frac{1-x^2}{2} \left| g'(x) \right|.
\]
If $L(g)$ is finite, then for any $\nu_1, \nu_2 \in \mathscr{P}_1(-1,1)$,
\[
\left| \int g \, d\nu_1 - \int g \, d\nu_2 \right| \leq L(g) \, W_1(\nu_1, \nu_2).
\]
\end{lemma}

\begin{proof}
First, for any $x, y \in (-1, 1)$,
\begin{align*}
\left| g(x) - g(y) \right|
&=
\left| \int_x^y g'(t)\,dt \right|
=
\left| \int_x^y \frac{1-t^2}{2} g'(t)\,\frac{2\,dt}{1-t^2} \right|
\leq
L(g) \, d_{\mathrm{hyp}}(x,y).
\end{align*}
In particular, $\int g\,d\nu$ exists for any $\nu \in \mathscr{P}_1(-1,1)$ by letting $y=0$.
Take any coupling $\eta \in \Pi(\nu_1,\nu_2)$. Then
\[
\left| \int g\,d\nu_1 - \int g\,d\nu_2 \right|
=
\left| \int \big( g(x)-g(y) \big)\,d\eta(x,y) \right|
\leq
L(g) \int d_{\mathrm{hyp}}(x,y)\,d\eta(x,y).
\]
Taking the infimum over $\eta$ proves the desired inequality.
\end{proof}

\begin{proof}[Proof of Proposition \ref{prop: cutting the tail of the Neumann series}]
We recall the notations: for each $r\in C$, write
\[
F(B_r)
=
\begin{pmatrix}
\alpha_r & \beta_r \\
\gamma_r & \delta_r
\end{pmatrix},
\qquad
\ell_r(x)
=
\mathrm{Log}\big( \gamma_r x + \delta_r \big)
\quad \text{for \;} x \in (-1,1).
\]
Also recall that
$
\boldsymbol \nu^{(n)}
=
\big( \nu_r^{(n)} \big)_{r\in C}
=
\mathscr H^n(\delta^0)
$
for each $n \geq 0$, where $\delta^0 = (\delta_0)_{r\in C}$ and $\delta_0$ denotes the Dirac measure at $0$.

A direct calculus computation shows that $\sup_{x\in(-1,1)} \frac{1-x^2}{1+ax} = \frac{2}{1+\sqrt{1-a^2}}$ for $a\in(-1,1)$. Therefore, for each $r\in C$,
\[
L(\ell_r)
=
\sup_{x\in(-1,1)}
\frac{1-x^2}{2}
\left| \frac{\gamma_r}{\gamma_r x+\delta_r} \right|
=
\frac{1}{2}\cdot \left| f_r^\top(0) \right| \sup_{x\in(-1,1)} \frac{1-x^2}{1+\left| f_r^\top(0) \right|x}
=
\frac{\left| f_r^\top(0) \right|}{1+\sqrt{1-{\left| f_r^\top(0) \right|}^2}}.
\]

Since $\ell_r(x)=\log \|B_r\psi(x)\|_1$ for $x\in[-1,1]$, Proposition
\ref{prop: infinite sum equals the integral} and equation
\eqref{eq: partial sum formula for Lambda-C} give
\begin{align*}
\Bigg| \sum_{k=0}^\infty \big[ \mathrm T^k \boldsymbol v \big]_{\boldsymbol 0} - \sum_{k=0}^{n-1} \big[ \mathrm T^k \boldsymbol v \big]_{\boldsymbol 0} \Bigg|
&=
\left| \sum_{r\in C} \pi_r \left( \int \ell_r(x)\,d\eta_r^C(x) - \int \ell_r(x)\,d\nu_r^{(n-1)}(x) \right) \right| \\
&\leq
\sum_{r\in C} \pi_r L(\ell_r) W_1\big( \eta_r^C, \nu_r^{(n-1)} \big)
\leq
\left( \max_{r\in C} L(\ell_r) \right) \pazocal W \big( \eta^C, \boldsymbol \nu^{(n-1)} \big),
\end{align*}
where in the second line we used Lemma \ref{lemma: hyperbolic lipschitz constant}.

Recall that by Proposition \ref{prop: contraction of H}, the map $\mathscr H$ is a $\rho$-contraction on $(\mathscr P_1^C, \pazocal W)$, and $\boldsymbol \nu^{(m)} \to \eta^C$ in $\pazocal W$. Therefore, by applying the triangle inequality to $\pazocal W(\boldsymbol \nu^{(m)}, \boldsymbol \nu^{(n-1)})$ and letting $m\to\infty$,
\[
\pazocal W \big( \eta^C, \boldsymbol \nu^{(n-1)} \big)
\leq
\sum_{k=n-1}^\infty \pazocal W \big( \boldsymbol \nu^{(k+1)}, \boldsymbol \nu^{(k)} \big).
\]
Using $\boldsymbol \nu^{(k)}=\mathscr H^k(\delta^0)$,
\begin{align*}
\pazocal W\big( \eta^C\!, \boldsymbol \nu^{(n-1)} \big) \hspace{-2pt}
\leq
\hspace{-3pt}\sum_{k=n-1}^\infty \!\pazocal W\big( \mathscr H^k(\boldsymbol \nu^{(1)}), \mathscr H^k(\delta^0) \big)
\leq
\hspace{-3pt}\sum_{k=n-1}^\infty \!\rho^k \, \pazocal W\big( \boldsymbol \nu^{(1)}\!, \delta^0 \big)
=
\frac{\rho^{n-1}}{1-\rho}\, \pazocal W\big( \boldsymbol \nu^{(1)}\!, \delta^0 \big).
\end{align*}

It remains to compute $\pazocal W\big( \boldsymbol \nu^{(1)}, \delta^0 \big)$. For each
$r' \in C$,
\[
\nu_{r'}^{(1)}
=
(\mathscr H\delta^0)_{r'}
=
\sum_{r \in C} \frac{\pi_rQ_{r,r'}}{\pi_{r'}}\,\delta_{f_r(0)}.
\]
Since $\delta_0$ is a Dirac measure,
\[
W_1\big( \nu_{r'}^{(1)}, \delta_0 \big)
=
\int d_{\mathrm{hyp}}(x,0)\,d\nu_{r'}^{(1)}(x)
=
\sum_{r \in C} \frac{\pi_rQ_{r,r'}}{\pi_{r'}}\,d_{\mathrm{hyp}}\big( f_r(0),0 \big).
\]
Hence
\begin{align*}
\pazocal W\big( \boldsymbol \nu^{(1)}, \delta^0 \big)
=
\sum_{r' \in C} \pi_{r'} W_1\big( \nu_{r'}^{(1)}, \delta_0 \big)
=
\sum_{r' \in C}\sum_{r\in C} \pi_rQ_{r,r'}\, d_{\mathrm{hyp}}\big( f_r(0),0 \big)
=
\sum_{r \in C} \pi_r\, d_{\mathrm{hyp}}\big( f_r(0),0 \big).
\end{align*}

Combining all the estimates, we obtain
\begin{align*}
\Bigg|
\sum_{k=0}^\infty \big[ \mathrm T^k \boldsymbol v \big]_{\boldsymbol 0}
-\sum_{k=0}^{n-1} \big[ \mathrm T^k \boldsymbol v \big]_{\boldsymbol 0}
\Bigg|
&\leq
\left( \max_{r\in C} \frac{\left| f_r^\top(0) \right|}{1+\sqrt{1-{\left| f_r^\top(0) \right|}^2}}
\right) \frac{\rho^{n-1}}{1-\rho} \left( \sum_{r\in C}\pi_r\, d_{\mathrm{hyp}}\big( f_r(0),0 \big) \right). \;\qedhere
\end{align*}
\end{proof}

\subsection{Integral representations} \label{subsection: Integral representations}

Here we introduce the integral representations for the full operator and the truncated operator. Let $i = \sqrt{-1}$.

\begin{proposition} \label{prop: integral representation}
Fix $r \in C$. Let $m\geq 2$ be a natural number, $x \in \mathbb{D}_\rho$ and $c \in \mathbb{C}$. Then, for any $k \geq 1$,
\begin{equation} \label{eq: integral representation of Tr}
\bigg( T_r v( x \, ; \, c ) \bigg)_k
=
- \frac{1}{2\pi i} \oint_{\partial \mathbb D_\rho} \frac{1}{1-z} \frac{\left( - f_r \left( \frac{x}{z} \right) \right)^k}{k}\, dz,
\end{equation}
and for any $1 \leq k \leq m-1$,
\begin{equation} \label{eq: integral representation of Trm}
\bigg( T_{r,m} v( x \, ; \, c ) \bigg)_k
=
- \frac{1}{2\pi i} \oint_{\partial \mathbb D_\rho} \frac{1-z^{m-1}}{1-z} \frac{\left( - f_r \left( \frac{x}{z} \right) \right)^k}{k}\, dz.
\end{equation}
\end{proposition}

Let us first prove the following lemma.

\begin{lemma} \label{lemma: coboundary property}
Let $h(w) = \sum_{n=0}^\infty a_n w^n$ be holomorphic on $\mathbb{D}_1$. Then, for any natural number $m\geq2$ and $w\in\mathbb{D}_\rho$,
\begin{align*}
&\frac{1}{2\pi i} \oint_{\partial \mathbb D_\rho} \frac{1}{1-z} h\left( \frac{w}{z} \right) dz
\;=\;
\sum_{n=1}^\infty a_n w^n
\;=\;
h(w)-h(0),
\\
&\frac{1}{2\pi i} \oint_{\partial \mathbb D_\rho} \frac{1-z^{m-1}}{1-z} h\left( \frac{w}{z} \right) dz
\;=\;
\sum_{n=1}^{m-1} a_n w^n.
\end{align*}
\end{lemma}

\begin{proof}
Take $\rho_0$ so that $|w|/\rho < \rho_0 < 1$. Since $h$ is holomorphic on $\mathbb D_1$, its Taylor series converges absolutely and uniformly on $\closure{\mathbb D_{\rho_0}}$. Then, we have the following absolutely and uniformly converging series for $|z|=\rho$.
\[
h\left( \frac{w}{z} \right)
=
\sum_{\ell=0}^\infty a_\ell w^\ell z^{-\ell}.
\]
Since $\frac{1}{1-z} = \sum_{q=0}^\infty z^q$ also converges absolutely and uniformly for $|z|=\rho$, we have
\[
\frac{1}{1-z} h\left( \frac{w}{z} \right)
=
\sum_{\ell,q\geq0} a_\ell w^\ell z^{q-\ell},
\]
with absolute and uniform convergence. By integrating termwise,
\[
\frac{1}{2\pi i} \oint_{\partial \mathbb D_\rho} \frac{1}{1-z} h\left( \frac{w}{z} \right) dz
=
\sum_{\ell=1}^\infty a_\ell w^\ell
=
h(w)-h(0).
\]
The same argument, with $\frac{1-z^{m-1}}{1-z} = \sum_{q=0}^{m-2}z^q$, proves the second assertion.
\end{proof}

\begin{proof}[Proof of Proposition \ref{prop: integral representation}]
Fix $r\in C$, $x\in\mathbb D_\rho$, and $c\in\mathbb C$. By equation \eqref{eq: Tv} and Lemma \ref{lemma: coboundary property}, for any $k\geq1$,
\begin{equation*}
\bigg( T_r v(x\,;\,c) \bigg)_k
=
-\left\{ \frac{\big(-f_r(x)\big)^k}{k} - \frac{\big(-f_r(0)\big)^k}{k} \right\}
=
-\frac{1}{2\pi i} \oint_{\partial \mathbb D_\rho} \frac{1}{1-z} \frac{\left(-f_r\left(\frac{x}{z}\right)\right)^k}{k}\,dz.
\end{equation*}
This proved the formula for $T_r$ in equation \eqref{eq: integral representation of Tr}.

Next, fix $1\leq k\leq m-1$, and define a holomorphic function $h: \mathbb D_1\to\mathbb C$ by
\[
h(w) \;\coloneqq\; -\frac{\big(-f_r(w)\big)^k}{k}.
\]
Write $h(w) = \sum_{n=0}^\infty a_n w^n$. Then, by equation \eqref{eq: Tv}, for every $w\in\mathbb D_1$,
\[
\sum_{n=1}^\infty a_n w^n
=
h(w)-h(0)
=
\Big( T_r v(w\,;\,c) \Big)_k
=
\sum_{n=1}^\infty b^{(r)}_{k,n}\left(-\frac{(-w)^n}{n}\right).
\]
Thus, by comparing the Taylor coefficients,
\[
a_n = b^{(r)}_{k,n}\frac{(-1)^{n+1}}{n}
\qquad(n\geq1).
\]
Now, using the definition of $T_{r,m}$, this identity, and Lemma \ref{lemma: coboundary property}, we obtain
\begin{equation*}
\Big( T_{r,m}v(x\,;\,c) \Big)_k
=
\sum_{n=1}^{m-1} b^{(r)}_{k,n}\left(-\frac{(-x)^n}{n}\right)
=
\sum_{n=1}^{m-1} a_n x^n
=
\frac{1}{2\pi i} \oint_{\partial \mathbb D_\rho} \frac{1-z^{m-1}}{1-z} h\left(\frac{x}{z}\right)dz. \qedhere
\end{equation*}
\end{proof}

\subsection{Integral representations for iterates} \label{subsection: Integral representations for iterates}

The integral representations above extend naturally to iterates. For $r \in C$ and $z \in \partial \mathbb{D}_\rho$, define
\[
f_{r,z}: \mathbb{D}_\rho \to \mathbb{D}_\rho,
\qquad
f_{r,z}(x) = f_r\!\left( \frac{x}{z} \right)
\qquad \big( x \in \mathbb{D}_\rho \big).
\]
This is well-defined because $|x|<\rho=|z|$ implies $x/z \in \mathbb{D}_1$, and hence
$f_r(x/z) \in \mathbb{D}_\rho$ by Lemma \ref{lemma: real contraction implies complex contraction}.

For $\mathbf r = (r_0, r_1, \ldots, r_n) \in C^{n+1}$ and $(z_0, \ldots, z_n) \in (\partial \mathbb{D}_\rho)^{n+1}$, define
\begin{equation} \label{eq: definition of mathscr F}
\mathscr{F}_{z_n, \ldots, z_0}^{r_n, \ldots, r_0}: \mathbb{D}_\rho \to \mathbb{D}_\rho,
\qquad
\mathscr{F}_{z_n, \ldots, z_0}^{r_n, \ldots, r_0}(x) = f_{r_n, z_n} \circ \cdots \circ f_{r_0, z_0} (x).
\end{equation}

\begin{proposition} \label{prop: integral representation for iterates}
Let $x \in \mathbb{D}_\rho$, $c \in \mathbb{C}$, $n \geq 0$, $m\geq 2$, and $(r_0, r_1, \ldots, r_n) \in C^{n+1}$. Under the notations above, we have for any $k \geq 1$,
\begin{equation} \label{eq: integral representation of power of T}
\bigg( T_{r_n} \cdots T_{r_0} \, v(x \, ; \, c) \bigg)_k
=
- \frac{1}{ \big( 2\pi i \big)^{n+1} } \oint_{\partial \mathbb D_\rho} \cdots \oint_{\partial \mathbb D_\rho} \left( \prod_{\ell = 0}^n \frac{ 1 }{ 1 - z_\ell } \right) \frac{ \left( - \mathscr{F}_{z_n, \ldots, z_0}^{r_n, \ldots, r_0}(x) \right)^k }{k} dz_n \cdots dz_0,
\end{equation}
and for any $1 \leq k \leq m-1$,
\begin{equation} \label{eq: integral representation of power of Tm}
\bigg( T_{r_n,m} \cdots T_{r_0,m} \, v(x \, ; \, c) \!\!\bigg)_k
\!\!=
- \frac{1}{ \big( 2\pi i \big)^{n+1} } \oint_{\partial \mathbb D_\rho} \!\!\!\! \cdots \oint_{\partial \mathbb D_\rho} \!\left( \prod_{\ell = 0}^n \frac{ 1 - {z_\ell}^{m-1} }{ 1 - z_\ell }\!\! \right) \frac{ \left( - \mathscr{F}_{z_n, \ldots, z_0}^{r_n, \ldots, r_0}(x) \right)^k }{k} dz_n \cdots dz_0.
\end{equation}
Here, all contour integrals are taken over the positively oriented circle $\partial \mathbb D_\rho$.
\end{proposition}

\begin{proof}
For $n=0$, both identities are exactly Proposition \ref{prop: integral representation}. Suppose $n\geq1$. First, equation \eqref{eq: integral representation of power of Tm} follows by induction on $n$ and Proposition \ref{prop: integral representation}, since we only exchange finite sums with integrals at each step.

Let us prove equation \eqref{eq: integral representation of power of T}. Fix $k\geq1$ and let
\[
h_k(w) \coloneqq -\frac{(-w)^k}{k}.
\]
Write $\mathbf r=(r_0,\ldots,r_n)$, and recall that $f_{\mathbf r}^{(n+1)}\coloneqq f_{r_n}\circ\cdots\circ f_{r_0}$. By applying equation \eqref{eq: Tv} iteratively, the subtracted terms cancel at each step, and we obtain
\[
\Big( T_{r_n}\cdots T_{r_0}v(x\,;\,c) \Big)_k
=
h_k\big( f_{\mathbf r}^{(n+1)}(x) \big)
-
h_k\big( f_{\mathbf r}^{(n+1)}(0) \big).
\]

Next, we evaluate the right-hand side of equation \eqref{eq: integral representation of power of T}. Since each $f_{r,z}$ maps $\mathbb D_\rho$ into itself, we have
\[
\left| h_k\left( \mathscr F_{z_n,\ldots,z_0}^{r_n,\ldots,r_0}(x) \right) \right|
\leq
\frac{\rho^k}{k}.
\]
Thus, all contour integrals below are absolutely convergent. Applying Lemma \ref{lemma: coboundary property} to $h=h_k\circ f_{r_n}$, we have
\begin{equation*}
\frac{1}{2\pi i} \oint_{\partial \mathbb D_\rho} \frac{1}{1-z_n}
h_k\circ f_{r_n}\left( \frac{\mathscr F_{z_{n-1},\ldots,z_0}^{r_{n-1},\ldots,r_0}(x)}{z_n} \right)dz_n
=
h_k\circ f_{r_n}\left( \mathscr F_{z_{n-1},\ldots,z_0}^{r_{n-1},\ldots,r_0}(x) \right)
-
h_k\circ f_{r_n}(0).
\end{equation*}
The second term is independent of the remaining contour variables, $z_0,\ldots,z_{n-1}$, and vanishes after the next contour integration, since
\[
\oint_{\partial \mathbb D_\rho} \frac{dz}{1-z}=0.
\]
We continue to apply Lemma \ref{lemma: coboundary property} in the same way, in the order $z_{n-1}, \ldots, z_0$. After the final integration in $z_0$, we obtain
\[
h_k\big( f_{\mathbf r}^{(n+1)}(x) \big)
-
h_k\big( f_{\mathbf r}^{(n+1)}(0) \big),
\]
which is exactly $\big( T_{r_n}\cdots T_{r_0}v(x\,;\,c) \big)_k$. This completes the proof.
\end{proof}

\subsection{Integral operators and telescoping} \label{subsection: Integral operators and telescoping}

We now express the operator difference $\mathrm T^\ell \boldsymbol v - \mathrm T_{m\times m}^\ell \boldsymbol v$ in terms of integral operators on holomorphic functions, and estimate the resulting sum by telescoping.

Denote by $\pazocal{H}(D)$ the set of holomorphic functions on a domain $D$. For $r \in C$ and $m \in \mathbb N$, define
$
\pazocal{L}_r,\; \pazocal{R}_{r,m},\; \pazocal{K}_{r,m} : \;
\pazocal{H}(\mathbb{D}_\rho) \to \pazocal{H}(\mathbb{D}_\rho)
$
by
\begin{align*}
\pazocal{L}_r h(w)
&\;\coloneqq\;
\frac{1}{2\pi i} \oint_{\partial \mathbb D_\rho} \frac{1}{1-z}\, h\left( f_r\!\left( \frac{w}{z} \right) \right)\,dz, \\[4pt]
\pazocal{R}_{r,m} h(w)
&\;\coloneqq\;
\frac{1}{2\pi i} \oint_{\partial \mathbb D_\rho} \frac{z^{m-1}}{1-z}\, h\left( f_r\!\left( \frac{w}{z} \right) \right)\,dz, \\[4pt]
\pazocal{K}_{r,m} h
&\;\coloneqq\;
\pazocal{L}_r h - \pazocal{R}_{r,m} h.
\end{align*}
These return holomorphic functions on $\mathbb{D}_\rho$. Indeed, if $w \in \mathbb D_\rho$ and $z \in \partial \mathbb D_\rho$, then
$w/z \in \mathbb D_1$, and hence $f_r(w/z) \in \mathbb D_\rho$ by Lemma \ref{lemma: real contraction implies complex contraction}. For each fixed $z \in \partial \mathbb D_\rho$, the map
\[
w \;\;\longmapsto\;\; \frac{1}{1-z}\, h\left( f_r\!\left( \frac{w}{z} \right) \right)
\]
is holomorphic on $\mathbb D_\rho$. If $\gamma$ is any closed piecewise $C^1$-curve inside a compact
set $K \subset \mathbb D_\rho$, then using the boundedness of $h$ on the compact set $\left\{ f_r(w/z) \setcond w \in K,\; z \in \partial \mathbb D_\rho \right\}$ and Fubini's theorem,
\[
\oint_\gamma \pazocal{L}_r h(w)\,dw
=
\frac{1}{2\pi i} \oint_{\partial \mathbb D_\rho} \oint_\gamma \frac{1}{1-z}\, h\left( f_r\!\left( \frac{w}{z} \right) \right)\,dw\,dz = 0.
\]
Thus Morera's theorem implies that $\pazocal{L}_r h$ is holomorphic on $\mathbb D_\rho$. The same argument applies to $\pazocal{R}_{r,m}$, and hence also to $\pazocal{K}_{r,m}$.

For any $h \in \pazocal{H}(\mathbb D_\rho)$, applying Lemma \ref{lemma: coboundary property} to $h\circ f_r$ we obtain
\begin{equation} \label{eq: operator L identity}
\pazocal{L}_r h(w) = h\big( f_r(w) \big) - h\big( f_r(0) \big).
\end{equation}

We define the oscillation seminorm $\| \cdot \|_{\mathrm{osc}}$ on $\pazocal{H}(\mathbb D_\rho)$ by
\[
\| h \|_{\mathrm{osc}}
\;\coloneqq\;
\sup_{u,v \in \mathbb D_\rho} |h(u)-h(v)|.
\]
Also, define a positive number $K(\rho)$ by
\begin{equation} \label{eq: definition of K-rho}
K(\rho)
\;\coloneqq\;
\frac{1}{2\pi} \oint_{\partial \mathbb D_\rho} \frac{|dz|}{|1-z|}.
\end{equation}

\begin{lemma} \label{lemma: basic properties of LRK}
Suppose $h \in \pazocal{H}(\mathbb D_\rho)$ satisfies $\|h\|_{\mathrm{osc}} < \infty$. Then for any
$r \in C$ and $m \geq 1$,
\[
\| \pazocal{L}_r h \|_{\mathrm{osc}} \leq \|h\|_{\mathrm{osc}},
\quad
\| \pazocal{R}_{r,m} h \|_{\mathrm{osc}} \leq \rho^{m-1} K(\rho)\,\|h\|_{\mathrm{osc}},
\quad
\| \pazocal{K}_{r,m} h \|_{\mathrm{osc}} \leq \big( 1+\rho^{m-1}K(\rho) \big)\|h\|_{\mathrm{osc}}.
\]
Also,
\begin{equation} \label{eq: LRKh0 is 0}
\pazocal{L}_r h(0)
=
\pazocal{R}_{r,m} h(0)
=
\pazocal{K}_{r,m} h(0)
=
0.
\end{equation}
\end{lemma}

\begin{proof}
The first inequality follows immediately from equation \eqref{eq: operator L identity}. For any $u,v \in \mathbb D_\rho$,
\begin{align*}
\left| \pazocal{R}_{r,m} h(u) - \pazocal{R}_{r,m} h(v) \right|
&=
\left| \frac{1}{2\pi i} \oint_{\partial \mathbb D_\rho} \frac{z^{m-1}}{1-z}
\Big( h\big( f_r(u/z) \big) - h\big( f_r(v/z) \big) \Big)\,dz \right| \\
&\leq
\frac{\|h\|_{\mathrm{osc}}}{2\pi} \oint_{\partial \mathbb D_\rho} \frac{\rho^{m-1}|dz|}{|1-z|}
=
\rho^{m-1}K(\rho)\,\|h\|_{\mathrm{osc}}.
\end{align*}
The estimate for $\pazocal{K}_{r,m}$ follows from the triangle inequality. Finally, equation \eqref{eq: LRKh0 is 0} is an immediate consequence of Cauchy's integral
theorem.
\end{proof}

The next lemma rewrites the difference between the products of the full operator $\pazocal L_r$ and the truncated operator $\pazocal K_{r,m}$ as a sum of single-insertions of the remainder operator $\pazocal R_{r,m}$.

\begin{lemma} \label{lemma: telescoping of L and K}
Fix $m \in \mathbb N$, and write $\pazocal{K}_r = \pazocal{K}_{r,m}$ and $\pazocal{R}_r = \pazocal{R}_{r,m}\,$ for $r\in C$. Then for any $(r_1,\ldots,r_n) \in C^n$,
\[
\pazocal{L}_{r_1}\cdots \pazocal{L}_{r_n} - \pazocal{K}_{r_1}\cdots \pazocal{K}_{r_n}
\;=\;
\sum_{\ell=1}^n \pazocal{L}_{r_1}\cdots \pazocal{L}_{r_{\ell-1}} \pazocal{R}_{r_\ell} \pazocal{K}_{r_{\ell+1}}\cdots \pazocal{K}_{r_n}.
\]
\end{lemma}

\begin{proof}
We argue by induction on $n$. For $n=1$, the identity is simply
\[
\pazocal{L}_{r_1} - \pazocal{K}_{r_1} = \pazocal{R}_{r_1}.
\]
Assume the statement holds for words of length $n$. Then
\begin{align*}
\pazocal{L}_{r_1} \cdots \pazocal{L}_{r_{n+1}} - \pazocal{K}_{r_1} \cdots \pazocal{K}_{r_{n+1}}
\;&=\; \pazocal{L}_{r_1} \left( \pazocal{L}_{r_2} \cdots \pazocal{L}_{r_{n+1}} - \pazocal{K}_{r_2} \cdots \pazocal{K}_{r_{n+1}} \right) + \left( \pazocal{L}_{r_1} - \pazocal{K}_{r_1} \right) \pazocal{K}_{r_2} \cdots \pazocal{K}_{r_{n+1}}
\\
&=\; \pazocal{L}_{r_1} \left( \sum_{\ell = 1}^n \pazocal{L}_{r_2} \cdots \pazocal{L}_{r_{\ell}} \pazocal{R}_{r_{\ell+1}} \pazocal{K}_{r_{\ell+2}} \cdots \pazocal{K}_{r_{n+1}} \right) + \pazocal{R}_{r_1} \pazocal{K}_{r_2} \cdots \pazocal{K}_{r_{n+1}} \\
&=\; \left( \sum_{\ell = 1}^n \pazocal{L}_{r_1} \cdots \pazocal{L}_{r_{\ell}} \pazocal{R}_{r_{\ell+1}} \pazocal{K}_{r_{\ell+2}} \cdots \pazocal{K}_{r_{n+1}} \right) + \pazocal{R}_{r_1} \pazocal{K}_{r_2} \cdots \pazocal{K}_{r_{n+1}} \\
&=\; \sum_{\ell = 1}^{n+1} \pazocal{L}_{r_1} \cdots \pazocal{L}_{r_{\ell-1}} \pazocal{R}_{r_{\ell}} \pazocal{K}_{r_{\ell+1}} \cdots \pazocal{K}_{r_{n+1}}. \qedhere
\end{align*}
\end{proof}

We are now ready to estimate the truncation error.

\begin{proposition} \label{prop: evaluating the truncation}
Let $D = \max_{r \in C} \left| f_r^\top(0) \right|$ and $\rho_m = \rho^{m-1} K(\rho)$. For any $N,m \in \mathbb N$ with $m \geq 2$,
\begin{equation*}
\left| \sum_{n=0}^{N-1} \Big[ \mathrm T^n \boldsymbol v - \mathrm T_{m\times m}^n \boldsymbol v \Big]_{\boldsymbol 0} \right|
\leq
2 \left( \!\log \frac{1}{1-\rho D} \right)
\left( \frac{(1+\rho_m)^{N-1} - 1 - (N-1)\rho_m}{\rho_m} \right)
+\frac{2(N-1)\rho^m D^m}{m(1-\rho D)}.
\end{equation*}
\end{proposition}

\begin{proof}
For each $k \geq 1$, define $h^{(k)} \in \pazocal H(\mathbb D_\rho)$ by $h^{(k)}(w) = - \frac{(-w)^k}{k}$. Then $\| h^{(k)} \|_{\mathrm{osc}} \leq \frac{2\rho^k}{k}$.

For $n \geq 0$, $s \in C$, and with the convention that the empty product is the identity,
\begin{align}
\big( \mathrm T^n \boldsymbol v \big)_s
&=
\frac{1}{\pi_s} \sum_{\mathbf r = (r_0,\ldots,r_n) \in C^{n+1}} p(\mathbf r)\,Q_{r_n,s}\, T_{r_n}\cdots T_{r_1} v\big( f_{r_0}(0)\,;\,\ell_{r_0}(0) \big), \label{eq: path expansion for Tn v} \\
\big( \mathrm T_{m\times m}^n \boldsymbol v \big)_s
&=
\frac{1}{\pi_s} \sum_{\mathbf r = (r_0,\ldots,r_n) \in C^{n+1}} p(\mathbf r)\,Q_{r_n,s}\, T_{r_n,m}\cdots T_{r_1,m} v\big( f_{r_0}(0)\,;\,\ell_{r_0}(0) \big). \label{eq: path expansion for Tmn v}
\end{align}
These identities follow by induction on $n$ from the definitions of $\boldsymbol v$, $\mathrm T$,
and $\mathrm T_{m \times m}$.

Now fix $n \geq 1$, $s \in C$, and $\mathbf r = (r_0,\ldots,r_n)\in C^{n+1}$. By Proposition \ref{prop: integral representation for iterates}, for every $k\geq 1$,
\begin{equation} \label{eq: operator formula for full iterate}
\left(
T_{r_n}\cdots T_{r_1} v\big( f_{r_0}(0)\,;\,\ell_{r_0}(0) \big)
\right)_k
=
\big( \pazocal{L}_{r_1}\cdots \pazocal{L}_{r_n} h^{(k)} \big)\big( f_{r_0}(0) \big),
\end{equation}
and for every $1 \leq k \leq m-1$,
\begin{equation} \label{eq: operator formula for truncated iterate}
\left(
T_{r_n,m}\cdots T_{r_1,m} v\big( f_{r_0}(0)\,;\,\ell_{r_0}(0) \big)
\right)_k
=
\big( \pazocal{K}_{r_1,m}\cdots \pazocal{K}_{r_n,m} h^{(k)} \big)\big( f_{r_0}(0) \big).
\end{equation}

For $1\leq\ell\leq n$, let
\[
\pazocal S
=
\pazocal{L}_{r_1}\cdots \pazocal{L}_{r_{\ell-1}}
\pazocal{R}_{r_\ell,m}
\pazocal{K}_{r_{\ell+1},m}\cdots \pazocal{K}_{r_n,m}.
\]
Since each of the operators $\pazocal{L}_{r}$, $\pazocal{R}_{r,m}$, and $\pazocal{K}_{r,m}$ sends
every holomorphic function to a new function vanishing at $0$ by equation \eqref{eq: LRKh0 is 0}, we have $\pazocal Sg(0)=0$ for every $g \in \pazocal{H}(\mathbb D_\rho)$. Hence,
\[
|\pazocal Sg(w)|
=
|\pazocal Sg(w)- \pazocal Sg(0)|
\leq
\|\pazocal Sg\|_{\mathrm{osc}}
\qquad \text{for every } w \in \mathbb D_\rho.
\]

Therefore, by Lemma \ref{lemma: basic properties of LRK} and Lemma \ref{lemma: telescoping of L and K}, for every $\mathbf r=(r_0,\ldots,r_n)\in C^{n+1}$ and $w\in \mathbb D_\rho$,
\begin{align}
&\left|
\big( \pazocal{L}_{r_1}\cdots \pazocal{L}_{r_n} h^{(k)} \big)(w)
-
\big( \pazocal{K}_{r_1,m}\cdots \pazocal{K}_{r_n,m} h^{(k)} \big)(w)
\right| \nonumber \\
&\leq
\left\| \big( \pazocal{L}_{r_1}\cdots \pazocal{L}_{r_n}
-
\pazocal{K}_{r_1,m}\cdots \pazocal{K}_{r_n,m} \big) h^{(k)} \right\|_{\mathrm{osc}}
\leq
\sum_{\ell=1}^{n}
\left\|
\pazocal{L}_{r_1}\cdots \pazocal{L}_{r_{\ell-1}}
\pazocal{R}_{r_\ell,m}
\pazocal{K}_{r_{\ell+1},m}\cdots \pazocal{K}_{r_n,m}
\, h^{(k)}
\right\|_{\mathrm{osc}} \nonumber \\
&\leq
\sum_{\ell=1}^{n} 1^{\ell-1} \, \rho_m (1+\rho_m)^{n-\ell}\, \|h^{(k)}\|_{\mathrm{osc}}
\leq
\frac{2\rho^k}{k}\Big( (1+\rho_m)^n - 1 \Big). \label{eq: bound for the difference of L and K compositions}
\end{align}
Combining this with equations \eqref{eq: path expansion for Tn v}, \eqref{eq: path expansion for Tmn v}, \eqref{eq: operator formula for full iterate}, \eqref{eq: operator formula for truncated iterate}, and using
\[
\frac{1}{\pi_s} \sum_{\mathbf r = (r_0,\ldots,r_n)\in C^{n+1}} p(\mathbf r)\,Q_{r_n,s} = 1,
\]
we obtain the following for every $s \in C$, $n \geq 1$, and $1 \leq k \leq m-1$.
\begin{equation} \label{eq: kth coordinate bound for difference}
\left| \Big(
\mathrm T^n \boldsymbol v - \mathrm T_{m\times m}^n \boldsymbol v
\Big)_{\!\!s,\,k} \right|
\leq
\frac{2\rho^k}{k}\Big( (1+\rho_m)^n - 1 \Big).
\end{equation}

Similarly, by equations \eqref{eq: operator formula for full iterate}, \eqref{eq: LRKh0 is 0}, and Lemma \ref{lemma: basic properties of LRK},
\[
\left| \big( \pazocal{L}_{r_1}\cdots \pazocal{L}_{r_n} h^{(k)} \big)\big( f_{r_0}(0) \big) \right|
\leq
\|h^{(k)}\|_{\mathrm{osc}}
\leq
\frac{2\rho^k}{k}.
\]
Using equation \eqref{eq: path expansion for Tn v} again, we get
\begin{equation} \label{eq: kth coordinate bound for full iterate}
\left| \big( \mathrm T^n \boldsymbol v \big)_{\!\!s,\,k} \right|
\leq
\frac{2\rho^k}{k}
\qquad
\text{for every } n \geq 0,\;\; k \geq 1.
\end{equation}

We now estimate the $0$-th coordinate. For $n \geq 1$ and $s \in C$, by the definitions of $\mathrm T$, $\mathrm T_{m\times m}$, $T_r$, and $T_{r,m}$,
\begin{flalign}
& \Big(
\mathrm T^n \boldsymbol v - \mathrm T_{m\times m}^n \boldsymbol v
\Big)_{\!\!s,0} & \nonumber \\
&=
\sum_{r\in C}
\frac{\pi_rQ_{r,s}}{\pi_s} \left\{ \sum_{k=1}^{m-1} \big( f_r^\top(0) \big)^k
\Big(
\mathrm T^{n-1} \boldsymbol v - \mathrm T_{m\times m}^{n-1} \boldsymbol v
\Big)_{\!\!r,\,k}
+
\sum_{k=m}^{\infty} \big( f_r^\top(0) \big)^k
\Big( \mathrm T^{n-1}\boldsymbol v \Big)_{\!\!r,\,k}
\right\}. \label{eq: zero coordinate recursion} \raisetag{25pt} &
\end{flalign}
Hence, using equations \eqref{eq: kth coordinate bound for difference}, \eqref{eq: kth coordinate bound for full iterate}, and the definition of $D$,
\begin{align*}
&\left| \sum_{n=0}^{N-1} \Big[ \mathrm T^n \boldsymbol v - \mathrm T_{m\times m}^n \boldsymbol v \Big]_{\boldsymbol 0} \right|
\leq
\sum_{n=1}^{N-1} \sum_{s\in C}\pi_s
\left| \Big(
\mathrm T^n \boldsymbol v - \mathrm T_{m\times m}^n \boldsymbol v
\Big)_{\!\!s,\,0} \right| \\
&\leq
\sum_{n=1}^{N-1} \sum_{s\in C}\pi_s \sum_{r\in C}
\frac{\pi_rQ_{r,s}}{\pi_s} \left\{ \sum_{k=1}^{m-1}
D^k \left| \Big(
\mathrm T^{n-1} \boldsymbol v - \mathrm T_{m\times m}^{n-1} \boldsymbol v
\Big)_{\!\!r,\,k} \right|
+
\sum_{k=m}^{\infty}
D^k \left| \Big( \mathrm T^{n-1}\boldsymbol v \Big)_{\!\!r,\,k} \right|
\right\} \\
&=
\sum_{n=1}^{N-1} \sum_{r\in C}\pi_r \left( \sum_{s\in C} Q_{r,s} \right)
\left\{
\sum_{k=1}^{m-1}
D^k \left| \Big(
\mathrm T^{n-1} \boldsymbol v - \mathrm T_{m\times m}^{n-1} \boldsymbol v
\Big)_{\!\!r,\,k} \right|
+
\sum_{k=m}^{\infty}
D^k \left| \Big( \mathrm T^{n-1}\boldsymbol v \Big)_{\!\!r,\,k} \right|
\right\} \\
&\leq
\sum_{n=2}^{N-1} \sum_{k=1}^{m-1}
\frac{2\rho^k D^k}{k} \Big( (1+\rho_m)^{n-1} - 1 \Big)
+
(N-1)\sum_{k=m}^{\infty} \frac{2\rho^k D^k}{k} \\
&\leq
2\left( \log \frac{1}{1-\rho D} \right) \frac{(1+\rho_m)^{N-1} - 1 - (N-1)\rho_m}{\rho_m}
+
2(N-1)\frac{\rho^m D^m}{m(1-\rho D)}. \qedhere
\end{align*}
\end{proof}

The following bound for $K(\rho)$ follows from an elementary calculus argument using Cauchy–Schwarz and elliptic integrals.
\begin{lemma}[{\cite[Lemma 4.9]{Alibabaei26-2}}] \label{lemma: bound for Kr}
We have
\[
K(\rho) \leq \min \left\{ \frac{\rho}{\sqrt{1-\rho^2}}, \hspace{8pt} \frac{2\rho}{\pi(1+\rho)} \left( \frac{\pi}{2} + \log \frac{1+\rho}{1-\rho} \right) \right\}.
\]
\end{lemma}

The following theorem gives an explicit error bound. We recall that $d$ is the period of the recurrent-class transition matrix before aperiodic reduction.

\begin{theorem} \label{thm: precise bounds}
Let $0 < \rho < 1$ be as in equation \eqref{eq: every f is a rho-contraction}, and define
\[ 
E_C
=
\frac{1}{1-\rho}
\left(
\max_{r\in C} \frac{\left| f_r^{\top}(0) \right|}{1+\sqrt{1-{\left| f_r^{\top}(0) \right|}^2}}
\right)
\left(
\sum_{r\in C}\pi_r\, d_{\mathrm{hyp}}\big( f_r(0),0 \big)
\right).
\]
Let $D = \max_{r \in C} \left| f_r^{\top}(0) \right|$. Then $0 \leq D < 1$, and for any $n, m \in \mathbb N$ with $m \geq 2$,
\begin{flalign*}
&\left|
\lambda_+(A,P)
-
\sum_{\ell=0}^{n-1} \Big[ \mathrm T_{m\times m}^\ell \widehat{\boldsymbol v} \Big]_{\boldsymbol 0}
\right| & \\
& \hspace{40pt} \leq
\frac{E_C}{d} \rho^{n-1}
+
\frac2d \left( \log \frac{1}{1-\rho D} \right)
\left( \frac{(1+\rho_m)^{n-1} - 1 - (n-1)\rho_m}{\rho_m} \right)
+
\frac{2(n-1)\rho^m D^m}{md(1-\rho D)}. &
\end{flalign*}
Here, $\rho_m = \rho^{m-1} K(\rho)$, and
\[
0 < K(\rho) \leq \min \left\{ \frac{\rho}{\sqrt{1-\rho^2}}, \hspace{8pt} \frac{2\rho}{\pi(1+\rho)} \left( \frac{\pi}{2} + \log \frac{1+\rho}{1-\rho} \right) \right\}.
\]
\end{theorem}

\begin{proof}
By equation \eqref{eq: bounding the transpose of fr}, we have $f_r^\top([-1,1]) \subset (-1,1)$ for every $r\in C$. Hence $0 \leq D < 1$.

Also, $\lambda_+(A,P) = \sum_{\ell=0}^{\infty} \Big[ \mathrm T^\ell \widehat{\boldsymbol v} \Big]_{\boldsymbol 0}$ by Theorem \ref{thm: infinite sum equals lambda}. Therefore, by the triangle inequality,
\begin{align*}
\left|
\lambda_+(A,P)
-
\sum_{\ell=0}^{n-1} \Big[ \mathrm T_{m\times m}^\ell \widehat{\boldsymbol v} \Big]_{\boldsymbol 0}
\right|
&=
\left|
\sum_{\ell=0}^{\infty} \Big[ \mathrm T^\ell \widehat{\boldsymbol v} \Big]_{\boldsymbol 0}
-
\sum_{\ell=0}^{n-1} \Big[ \mathrm T_{m\times m}^\ell \widehat{\boldsymbol v} \Big]_{\boldsymbol 0}
\right| \\
&\leq
\frac1d \left|
\sum_{\ell=0}^{\infty} \Big[ \mathrm T^\ell \boldsymbol v \Big]_{\boldsymbol 0}
-
\sum_{\ell=0}^{n-1} \Big[ \mathrm T^\ell \boldsymbol v \Big]_{\boldsymbol 0}
\right|
+
\frac1d \left|
\sum_{\ell=0}^{n-1} \Big[ \mathrm T^\ell \boldsymbol v \Big]_{\boldsymbol 0}
-
\sum_{\ell=0}^{n-1} \Big[ \mathrm T_{m\times m}^\ell \boldsymbol v \Big]_{\boldsymbol 0}
\right|.
\end{align*}

The inequality follows by applying Proposition \ref{prop: cutting the tail of the Neumann series} to the first term, and Proposition \ref{prop: evaluating the truncation} to the second term. The bound for $K(\rho)$ is by Lemma \ref{lemma: bound for Kr}.
\end{proof}

\begin{theorem} \label{thm: complexity statement}
Under the assumptions of Theorem \ref{thm: precise bounds}, for any $\vep > 0$ there exist
natural numbers $n$ and $m$ such that $n = O \! \left( \log (1/\vep) \right)$, $m = O \! \left( \log (n/\vep) \right)$, and
\[
\left|
\lambda_+(A,P)
-
\sum_{\ell=0}^{n-1} \Big[ \mathrm T_{m\times m}^{\ell} \widehat{\boldsymbol v} \Big]_{\boldsymbol 0}
\right|
< \vep.
\]
In particular, $\lambda_+(A,P)$ can be approximated with $O \! \left( \big( \log ( 1/\vep ) \big)^3 \right)$ arithmetic operations.
\end{theorem}

\begin{proof}
Recall the notations of Theorem \ref{thm: precise bounds}. Let $\vep > 0$ be given. Since $0 < \rho < 1$, we can choose a natural number
$n = O\!\left( \log(1/\vep) \right)$ such that
\[
\frac{E_C}{d} \rho^{n-1} < \frac{\vep}{2}.
\]

Next, we can take a natural number
$m = O\!\left( \log(n/\vep) \right)$ such that
\[
(n-1)\rho_m \leq 1,
\quad \text{and} \quad
e \cdot \left( \log \frac{1}{1-\rho D} \right) (n-1)^2 \rho_m
+
\frac{2(n-1)(\rho D)^m}{m(1-\rho D)}
<\frac{\vep}{2} \cdot d.
\]

Since $(n-1)\rho_m \leq 1$, for any $t \geq 0$ we have
$(1+t)^{n-1} \leq e^{(n-1)t}$ and
$e^t - 1 - t \leq \frac12 t^2 e^t$, and hence
\begin{align*}
\frac{(1+\rho_m)^{n-1} - 1 - (n-1)\rho_m}{\rho_m}
&\leq
\frac{e^{(n-1)\rho_m} - 1 - (n-1)\rho_m}{\rho_m} \\
&\leq
\frac{(n-1)^2 \rho_m^2 e^{(n-1)\rho_m}}{2\rho_m}
\leq
\frac{e}{2}(n-1)^2 \rho_m.
\end{align*}

Applying Theorem \ref{thm: precise bounds}, we obtain
\begin{align*}
&\left|
\lambda_+(A,P)
-
\sum_{\ell=0}^{n-1} \Big[ \mathrm T_{m\times m}^{\ell} \widehat{\boldsymbol v} \Big]_{\boldsymbol 0}
\right| \\
&\hspace{30pt}\leq
\frac{E_C}{d} \rho^{n-1}
+
\frac{2}{d} \left( \log \frac{1}{1-\rho D} \right)
\left(
\frac{(1+\rho_m)^{n-1} - 1 - (n-1)\rho_m}{\rho_m}
\right)
+
\frac{2(n-1)\rho^m D^m}{md(1-\rho D)} \\
&\hspace{30pt}\leq
\frac{E_C}{d} \rho^{n-1}
+
\frac{e}{d} \left( \log \frac{1}{1-\rho D} \right) (n-1)^2 \rho_m
+
\frac{2(n-1)(\rho D)^m}{md(1-\rho D)}
<
\vep.
\end{align*}
This proves the approximation statement.

\smallskip

It remains to estimate the computational cost. The non-zero entries of $\mathrm T_{m\times m}$ can be assembled in
\[
O\!\left( \sum_{r,r' \in C} \sum_{k=0}^{m-1} \sum_{n=0}^{m-1} \min\{k,n\} \right)
=
O\!\left( m^3 \right)
\]
arithmetic operations.

Next, repeated matrix-vector multiplication costs $O(m^2)$ arithmetic operations per step. Therefore, computing the vectors
\[
\widehat{\boldsymbol v},\;\;
\mathrm T_{m\times m} \widehat{\boldsymbol v},\;\;
\mathrm T_{m\times m}^2 \widehat{\boldsymbol v},\;\;
\dots\;,\;\;
\mathrm T_{m\times m}^{n-1} \widehat{\boldsymbol v},
\]
costs $O\!\left( nm^2 \right)$ arithmetic operations. Thus the total complexity is
\[
O\!\left( m^3 \right) + O\!\left( nm^2 \right)
=
O\!\left( \big( \log(1/\vep) \big)^3 \right). \qedhere
\]
\end{proof}

This completes the proof of Theorem \ref{thm: polynomial computation of the Lyapunov exponent in introduction} from the introduction.

\section{Analytic dependence on the underlying data} \label{section: Analytic dependence on the underlying data}

This section proves real analyticity of the Lyapunov exponent by considering the holomorphic extensions of the relevant data. The explicit domain on which the extensions are holomorphic is given in Remark \ref{rmk: condition on domain of holomorphic extension}.

\begin{theorem} \label{thm: analytic dependence on parameter}
Let $\Sigma \subset X^{\mathbb N_0}$ be a transitive topological Markov shift, and $J \subset \mathbb{R}$ be an open interval. Assume the following:
\begin{enumerate}
\item For each $t\in J$, $\pazocal P(t) = \big( P_{ij}(t) \big)_{i,j\in X}$ is a row-stochastic matrix whose support is the edge set of $\Sigma$, and $t \longmapsto P_{ij}(t)$ is real analytic for every $i,j\in X$.
\item For each $t\in J$, the map $\pazocal A(t): \Sigma \to \GL_2(\mathbb R)$ is locally constant: $\pazocal A(t)(x) = A_{x_0}(t)$. In addition, $t \longmapsto A_i(t)$ is real analytic for every $i\in X$.
\end{enumerate}
Then, if $\pazocal A(t_0)$, where $t_0\in J$, is projectively uniformly hyperbolic with respect to $\Sigma$, the map
\[
t \;\;\longmapsto\;\;
\lambda_+\big( \pazocal A(t), \pazocal P(t) \big)
\]
is real analytic around $t_0$.
\end{theorem}

\begin{proof}
We will write $\lambda(t) \coloneqq \lambda_+\big( \pazocal A(t), \pazocal P(t) \big)$. Fix a multicone $\{M_i\}_{i\in X}$ for $\pazocal A(t_0)$ with respect to $\Sigma$, and write it as a finite disjoint union of connected components,
\[
M_i = \bigsqcup_{a=1}^{m(i)} M_{i,a},
\qquad i\in X.
\]
Now, for each $i,a$ and $j$ with $P_{ij}(t_0)>0$, we define $\beta(i,a,j)$ as in Subsection \ref{subsection: Branch-state extension} to be the unique index satisfying
\[
\big[A_j(t_0)\big]\big( \overline{M_{i,a}} \big) \subset M_{j, \beta(i,a,j)}.
\]
By the continuity of each $A_j(t)$ and assumption (1), there is an open interval $J_0\subset J$ containing $t_0$ such that the following is true for every $t\in J_0$.
\begin{equation} \label{eq: condition that a multicone is valid for t}
\big[A_j(t)\big]\big( \overline{M_{i,a}} \big) \subset M_{j, \beta(i,a,j)}
\quad
\text{for every } i\in X, \; 1\leq a\leq m(i), \text{ and } j\in X \text{ with } P_{ij}(t) > 0.
\end{equation}

By assumption (1), the branch-state set
\[
R
=
\left\{
\big( (i,a),(j,b) \big)
\;\middle|\;
P_{ij}(t_0)>0
\text{ and }
b=\beta(i,a,j)
\right\}
\]
and a fixed recurrent class $C\subset R$ remain unchanged on $J_0$.

Recall that $\pazocal V$ denotes the non-negative cone. For $s=(i,a)$, let $L_s\in \GL_2(\mathbb R)$ be a local chart for $\overline{M_{i,a}}$, that is, $[L_s](\pazocal V) = \overline{M_{i,a}}$. For $r\in C$, consider
\[
B_r(t)
\;\coloneqq\;
\vep(r)\,L_{t(r)}^{-1} A_{\tau(r)}(t) L_{s(r)},
\]
where $\vep(r)\in \{\pm1\}$ is chosen so that $B_r(t_0)$ is a positive matrix. Now, using equation \eqref{eq: condition that a multicone is valid for t} and the same argument used to prove equation \eqref{eq: g maps non-negative cone to positive cone}, the matrix $L_{t(r)}^{-1} A_{\tau(r)}(t) L_{s(r)}$ is either entrywise positive or negative, for each $t\in J_0$. Then, by the continuity of $\pazocal A$ and connectedness of $J_0$,  $B_r(t)$ is a positive matrix for every $t\in J_0$.

\medskip

Now, the cyclic decomposition in the proof of Proposition \ref{prop: aperiodic reduction} and the period $d$ depend only on the support of $Q(t)|_C$, and the support is independent of $t$ by assumption (1). Therefore, we may perform aperiodic reduction simultaneously for every $t\in J_0$, and as before, we continue to use the same symbols. The resulting $Q(t)|_C$ is irreducible and aperiodic for every $t\in J_0$.

\medskip

Choose an open neighborhood $\Omega\subset \mathbb C$ of $t_0$ on which every entry of each $A_i(t)$ and $\pazocal P(t)$ extends holomorphically, and $\Omega \cap \mathbb R \subset J_0$. We continue to denote the holomorphic extensions by the same symbols: $A_i(z)$, $P_{ij}(z)$, $Q(z)|_C$, and $B_r(z)$.

Each entry of $Q(z)|_C$ is holomorphic on $\Omega$. Since $Q(t)|_C$ is row-stochastic for every real $t\in J_0\cap \Omega$, the identity theorem gives
\begin{equation} \label{eq: Qz is row stochastic}
\sum_{r'\in C} Q_{r,r'}(z)=1
\qquad
\text{for every } r\in C \text{ and every } z\in \Omega.
\end{equation}

\begin{claim} \label{claim: Mt0 is invertible}
Fix $r_*\in C$, and let $M(z)$ be the matrix obtained by replacing the $r_*$-th column of $I-Q(z)|_C$ with $(1,\dots,1)^\top$. Then, $M(t_0)$ is invertible.
\end{claim}
\begin{proof}[Proof of Claim \ref{claim: Mt0 is invertible}]
Let $\eta$ be a column vector such that $\eta^{\!\top}\,M(t_0) = 0$. This means that
\begin{equation} \label{eq: M stationary row vector condition}
\sum_{r\in C} \eta_r = 0,
\quad \text{and} \quad
\Big( \eta^{\!\top} \big(I-Q(t_0)|_C\big) \Big)_j = 0
\quad
\text{for every } j\neq r_*.
\end{equation}
Also, since $Q(t_0)|_C$ is row-stochastic,
$
\big(I-Q(t_0)|_C\big)\,\mathbf 1 = 0
$,
where $\mathbf 1 = (1,\dots,1)^\top$. Therefore
\[
\sum_{j\in C} \Big( \eta^{\!\top} \big(I-Q(t_0)|_C\big) \Big)_j
=
\eta^{\!\top} (I-Q(t_0)|_C)\mathbf 1
=
0.
\]
Combining with equation \eqref{eq: M stationary row vector condition}, this implies that the $r_*$-th component of $\eta^{\!\top} \big(I-Q(t_0)|_C\big)$ is also $0$. Thus it is a zero vector, and $\eta^{\!\top}\,Q(t_0)|_C = \eta^{\!\top}$. So $\eta^{\!\top}$ is a left eigenvector of $Q(t_0)|_C$ for eigenvalue $1$. Since $Q(t_0)|_C$ is irreducible and row-stochastic, the eigenspace for eigenvalue $1$ is spanned by the stationary probability vector (with positive entries) by Perron--Frobenius. Therefore $\eta=0$ by equation \eqref{eq: M stationary row vector condition}. This proves that $M(t_0)$ is invertible.
\end{proof}

After shrinking $\Omega$, we may assume $M(z)$ is invertible for every $z\in \Omega$. Define the column vector $\pi(z)=\big(\pi_r(z)\big)_{r\in C}$ by
\[
\pi(z)^\top M(z) = e_{r_*}^\top,
\]
where $e_{r_*}$ is the column vector whose $r_*$-th entry is $1$ and all others are $0$. Then $z \mapsto \pi(z)$ is holomorphic on $\Omega$ by the holomorphy of $z\mapsto M(z)$. For real $t\in J_0\cap \Omega$, it coincides with the stationary probability vector of $Q(t)|_C$. Since $\pi_r(t_0)>0$ for every $r\in C$, after shrinking $\Omega$, we may assume that $\pi_r(z) \ne 0$ for every $r\in C$ and $z\in \Omega$. By the construction and equation \eqref{eq: Qz is row stochastic},
\begin{equation} \label{eq: complex stationary probability equation}
\sum_{r\in C} \pi_r(z) Q_{r,r'}(z) = \pi_{r'}(z) \; \text{ for } \; r' \in C,
\quad \text{and} \quad
\sum_{r \in C} \pi_r(z) = 1
\end{equation}
for every $z\in \Omega$.

\bigskip

Recall the definition of $F$ in equation \eqref{eq: definition of F}. For each $r\in C$ and $z\in \Omega$, define the M\"obius transformation $f_r[z]$ by
\begin{equation*}
f_r[z](x)
\coloneqq
\frac{
\alpha_r(z)x + \beta_r(z)}{
\gamma_r(z)x + \delta_r(z)}
\quad \text{where} \quad
\begin{pmatrix}
\alpha_r(z) & \beta_r(z) \\
\gamma_r(z) & \delta_r(z)
\end{pmatrix}
\coloneqq
F\big( B_r(z) \big).
\end{equation*}
Here, and from now on, we use $f_r[z](x)$ to denote the value of the M\"obius map $f_r[z]$ evaluated at $x\in \mathbb C$. Since $B_r(t_0)$ is positive, there is a number $0<\rho<1$ such that
$ f_r[t_0]([-1,1]) \subset [-\rho,\rho]$
for every $r\in C$. By Lemma \ref{lemma: real contraction implies complex contraction},
$f_r[t_0]\big( \overline{\mathbb D_1} \big) \subset \closure{\mathbb D_\rho}$. Choose any $\bar{\rho}$ with $\rho<\bar{\rho}<1$. By the continuity of entries of $B_r(z)$, after shrinking $\Omega$ if necessary,
\begin{equation} \label{eq: analytic parameter complex image bound}
f_r[z]\big( \overline{\mathbb D_1} \big)
\subset
\mathbb D_{\bar \rho}
\qquad
\text{for every } r\in C \text{ and } z\in \Omega.
\end{equation}

For each $r\in C$, we have $\delta_r(t_0)>0$ by Lemma \ref{lemma: denominator is positive}. Also, $\left| \gamma_r(t_0) / \delta_r(t_0) \right| = \big| f_r^\top[t_0](0) \big| < 1$ by Lemma \ref{lemma: Mobius contraction implies transpose contraction}. Therefore, for every $x\in \overline{\mathbb D_1}$,
\[
\Re\big( \gamma_r(t_0)x+\delta_r(t_0) \big)
\geq
\delta_r(t_0)-|\gamma_r(t_0)|
=
\delta_r(t_0)\left(
1-\left| \frac{\gamma_r(t_0)}{\delta_r(t_0)} \right|
\right)
>0,
\]
where $\Re w$ is the real part of $w \in \mathbb C$. By continuity and compactness, after shrinking $\Omega$ once more we may assume that
\begin{equation} \label{eq: analytic parameter positive real part}
\Re\big( \gamma_r(z)x+\delta_r(z) \big) >0
\qquad
\text{for every } r\in C,\ z\in \Omega,\ x\in \closure{\mathbb D_1}.
\end{equation}
Then, we can define
\[
\ell_r: \Omega\times \closure{\mathbb D_1} \to \mathbb C,
\qquad
\ell_r(z,x)
=
\mathrm{Log}\big( \gamma_r(z)x+\delta_r(z) \big).
\]
This is holomorphic on $\Omega\times \mathbb D_1$.

Also, by equation \eqref{eq: analytic parameter complex image bound} and Lemma \ref{lemma: Mobius contraction implies transpose contraction},
\begin{equation} \label{eq: analytic parameter transpose bound}
\big| f_r^\top[z](0) \big|
=
\left| \frac{\gamma_r(z)}{\delta_r(z)} \right|
<1
\qquad
\text{for every } r\in C \text{ and } z\in \Omega.
\end{equation}

Since $Q(t_0)|_C$ is a non-negative stochastic matrix, by shrinking $\Omega$ if necessary we assume that
\begin{equation} \label{eq: analytic parameter summability}
\bar{\rho}\widetilde Q < 1
\quad \text{where} \quad
\widetilde Q
\coloneqq
\max_{r\in C}\sup_{z\in \Omega}
\sum_{r'\in C} \big| Q_{r,r'}(z) \big|.
\end{equation}

For each $z\in \Omega$, define $T_r(z)$, $\mathrm T(z)$ and $\boldsymbol v(z)$ by exactly the same formulae as in Section \ref{section: Explicit representation of the Lyapunov exponent}, using the data $Q(z)|_C$, $\pi(z)$, $f_r[z]$, $f_r^{\top}[z]$, and $\ell_r(z,\cdot)$. For each finite path $\mathbf r=(r_0,\dots,r_m)\in C^{m+1}$, $z\in \Omega$, and $0\leq k\leq m+1$, write
\begin{equation} \label{eq: definition of analytic parameter compositions}
p_z(\mathbf r) = \pi_{r_0}(z)\,Q_{r_0,r_1}(z)\cdots Q_{r_{m-1},r_m}(z),
\quad
f_{\mathbf r}^{(0)}[z] = \mathrm{id},
\quad
f_{\mathbf r}^{(k)}[z] = f_{r_{k-1}}[z]\circ\cdots\circ f_{r_0}[z].
\end{equation}
Let $n\geq 1$. Define $S_n(z)\in V_C$ by, for each $r \in C$,
\begin{equation} \label{eq: analytic parameter definition of Sn}
\big( S_n(z) \big)_r
=
\frac{1}{\pi_r(z)} \sum_{\mathbf r=(r_0,\dots,r_{n-1})\in C^n}
p_z(\mathbf r)\,Q_{r_{n-1},r}(z)\,v\left( f_{\mathbf r}^{(n)}[z](0) \, ; \, \ell_{r_{n-1}}\big( z, f_{\mathbf r}^{(n-1)}[z](0) \big) \right).
\end{equation}
Because of equation \eqref{eq: analytic parameter transpose bound}, the proof of Lemma \ref{lemma: coboundary identity for Tr} applies verbatim and gives
\begin{equation} \label{eq: analytic parameter coboundary identity}
T_r(z) v(x\,;\,c)
\;=\;
v\big( f_r[z](x)\,;\,\ell_r(z,x) \big) - v\big( f_r[z](0)\,;\,\ell_r(z,0) \big),
\end{equation}
for every $z\in \Omega$, $r\in C$, and generator $v(x\,;\,c) \in V$. Furthermore, using equation \eqref{eq: complex stationary probability equation} and repeating the proof of Proposition \ref{prop: Sn identity} yield
\begin{equation} \label{eq: analytic parameter Sn-identity}
\sum_{j=0}^{n-1} \mathrm T(z)^j \boldsymbol v(z) = S_n(z)
\qquad\text{for every } z\in \Omega \text{ and } n\geq 1.
\end{equation}

We extend the notation $[\cdot]_{\boldsymbol 0}$, and define
\[
[\boldsymbol u]_{\boldsymbol 0, \,z}
\;\coloneqq\;
\sum_{r\in C}\pi_r(z) (\boldsymbol u)_{r,0}
\quad \text{for} \quad
\boldsymbol u \in \ell^\infty(\mathbb N_0)^C.
\]
For $n\geq 0$, let
\[
a_n(z) \coloneqq \big[ \mathrm T(z)^n \boldsymbol v(z) \big]_{\boldsymbol 0, \,z}.
\]
Theorem \ref{thm: infinite sum equals lambda} applies for real $t\in J_0$, and thus
\begin{equation} \label{eq: analytic parameter real series formula}
\lambda(t)
=
\frac1d \sum_{n=0}^{\infty} a_n(t)
\qquad\text{for every } t\in J_0\cap\Omega.
\end{equation}
It remains to prove that the series $\sum_{n=0}^\infty a_n(z)$ converges compactly on $\Omega$.

For $n\geq 1$, by equation \eqref{eq: analytic parameter Sn-identity},
\[
a_n(z)
=
\big[ S_{n+1}(z) \big]_{\boldsymbol 0, \,z} - \big[ S_n(z) \big]_{\boldsymbol 0, \,z}.
\]
By the definition of $S_{n+1}(z)$ and equation \eqref{eq: Qz is row stochastic}, we obtain
\[
\big[ S_{n+1}(z) \big]_{\boldsymbol 0, \,z}
=
\sum_{r \in C} \pi_r(z) \big( S_{n+1}(z) \big)_{r,0}
=
\sum_{\mathbf r=(r_0,\dots,r_n)\in C^{n+1}}
p_z(\mathbf r) \, \ell_{r_n}\left( z, f_{\mathbf r}^{(n)}[z](0) \right).
\]
For $S_n(z)$,
\[
\big[ S_n(z) \big]_{\boldsymbol 0, \,z}
=
\sum_{\mathbf r=(r_0,\dots,r_{n-1})\in C^n}
\pi_{r_0}(z)\,Q_{r_0,r_1}(z)\cdots Q_{r_{n-2},r_{n-1}}(z) \, \ell_{r_{n-1}}\left( z, f_{(r_0,\dots,r_{n-1})}^{(n-1)}[z](0) \right).
\]
By renaming the variable from $(r_0,\dots,r_{n-1})$ to $(r_1,\dots,r_n)$ and using equation \eqref{eq: complex stationary probability equation},
\begin{align*}
\big[ S_n(z) \big]_{\boldsymbol 0, \,z}
&=
\sum_{\mathbf r=(r_1,\dots,r_n)\in C^n}
\pi_{r_1}(z)\,Q_{r_1,r_2}(z)\cdots Q_{r_{n-1},r_n}(z) \, \ell_{r_n}\left( z, f_{(r_1,\dots,r_n)}^{(n-1)}[z](0) \right) \\
&=
\sum_{\mathbf r=(r_0,\dots,r_n)\in C^{n+1}}
\pi_{r_0}(z)\,Q_{r_0,r_1}(z)\,Q_{r_1,r_2}(z)\cdots Q_{r_{n-1},r_n}(z) \, \ell_{r_n}\left( z, f_{(r_1,\dots,r_n)}^{(n-1)}[z](0) \right) \\
&=
\sum_{\mathbf r=(r_0,\dots,r_n)\in C^{n+1}}
p_z(\mathbf r) \, \ell_{r_n}\left( z, f_{(r_1,\dots,r_n)}^{(n-1)}[z](0) \right).
\end{align*}
Therefore
\begin{equation} \label{eq: analytic parameter explicit an}
a_n(z)
\;=
\sum_{\mathbf r=(r_0,\dots,r_n)\in C^{n+1}} p_z(\mathbf r) \, \Delta_{\mathbf r}(z),
\end{equation}
where
\[
\Delta_{(r_0,\dots,r_n)}(z)
\;\coloneqq\;
\ell_{r_n}\left( z, f_{(r_0,\dots,r_n)}^{(n)}[z](0) \right)
-
\ell_{r_n}\left( z, f_{(r_1,\dots,r_n)}^{(n-1)}[z](0) \right).
\]

We now bound $p_z(\mathbf r)$ and $\Delta_{\mathbf r}(z)$. Take any compact set $K\subset \Omega$. Then, the following numbers are finite:
\begin{equation*} %\label{eq: analytic constants that make lambda holomorphic}
\overline M_{\Sigma\pi}(K)
\;\coloneqq\;
\max_{z\in K} \sum_{r\in C} \big| \pi_r(z) \big|,
\qquad
\overline M_\ell(K)
\;\coloneqq\;
\max_{r\in C}\!\!\max_{(z,w)\in K\times \closure{\mathbb D_{\bar{\rho}}}}
\big| \partial_2 \ell_r(z,w) \big|.
\end{equation*}
Here, and in what follows, $\partial_j$ denotes the partial derivative on the $j$-th coordinate.

\begin{claim} \label{claim: estimate of Delta}
For any $n\geq 1$, $\mathbf r\in C^{n+1}$ and $z\in K$,
\begin{equation} \label{eq: analytic parameter Delta bound}
\big| \Delta_{\mathbf r}(z) \big|
\leq
\overline M_\ell(K)\, \mathrm{artanh}(\bar{\rho})\,\bar{\rho}^{\,n-1}.
\end{equation}
\end{claim}
\begin{proof}[Proof of Claim \ref{claim: estimate of Delta}]
Fix $r\in C$ and $z\in K$. Since $f_r[z]\big( \mathbb D_1 \big) \subset \mathbb D_{\bar{\rho}}$ by equation \eqref{eq: analytic parameter complex image bound}, Lemma \ref{lemma: small image implies contraction} implies
\[
d_{\mathrm{hyp}}\big( f_r[z](u), f_r[z](v) \big)
\leq
\bar{\rho}\, d_{\mathrm{hyp}}(u,v)
\qquad
\text{for } u,v\in \mathbb D_1.
\]
Therefore, for every path $\mathbf r=(r_0,\dots,r_n)\in C^{n+1}$,
\begin{equation} \label{eq: analytic parameter hyperbolic path bound}
d_{\mathrm{hyp}}\left( f_{\mathbf r}^{(n)}[z](0), \,f_{(r_1,\dots,r_n)}^{(n-1)}[z](0) \right)
\leq
\bar{\rho}^{\,n-1}\, d_{\mathrm{hyp}}\big( f_{r_0}[z](0), 0 \big)
\leq
2\,\mathrm{artanh}(\bar{\rho})\,\bar{\rho}^{\,n-1}.
\end{equation}

For $u,v\in \mathbb D_{\bar{\rho}}$, by the definition of $\overline M_\ell(K)$, we have
\[
\big| \ell_r(z,u)-\ell_r(z,v) \big|
\leq
\overline M_\ell(K)\, |u-v|.
\]
Also, the density of the hyperbolic metric on $\mathbb D_1$ is
\[
\frac{2\,|dw|}{1-|w|^2} \geq 2\,|dw|,
\]
so $|u-v| \leq \frac12\, d_{\mathrm{hyp}}(u,v)$ for any $u,v\in \mathbb D_1$. Therefore,
\[
\big| \ell_r(z,u)-\ell_r(z,v) \big|
\leq
\frac{\overline M_\ell(K)}{2}\, d_{\mathrm{hyp}}(u,v)
\qquad
\text{for any } z\in K, \;\; r\in C, \;\; u,v\in \mathbb D_{\bar{\rho}}.
\]
Combining this with equation \eqref{eq: analytic parameter hyperbolic path bound}, we obtain equation \eqref{eq: analytic parameter Delta bound}.
\end{proof}

Let
\[
\overline Q(K)
\coloneqq
\max_{r\in C}\max_{z\in K}
\sum_{r'\in C} \big| Q_{r,r'}(z) \big|.
\]
Then $\overline Q(K) \leq \widetilde Q$, and for any $z\in K$,
\[
\sum_{\mathbf r=(r_0,\dots,r_n)\in C^{n+1}} \big| p_z(\mathbf r) \big|
=
\sum_{r_0\in C} \big| \pi_{r_0}(z) \big|
\sum_{r_1,\dots,r_n\in C}
\big| Q_{r_0r_1}(z) \big| \cdots \big| Q_{r_{n-1}r_n}(z) \big|
\leq
\overline M_{\Sigma\pi}(K)\,\overline Q(K)^n.
\]
Therefore,
\begin{equation} \label{eq: bound on anz}
|a_n(z)|
\leq
\frac{\overline M_{\Sigma\pi}(K)\overline M_\ell(K)\,\mathrm{artanh}(\bar{\rho})}{\bar{\rho}}
\left( \bar{\rho}\overline Q(K) \right)^n
\qquad
\text{for every } n\geq 1 \text{ and } z\in K.
\end{equation}
Combining with $\bar{\rho}\widetilde Q < 1$ from equation \eqref{eq: analytic parameter summability} and the holomorphy of $a_n(z)$, the series
\[
\frac1d \sum_{n=0}^\infty a_n(z)
\]
converges uniformly on $K$ to a holomorphic function. Since $K\subset \Omega$ was an arbitrary compact set, the series is compactly convergent on $\Omega$ and the limit is holomorphic on it.

Finally, by equation \eqref{eq: analytic parameter real series formula}, for every real $t\in J_0\cap \Omega$ this holomorphic limit is equal to $\lambda(t)$. Thus $t \longmapsto \lambda(t)$ is the restriction of a holomorphic function defined on a neighborhood of $t_0$. Therefore it is real analytic around $t_0$.
\end{proof}

\begin{remark} \label{rmk: condition on domain of holomorphic extension}
We give a clean list of conditions for $\Omega\subset \mathbb C$ so that the function
\[
\lambda(t) = \lambda_+\big( \pazocal A(t), \pazocal P(t) \big)
\]
admits a holomorphic extension on $\Omega$. Under the setting of Theorem \ref{thm: analytic dependence on parameter}, take $J_0$ as in the proof to be an open interval containing $t_0$ such that, for every $t\in J_0$, equation \eqref{eq: condition that a multicone is valid for t} is true for a fixed multicone $\{M_i\}_{i\in X}$, $M_i = \bigsqcup_a M_{i,a}$, for $\pazocal A(t_0)$ with respect to $\Sigma$. Fix the local coordinates $L_{i,a}$ and a recurrent class $C$. We can perform aperiodic reduction simultaneously across $t\in J_0$, and after updating the notation, $(\Sigma_C, Q|_C(t))$ is irreducible and aperiodic for every $t\in J_0$. Denote by $B_r(t)$ the local positive matrix for each $r\in C$ and $t\in J_0$.

Suppose that a complex neighborhood $\Omega$ of $t_0$ satisfies the following conditions.
\begin{enumerate}
\item[(\(\Omega\)1)] $\Omega\cap \mathbb R \subset J_0$, and every entry of each $A_i(t)$ and $\pazocal P(t)$ extend holomorphically on $\Omega$.
\item[(\(\Omega\)2)] Fix $r_* \in C$, and complexify the transition matrix $Q(t)|_C$ and write $Q(z)|_C$. If $M(z)$ is obtained by replacing the $r_*$-th column of $I-Q(z)|_C$ with $(1,\dots,1)^\top$, then $M(z)$ is invertible for every $z\in \Omega$. In addition, defining $\pi(z) = \big(\pi_r(z)\big)_{r\in C}$ by $\pi(z)^\top M(z)=e_{r_*}^\top$, we have $\pi_r(z)\neq 0$ for every $r\in C$ and $z\in \Omega$.
\item[(\(\Omega\)3)] For each $r\in C$, complexify $B_r(t)$ and write $B_r(z)$. Define the M\"obius map $f_r[z]$ by
\begin{equation*}
f_r[z](x)
\coloneqq
\frac{
\alpha_r(z)x + \beta_r(z)}{
\gamma_r(z)x + \delta_r(z)}
\quad \text{where} \quad
\begin{pmatrix}
\alpha_r(z) & \beta_r(z) \\
\gamma_r(z) & \delta_r(z)
\end{pmatrix}
\coloneqq
F\big( B_r(z) \big).
\end{equation*}
Then there exist numbers $0<\rho<\bar{\rho}<1$ such that
\[
f_r[t_0]\big( \overline{\mathbb D_1} \big)
\subset
\closure{\mathbb D_\rho},
\quad \text{and} \quad
f_r[z]\big( \overline{\mathbb D_1} \big)
\subset
\mathbb D_{\bar \rho}
\quad
\text{for every } r\in C \text{ and } z\in \Omega.
\]
\item[(\(\Omega\)4)] For every $r\in C$, $z\in \Omega$, and $x\in \overline{\mathbb D_1}$,
\[
\Re\big( \gamma_r(z)x+\delta_r(z) \big) >0.
\]
\item[(\(\Omega\)5)] We have $\bar{\rho}\widetilde Q < 1$, where
\[
\widetilde Q
\;\coloneqq\;
\max_{r\in C}\sup_{z\in \Omega}
\sum_{r'\in C} \big| Q_{r,r'}(z) \big|.
\]
\end{enumerate}
Then $Q(z)|_C$, $\pi(z)$, and $B_r(z)$ for each $r\in C$ are entrywise holomorphic on $\Omega$, and the maps
\[
(z,x)
\;\longmapsto\;
f_r[z](x),
\qquad
(z,x)
\;\longmapsto\;
\ell_r(z,x) \coloneqq \Log\big( \gamma_r(z)x+\delta_r(z) \big)
\]
are defined on $\Omega\times \closure{\mathbb D_1}$ and holomorphic on $\Omega\times \mathbb D_1$ for every $r\in C$. Defining $T_r(z)$, $\mathrm T(z)$, and $\boldsymbol v(z)$ by exactly the same formulae as in Section \ref{section: Explicit representation of the Lyapunov exponent}, they are entrywise holomorphic. Furthermore, the series
\[
\lambda(z)
\;\coloneqq\;
\frac1d \sum_{n=0}^\infty \big[ \mathrm T(z)^n \boldsymbol v(z) \big]_{\boldsymbol 0,\,z}
\quad \text{where} \quad
[\boldsymbol u]_{\boldsymbol 0, \,z}
\;\coloneqq\;
\sum_{r\in C}\pi_r(z) (\boldsymbol u)_{r,0}.
\]
converges compactly on $\Omega$, where $d$ is the period of the chosen recurrent class before aperiodic reduction. In particular, $\lambda(z)$ is holomorphic on $\Omega$.
\end{remark}

Theorem \ref{thm: analytic dependence on parameter}, together with Remark \ref{rmk: condition on domain of holomorphic extension}, completes the proof of Theorem \ref{thm: real analyticity of Lyapunov exponent} from the introduction.

\section{Calculating the derivatives} \label{section: Calculating the derivatives}

In the previous section we proved that, around projectively uniformly hyperbolic systems, the Lyapunov exponent is real analytic with respect to the underlying data. We also provided an explicit domain on which the Taylor expansion converges. The purpose of this section is to prove that each derivative, and therefore each Taylor coefficient, can be calculated in polynomial time.

\begin{theorem} \label{theorem: polynomial time calculation of every derivative}
Assume the hypotheses of Theorem \ref{thm: analytic dependence on parameter} and write $\lambda(t) = \lambda_+\big( \pazocal A(t), \, \pazocal P(t) \big)$. Then, for each fixed integer $q\geq1$, the $q$-th derivative $\lambda^{\!(q)}(t_0) = \left.\frac{d^q}{dt^q}\right|_{t=t_0} \lambda(t)$ can be approximated in polynomial time. That is, one can compute $\lambda^{\!(q)}(t_0)$ to error $\vep > 0$ with $O\big( (\log(1/\vep))^3 \big)$ arithmetic operations.
\end{theorem}

We work under the hypotheses and notation of Theorem \ref{thm: analytic dependence on parameter} and Remark \ref{rmk: condition on domain of holomorphic extension}. We fix $J_0$ and a complex neighborhood $\Omega$ of $t_0$ satisfying the conditions in Remark \ref{rmk: condition on domain of holomorphic extension}. Let $d$ denote the period of the chosen recurrent class before aperiodic reduction.

For each $r\in C$, $z\in \Omega$ and integer $m\geq 2$, let the infinite matrix $T_{r,m}(z)$ be the upper left $m\times m$ truncation of $T_r(z)$, and define $\mathrm T_{m\times m}(z)$ as in Section \ref{section: Evaluation and Polynomial-time Convergence}. For natural numbers $n\geq 1$, $m\geq 2$, and $z\in \Omega$, define
\[
\Lambda_{n,m}(z)
\coloneqq
\frac1d \sum_{\ell=0}^{n-1} \Big[ \mathrm T_{m\times m}(z)^\ell \boldsymbol v(z) \Big]_{\boldsymbol 0,\,z}
\quad \text{where} \quad
[\boldsymbol u]_{\boldsymbol 0, \,z}
\;\coloneqq\;
\sum_{r\in C}\pi_r(z) (\boldsymbol u)_{r,0}.
\]
For each integer $q \geq 0$, define
\[
\Lambda_{n,m}^{\!(q)}(t_0)
\;\coloneqq\;
\left.\frac{d^q}{dz^q}\right|_{z = t_0} \Lambda_{n,m}(z).
\]

We will prove that $\Lambda_{n,m}^{\!(q)}(t_0)$ approximates $\lambda^{\!(q)}(t_0)$ exponentially fast as $n,m$ grow large. In what follows, we will prove two kinds of error bounds. In Subsection \ref{First-order derivative I: gathering bounds}, we present a bound for the first-order derivative $q=1$ that is independent of the domain of holomorphic extension, whereas a bound via the Cauchy integral formula is given for general $q$ in Subsection \ref{subsection: Higher-order derivatives}. Although it is possible to extend the $q=1$ bounds for general $q$ using the Fa\`a di Bruno formula, we will omit this for brevity.

\medskip

In practice, one computes $\Lambda_{n,m}^{\!(q)}(t_0)$ using the following construction. Fix $m\geq 2$. Denote by $T_{r,m}^{(q)}(z)$ the entrywise $q$-th derivative of $T_{r,m}(z)$. We define $\mathrm T_{m\times m}^{(q)}(z): \ell^\infty(\mathbb N_0)^C \to \ell^\infty(\mathbb N_0)^C$ by
\[
\big( \mathrm T_{m\times m}^{(q)}(z) \boldsymbol u \big)_{r'}
\;\coloneqq\;
\sum_{r\in C} \sum_{j=0}^q \binom{q}{j} \left( \frac{d^{q-j}}{dz^{q-j}} \frac{\pi_r(z) Q_{r,r'}(z)}{\pi_{r'}(z)} \right) T_{r,m}^{(j)}(z) u_r
\]
for $\boldsymbol u=(u_r)_{r\in C}\in \ell^\infty(\mathbb N_0)^C$ and $r' \in C$. We define $\boldsymbol w^{(q)}_n \in \ell^\infty(\mathbb N_0)^C$ for integers $q, n \geq 0$ by
\[
\boldsymbol w^{(q)}_0
\;\coloneqq\;
\left. \frac{d^q}{dz^q}\right|_{z = t_0} \boldsymbol v(z)
\qquad \text{and} \qquad
\boldsymbol w^{(q)}_n
\;\coloneqq\;
\sum_{\ell=0}^q \binom{q}{\ell}
\mathrm T_{m\times m}^{(\ell)}(t_0)\, \boldsymbol w^{(q-\ell)}_{n-1}
\; \text{ for } n\geq 1.
\]
Here, the derivative of $\boldsymbol v(z)$ is performed entrywise. Then, by Leibniz’s rule and induction on $n$,
\[
\boldsymbol w_n^{(q)}
=
\left.\frac{d^q}{dz^q}\right|_{z=t_0}
\big( \mathrm T_{m\times m}(z)^n \boldsymbol v(z) \big).
\]
Hence
\begin{equation} \label{eq: inductive identity of finite approximation of general derivatives of lambda}
\Lambda_{n,m}^{\!(q)}(t_0)
=
\frac1d \sum_{r\in C} \sum_{\ell=0}^{n-1} \sum_{j=0}^q
\binom{q}{j} \left( \left.\frac{d^{q-j}}{dz^{q-j}}\right|_{z = t_0} \pi_r(z) \right) \big( \boldsymbol w^{(j)}_\ell \big)_{r,0}.
\end{equation}

\subsection{First-order derivative \RN 1: gathering bounds} \label{First-order derivative I: gathering bounds}

For the first derivative, equation \eqref{eq: inductive identity of finite approximation of general derivatives of lambda} is
\begin{equation} \label{eq: inductive identity of finite approximation of first derivative of lambda}
\Lambda_{n,m}'(t_0)
=
\Lambda_{n,m}^{\!(1)}(t_0)
=
\frac1d \sum_{\ell=0}^{n-1}
\left(
\sum_{r\in C} \pi_r'(t_0)\, \big( \boldsymbol w^{(0)}_\ell \big)_{r,0}
+ \sum_{r\in C} \pi_r(t_0)\, \big( \boldsymbol w^{(1)}_\ell \big)_{r,0}
\right).
\end{equation}

We define the following constants.
\begin{gather*}
m_\pi
\;=\;
\min_{r\in C} \pi_r(t_0),
\qquad
M_\pi^{(1)}
\;=\;
\max_{r\in C}
|\pi_r'(t_0)|,
\qquad
M_{\Sigma \pi}^{(1)}
\;=\;
\sum_{r\in C} |\pi_r'(t_0)|,
\\
M_Q^{(1)}
\;=\;
\max\left\{
\max_{r\in C} \sum_{r'\in C} |Q_{r,r'}'(t_0)|,\;
\max_{r'\in C} \sum_{r\in C} |Q_{r,r'}'(t_0)|
\right\},
\qquad
M_\top^{(1)}
\;=\;
\max_{r\in C} \left| \left.\frac{d}{dt}\right|_{t=t_0} f_r^\top[t](0) \right|,
\\
M_f^{(1)}
\;=\;
\max_{r\in C} \sup_{x\in \closure{\mathbb D_1}} \left| \left.\frac{d}{dt}\right|_{t=t_0} f_r[t](x) \right|,
\qquad
M_{\ell,2}
\;=\;\;
\max_{r\in C}\!\!\sup_{x\in [-\rho,\rho]}
\big| \partial_2 \ell_r(t_0,x) \big|,
\\
M_{\ell,22}
\;=\;\;
\max_{r\in C}\!\!\sup_{x\in [-\rho,\rho]}
\big| \partial_2^2 \ell_r(t_0,x) \big|,
\qquad
M_{\ell,12}
\;=\;\;
\max_{r\in C}\!\!\sup_{x\in [-\rho,\rho]}
\big| \partial_1 \partial_2 \ell_r(t_0,x) \big|.
\end{gather*}

Recall the definition of $f_{\mathbf r}^{(n)}[z]$ in equation \eqref{eq: definition of analytic parameter compositions}. Also, let
\[
\Xi \;=\; \frac{M_f^{(1)}}{(1-\rho)(1-\rho^2)}.
\]

\begin{lemma} \label{lemma: pathwise first order derivative bound}
For any path $\mathbf r=(r_0,\dots,r_{n-1})\in C^n$,
\[
\left| \left. \frac{d}{dt}\right|_{t=t_0} f_{\mathbf r}^{(n)}[t](0) \right|
\;\leq\;
\Xi(1-\rho^n).
\]
\end{lemma}

\begin{proof}
Fix a path $\mathbf r=(r_0,\dots,r_{n-1})\in C^n$, and define
\[
y_0(t)\coloneqq0,
\qquad
y_{j+1}(t)\coloneqq f_{r_j}[t]\big(y_j(t)\big)
\qquad
(0\leq j\leq n-1).
\]
Then
\[
y_n(t)=f_{\mathbf r}^{(n)}[t](0).
\]
Take any $\hat{\rho}$ with $\rho<\hat{\rho}<1$ and let $\delta>0$ such that if $|t-t_0|<\delta$, then $f_r[t]\big([-1,1]\big)\subset [-\hat{\rho},\hat{\rho}]$ for every $r\in C$. For $|u|<\delta$, we have $y_j(t)\in[-\hat{\rho},\hat{\rho}]$ for every $j$ and $t\in \big[t_0-|u|,\,t_0+|u|\big]$.

For $u$ with $|u|<\delta$, using the triangle inequality for $d_{\mathrm{hyp}}$,
\begin{align*}
&d_{\mathrm{hyp}}\big(y_{j+1}(t_0+u),y_{j+1}(t_0)\big)
=
d_{\mathrm{hyp}} \Big( f_{r_j}[t_0+u]\big(y_j(t_0+u)\big),\, f_{r_j}[t_0]\big(y_j(t_0)\big) \Big)
\\
&\leq
d_{\mathrm{hyp}} \Big( f_{r_j}[t_0+u]\big(y_j(t_0+u)\big),\, f_{r_j}[t_0+u]\big(y_j(t_0)\big) \Big)
+
d_{\mathrm{hyp}} \Big( f_{r_j}[t_0+u]\big(y_j(t_0)\big),\, f_{r_j}[t_0]\big(y_j(t_0)\big) \Big).
\end{align*}
First term: since $f_{r_j}[t_0+u]\big([-1,1]\big)\subset[-\hat{\rho},\hat{\rho}]$, Lemma \ref{lemma: real contraction implies complex contraction} and Lemma \ref{lemma: small image implies contraction} give
\[
d_{\mathrm{hyp}} \Big( f_{r_j}[t_0+u]\big(y_j(t_0+u)\big),\, f_{r_j}[t_0+u]\big(y_j(t_0)\big) \Big)
\leq
\hat{\rho}\, d_{\mathrm{hyp}}\big(y_j(t_0+u),\, y_j(t_0)\big).
\]
Second term: both points $f_{r_j}[t_0+u]\big(y_j(t_0)\big)$ and $f_{r_j}[t_0]\big(y_j(t_0)\big)$ belong to $[-\hat{\rho},\hat{\rho}]$, so by the integral formula of $d_{\mathrm{hyp}}$,
\begin{align*}
&d_{\mathrm{hyp}} \Big( f_{r_j}[t_0+u]\big(y_j(t_0)\big), f_{r_j}[t_0]\big(y_j(t_0)\big) \Big)
\leq
\left|\int_{f_{r_j}[t_0]\big(y_j(t_0)\big)}^{f_{r_j}[t_0+u]\big(y_j(t_0)\big)} \frac{2}{1-x^2}\,dx\right|
\\
&\hspace{90pt} \leq
\frac{2}{1-\hat{\rho}^2} \left| f_{r_j}[t_0+u]\big(y_j(t_0)\big) - f_{r_j}[t_0]\big(y_j(t_0)\big) \right|
\leq
\frac{2G(u)}{1-\hat{\rho}^2}|u|,
\end{align*}
where
\[
G(u) = \max_{t\in \big[ t_0-|u|, \,t_0+|u| \big]} \max_{r\in C} \max_{x\in \closure{\mathbb D_1}} \left| \partial_t f_r[t](x) \right|.
\]

Therefore, for every $0\leq j\leq n-1$,
\begin{equation} \label{eq: recurrence bound for sequence yj}
d_{\mathrm{hyp}}\big(y_{j+1}(t_0+u),\,y_{j+1}(t_0)\big)
\leq
\hat{\rho}\,d_{\mathrm{hyp}}\big(y_j(t_0+u),\,y_j(t_0)\big)
+ \frac{2G(u)}{1-\hat{\rho}^2}|u|.
\end{equation}

Now define
\[
D_j = \limsup_{u\to 0} \frac{d_{\mathrm{hyp}}\big(y_j(t_0+u),\,y_j(t_0)\big)}{|u|}.
\]
Since $y_0(t)\equiv 0$, we have $D_0=0$. By the analyticity of each $f_r$ and equation \eqref{eq: recurrence bound for sequence yj},
\[
D_{j+1} \leq \hat{\rho}D_j + \frac{2M_f^{(1)}}{1-\hat{\rho}^2}.
\]
By induction,
\[
D_n
\leq
\frac{2M_f^{(1)}}{1-\hat{\rho}^2}\sum_{m=0}^{n-1}\hat{\rho}^m
=
\frac{2M_f^{(1)}}{1-\hat{\rho}^2}\, \frac{1-\hat{\rho}^n}{1-\hat{\rho}}.
\]

Since the density of the hyperbolic metric is bounded below by $2$ on $(-1,1)$, we have
\[
\left| \left. \frac{d}{dt}\right|_{t=t_0} f_{\mathbf r}^{(n)}[t](0) \right|
=
|y_n'(t_0)|
\leq
\frac12 D_n
\leq
\frac{M_f^{(1)}(1-\hat{\rho}^n)}{(1-\hat{\rho})(1-\hat{\rho}^2)}.
\]
Since this is true for every $\hat{\rho}>\rho$, letting $\hat{\rho}\downarrow\rho$ we obtain the desired inequality.
\end{proof}

For $n\geq 0$, $\mathbf r=(r_0,\dots,r_n)\in C^{n+1}$, and $s\in C$, write $(\mathbf r, s)$ for the concatenation and define
\[
c_{\mathbf r,s}(t)
\coloneqq
\frac{p_t\big((\mathbf r, s)\big)}{ \pi_s(t) }
=
\frac{ \pi_{r_0}(t)\,Q_{r_0r_1}(t)\cdots Q_{r_{n-1}r_n}(t)\,Q_{r_n,s}(t) }{ \pi_s(t) }.
\]
Since $Q(t_0)|_C$ is row-stochastic and $\pi(t_0)$ is its stationary probability vector, we have
\begin{equation} \label{eq: path coefficient sums to 1}
\sum_{\mathbf r\in C^{n+1}} p_{t_0}\big((\mathbf r, s)\big) = \pi_s(t_0),
\qquad
\sum_{\mathbf r\in C^{n+1}} c_{\mathbf r,s}(t_0) = 1
\qquad (s\in C).
\end{equation}
Write $p_{t_0}'(\mathbf r) \coloneqq \left.\frac{d}{dt}\right|_{t=t_0} p_t(\mathbf r)$.

\begin{lemma} \label{lemma: bounding the derivative of path coefficient}
For any $n\geq 0$ and $s\in C$,
\begin{equation*}
\sum_{\mathbf r\in C^{n+1}}
\left| p_{t_0}'(\mathbf r) \right|
\leq
M_{\Sigma\pi}^{(1)} + nM_Q^{(1)},
\qquad
\sum_{\mathbf r\in C^{n+1}}
|c_{\mathbf r,s}'(t_0)|
\leq
\frac{2M_\pi^{(1)} + (n+1)M_Q^{(1)}}{m_\pi}.
\end{equation*}
\end{lemma}

\begin{proof}
We have
\[
c_{\mathbf r,s}'(t_0)
=
\frac{p_{t_0}'\big((\mathbf r, s)\big)}{\pi_s(t_0)}
-
\frac{\pi_s'(t_0)}{\pi_s(t_0)^2}\,p_{t_0}\big((\mathbf r, s)\big).
\]
Therefore, using $\pi_s(t_0)\geq m_\pi$ and equation \eqref{eq: path coefficient sums to 1}, we obtain
\begin{equation*}
\sum_{\mathbf r\in C^{n+1}} |c_{\mathbf r,s}'(t_0)|
\leq
\frac{1}{\pi_s(t_0)}
\sum_{\mathbf r\in C^{n+1}} |p_{t_0}'\big((\mathbf r, s)\big)|
+
\frac{|\pi_s'(t_0)|}{\pi_s(t_0)}
\leq
\frac{1}{\pi_s(t_0)}
\sum_{\mathbf r\in C^{n+1}} |p_{t_0}'\big((\mathbf r, s)\big)|
+
\frac{M_\pi^{(1)}}{m_\pi}.
\end{equation*}

We now bound the first term. For the derivative falling on $\pi_{r_0}(t)$, we have
\begin{align*}
&\frac{1}{\pi_s(t_0)} \sum_{\mathbf r\in C^{n+1}}
|\pi_{r_0}'(t_0)|\,Q_{r_0r_1}(t_0)\cdots Q_{r_n,s}(t_0) \\
&\leq
\frac{1}{\pi_s(t_0)} \left( \max_{r\in C}\frac{|\pi_r'(t_0)|}{\pi_r(t_0)} \right)
\sum_{\mathbf r\in C^{n+1}} \pi_{r_0}(t_0)\,Q_{r_0r_1}(t_0)\cdots Q_{r_n,s}(t_0)
=
\frac{1}{\pi_s(t_0)} \max_{r\in C}\frac{|\pi_r'(t_0)|}{\pi_r(t_0)}\,\pi_s(t_0)
\leq
\frac{M_\pi^{(1)}}{m_\pi}.
\end{align*}
Next fix $0\leq j\leq n-1$, and consider the term with derivative on $Q_{r_j,r_{j+1}}(t)$. By the stationarity of $\pi(t_0)$ and the row-stochasticity of $Q(t_0)$,
\begin{align*}
&\frac{1}{\pi_s(t_0)}
\sum_{\mathbf r\in C^{n+1}}
\pi_{r_0}(t_0)\,Q_{r_0r_1}(t_0)\cdots
Q_{r_{j-1}r_j}(t_0)\,
|Q_{r_j,r_{j+1}}'(t_0)|\,
Q_{r_{j+1}r_{j+2}}(t_0)\cdots Q_{r_n,s}(t_0) \\
&\leq
\frac{1}{\pi_s(t_0)}
\sum_{r_j\in C}\pi_{r_j}(t_0)\sum_{r_{j+1}\in C}|Q_{r_j,r_{j+1}}'(t_0)| \big(Q(t_0)^{n-j}\big)_{r_{j+1},s}
\leq
\frac{M_Q^{(1)}}{m_\pi}.
\end{align*}
Finally, the term with derivative on $Q_{r_n,s}(t)$ is bounded by
\[
\frac{1}{\pi_s(t_0)} \sum_{r_n\in C} \pi_{r_n}(t_0) |Q_{r_n,s}'(t_0)|
\leq
\frac{1}{m_\pi} \max_{s'\in C} \sum_{r\in C} |Q_{r, s'}'(t_0)|
\leq
\frac{M_Q^{(1)}}{m_{\pi}}.
\]
This completes the proof for $c_{\mathbf r,s}$. The proof for $p_t(\mathbf r)$ is similar.
\end{proof}

For each $r\in C$, let
\[
g_r: J_0\times \mathbb R \to \mathbb R,
\qquad
g_r(t,s)
\coloneqq
\mathrm{artanh}\big( f_r[t](\tanh s) \big).
\]
Define
\[
M_{g,12}
\coloneqq
\max_{r\in C} \max_{|s|\leq \mathrm{artanh}(\rho)}
\big| \partial_1\partial_2 g_r(t_0,s) \big|,
\qquad
M_{g,22}
\coloneqq
\max_{r\in C} \max_{|s|\leq \mathrm{artanh}(\rho)}
\big| \partial_2^2 g_r(t_0,s) \big|,
\]
and
\[
L_g
\coloneqq
M_{g,12} + \frac{\Xi}{1-\rho^2}\,M_{g,22}.
\]

\begin{lemma} \label{lemma: first order derivative of pathwise difference of iterates}
Let $n\geq 1$ and $\mathbf r=(r_0,\dots,r_n)\in C^{n+1}$. Define
\[
d_{\mathbf r}(t)
\;\coloneqq\;
f_{(r_0,\dots,r_n)}^{(n)}[t](0)
\;-\;
f_{(r_1,\dots,r_n)}^{(n-1)}[t](0).
\]
Then
\[
|d_{\mathbf r}(t_0)|
\leq
\mathrm{artanh}(\rho)\,\rho^{n-1},
\quad
\big| d_{\mathbf r}'(t_0) \big|
\leq
\left(
\frac{M_f^{(1)}}{1-\rho^2}
+ \frac{n-1}{\rho}\,L_g\,\mathrm{artanh}(\rho)
+ \frac{\Xi\,\mathrm{artanh}(\rho)}{1-\rho^2}
\right)\rho^{n-1}.
\]
\end{lemma}

\begin{proof}
Let $\eta \coloneqq \mathrm{artanh}(\rho)$. Let $u_0(t)\coloneqq f_{r_0}[t](0)$ and $v_0(t)\coloneqq0$. For $0\leq j\leq n-2$, we define $u_j(t), v_j(t)$ recursively by
\[
u_{j+1}(t)\coloneqq f_{r_{j+1}}[t]\big( u_j(t) \big),
\qquad
v_{j+1}(t)\coloneqq f_{r_{j+1}}[t]\big( v_j(t) \big).
\]
Then $u_{n-1}(t) = f_{(r_0,\dots,r_{n-1})}^{(n)}[t](0)$ and $v_{n-1}(t) = f_{(r_1,\dots,r_n)}^{(n-1)}[t](0)$. Also, $u_j(t_0), v_j(t_0)\in [-\rho,\rho]$ for every $j$. We pass to hyperbolic coordinates: for $0\leq j\leq n-1$,
\[
U_j(t)\coloneqq\mathrm{artanh}\big( u_j(t) \big),
\qquad
V_j(t)\coloneqq\mathrm{artanh}\big( v_j(t) \big),
\qquad
\delta_j(t) \coloneqq U_j(t)-V_j(t).
\]
Then $U_j(t_0),V_j(t_0)\in[-\eta,\eta]$ for every $j$. Also, by the definition of $g_r$, we have $U_0(t)=g_{r_0}(t,0)$, $V_0(t)=0$, and for $0\leq j\leq n-2$,
\[
U_{j+1}(t)=g_{r_{j+1}}\big( t,U_j(t) \big),
\qquad
V_{j+1}(t)=g_{r_{j+1}}\big( t,V_j(t) \big).
\]

We first prove that
\begin{equation} \label{eq: bound on hyperbolic zeroth order difference}
|\delta_j(t_0)|
\leq
\eta\,\rho^j
\qquad
(0\leq j\leq n-1).
\end{equation}
Indeed, since $|u_0(t_0)|\leq \rho$, we have $|\delta_0(t_0)|=|U_0(t_0)|=\big| \mathrm{artanh}(u_0(t_0)) \big|\leq\eta$. Let $0\leq j\leq n-2$. Notice that $2|\delta_j(t_0)| = d_{\mathrm{hyp}}\big( u_j(t_0),v_j(t_0) \big)$. By
$
f_r[t_0]\big( \overline{\mathbb D_1} \big)
\subset
\closure{\mathbb D_\rho}
$ and Lemma \ref{lemma: small image implies contraction},
\begin{align*}
|\delta_{j+1}(t_0)|
=
\frac{1}{2} d_{\mathrm{hyp}}\big( u_{j+1}(t_0),v_{j+1}(t_0) \big)
&=
\frac{1}{2} d_{\mathrm{hyp}}
\Big( f_{r_{j+1}}[t_0]\big( u_j(t_0) \big), f_{r_{j+1}}[t_0]\big( v_j(t_0) \big) \Big) \\
&\leq
\frac{\rho}{2} d_{\mathrm{hyp}}\big( u_j(t_0),v_j(t_0) \big)
=
\rho\,|\delta_j(t_0)|.
\end{align*}
This proves \eqref{eq: bound on hyperbolic zeroth order difference} by induction.

Next, we claim that
\begin{equation} \label{eq: derivative of delta j bound}
|\delta_j'(t_0)|
\leq
\left(
\frac{M_f^{(1)}}{1-\rho^2} + \frac{j}{\rho}\,L_g\,\eta
\right)\rho^j
\qquad
(0\leq j\leq n-1).
\end{equation}
For $j=0$, since $f_r[t_0]\big([-1,1]\big)\subset[-\rho,\rho]$,
\[
|\delta_0'(t_0)|
=
\left| U_0'(t_0) \right|
=
\left| \partial_1 g_{r_0}(t_0,0) \right|
=
\left| \frac{\left.\frac{d}{dt}\right|_{t=t_0} f_{r_0}[t](\tanh 0)}{1-\big(f_{r_0}[t_0](\tanh 0)\big)^2} \right|
\leq
\frac{M_f^{(1)}}{1-\rho^2}.
\]
Now assume that \eqref{eq: derivative of delta j bound} holds for some $j$. From $\delta_{j+1}(t) = g_{r_{j+1}}\big( t,U_j(t) \big) - g_{r_{j+1}}\big( t,V_j(t) \big)$,
\begin{align*}
\delta_{j+1}'(t_0)
&=
\partial_1 g_{r_{j+1}}\big( t_0,U_j(t_0) \big)
- \partial_1 g_{r_{j+1}}\big( t_0,V_j(t_0) \big) \\
&\hspace{-30pt}+
\partial_2 g_{r_{j+1}}\big( t_0,U_j(t_0) \big)\,
\big( U_j'(t_0) - V_j'(t_0) \big)
+
\Big(
\partial_2 g_{r_{j+1}}\big( t_0,U_j(t_0) \big)
- \partial_2 g_{r_{j+1}}\big( t_0,V_j(t_0) \big)
\Big)
V_j'(t_0).
\end{align*}
Now, note that for every $r\in C$ and every $s,s'\in[-\eta,\eta]$,
\begin{align*}
\big| g_r(t_0,s)-g_r(t_0,s') \big|
&=
\frac12 d_{\mathrm{hyp}}\Big( f_r[t_0](\tanh s), f_r[t_0](\tanh s') \Big) \\
&\leq
\frac{\rho}{2} d_{\mathrm{hyp}}\big( \tanh s,\tanh s' \big)
=
\rho\,|s-s'|.
\end{align*}
Therefore $g_r(t_0,\cdot)$ is $\rho$-Lipschitz on $[-\eta,\eta]$, and hence
\begin{align} \label{eq: recurrence for delta j intermediate}
|\delta_{j+1}'(t_0)|
&\leq
M_{g,12}\,|\delta_j(t_0)| + \rho\,|\delta_j'(t_0)| + M_{g,22}\,|\delta_j(t_0)|\,|V_j'(t_0)|.
\end{align}
Since $|v_j(t_0)|\leq \rho$, Lemma \ref{lemma: pathwise first order derivative bound} gives
\begin{equation} \label{eq: bound of derivative of V}
|V_j'(t_0)|
=
\frac{|v_j'(t_0)|}{1-v_j(t_0)^2}
\leq
\frac{\Xi(1-\rho^j)}{1-\rho^2}
\leq
\frac{\Xi}{1-\rho^2}.
\end{equation}
Substituting this and \eqref{eq: bound on hyperbolic zeroth order difference} into
\eqref{eq: recurrence for delta j intermediate}, we get
\[
|\delta_{j+1}'(t_0)|
\leq
\rho\,|\delta_j'(t_0)|
+
L_g\,\eta\,\rho^j.
\]
Using the induction hypothesis, we see that \eqref{eq: derivative of delta j bound} holds for $j+1$:
\begin{align*}
|\delta_{j+1}'(t_0)|
\leq
\rho
\left(
\frac{M_f^{(1)}}{1-\rho^2} + \frac{j}{\rho}\,L_g\,\eta
\right)\rho^j
+
L_g\,\eta\,\rho^j
=
\left(
\frac{M_f^{(1)}}{1-\rho^2} + \frac{j+1}{\rho}\,L_g\,\eta
\right)\rho^{j+1}.
\end{align*}

We now return to Euclidean coordinates. Since
$
d_{\mathbf r}(t) = \tanh\big( U_{n-1}(t) \big) - \tanh\big( V_{n-1}(t) \big),
$
the zeroth-order bound follows from \eqref{eq: bound on hyperbolic zeroth order difference}, because $\tanh$ is $1$-Lipschitz:
\[
|d_{\mathbf r}(t_0)|
\leq
|\delta_{n-1}(t_0)|
\leq
\eta\,\rho^{n-1}.
\]

For the derivative, we compute
\[
d_{\mathbf r}'(t_0)
=
\tanh'\big( U_{n-1}(t_0) \big)\,\delta_{n-1}'(t_0)
+
\Big(
\tanh'\big( U_{n-1}(t_0) \big) - \tanh'\big( V_{n-1}(t_0) \big)
\Big)
V_{n-1}'(t_0).
\]
Since $|\tanh'|\leq 1$ on $\mathbb R$, and since $\sup_{s\in\mathbb R} |\tanh''(s)| = 4/(3\sqrt{3}) < 1$, we have
\[
|d_{\mathbf r}'(t_0)|
\leq
|\delta_{n-1}'(t_0)|
+
|\delta_{n-1}(t_0)|\,|V_{n-1}'(t_0)|.
\]
Using equations \eqref{eq: bound on hyperbolic zeroth order difference}, \eqref{eq: derivative of delta j bound}, and \eqref{eq: bound of derivative of V}, we conclude that
\[
|d_{\mathbf r}'(t_0)|
\leq
\left(
\frac{M_f^{(1)}}{1-\rho^2}
+ \frac{n-1}{\rho}\,L_g\,\eta
+ \frac{\Xi\eta}{1-\rho^2}
\right)\rho^{n-1}. \qedhere
\]
\end{proof}

Define
\[
a_n(t) = \Big[ \mathrm T(t)^n \boldsymbol v(t) \Big]_{\boldsymbol 0, \,t},
\qquad
\Lambda_n(t) = \frac1d \sum_{\ell=0}^{n-1} a_\ell(t)
\qquad
(n\geq 0,\ t\in J_0).
\]

\begin{lemma} \label{lemma: derivative tail bound}
For every $n\geq 1$,
\[
|a_n'(t_0)|
\leq
\bigg( nA_{\mathrm{tail}}+ B_{\mathrm{tail}} \bigg) \rho^{n-1},
\]
where
\[
A_{\mathrm{tail}}
=
\left( \frac{1}{\rho} L_g + M_Q^{(1)} \right) M_{\ell,2} \mathrm{artanh}(\rho),
\]
and
\[
B_{\mathrm{tail}}
=
\left(
M_{\Sigma\pi}^{(1)} M_{\ell,2} + M_{\ell,12} + \Xi M_{\ell,22}
+\frac{\Xi\,M_{\ell,2}}{1-\rho^2}
\right)\mathrm{artanh}(\rho)
+ \frac{M_f^{(1)} M_{\ell,2}}{1-\rho^2}.
\]
Consequently, for every $n\in \mathbb N$,
\[
\left|
\lambda'(t_0)- \Lambda_n'(t_0)
\right|
\;\leq\;
\frac1d \left(
\left( \frac{n}{1-\rho}+\frac{\rho}{(1-\rho)^2} \right)A_{\mathrm{tail}}
+ \frac{B_{\mathrm{tail}}}{1-\rho}
\right) \rho^{n-1}.
\]
\end{lemma}

\begin{proof}
For $n\geq 1$, exactly as in the proof of Theorem \ref{thm: analytic dependence on parameter}, we have
\[
a_n(t)
=
\sum_{\mathbf r=(r_0,\dots,r_n)\in C^{n+1}} p_t(\mathbf r)\,\Delta_{\mathbf r}(t),
\]
where
\[
p_t(\mathbf r)
=
\pi_{r_0}(t)\,Q_{r_0r_1}(t)\cdots Q_{r_{n-1}r_n}(t),
\]
and
\[
\Delta_{(r_0,\dots,r_n)}(t)
=
\ell_{r_n}\!\left( t, f_{(r_0,\dots,r_n)}^{(n)}[t](0) \right)
- \ell_{r_n}\!\left( t, f_{(r_1,\dots,r_n)}^{(n-1)}[t](0) \right).
\]

Also, by rewriting the same argument in Claim \ref{claim: estimate of Delta} with real parameter $t$,
\begin{equation} \label{eq: zeroth order Delta bound at t0}
\big| \Delta_{\mathbf r}(t_0) \big|
\leq
M_{\ell,2} \,\mathrm{artanh}(\rho)\,\rho^{n-1}
\qquad
\text{for every } \mathbf r\in C^{n+1}.
\end{equation}

Now fix $\mathbf r=(r_0,\dots,r_n)\in C^{n+1}$, and define
\[
u_{\mathbf r}(t)
\coloneqq f_{(r_0,\dots,r_n)}^{(n)}[t](0),
\qquad
v_{\mathbf r}(t)
\coloneqq f_{(r_1,\dots,r_n)}^{(n-1)}[t](0).
\]
Then
\[
\Delta_{\mathbf r}(t)
=
\ell_{r_n}\big( t,u_{\mathbf r}(t)\big)
- \ell_{r_n}\big( t,v_{\mathbf r}(t)\big).
\]
Differentiating at $t=t_0$, we obtain
\begin{align*}
\Delta_{\mathbf r}'(t_0)
&=
\partial_1 \ell_{r_n}\big( t_0,u_{\mathbf r}(t_0)\big)
- \partial_1 \ell_{r_n}\big( t_0,v_{\mathbf r}(t_0)\big)
\\
&\quad
+
\Big(
\partial_2 \ell_{r_n}\big( t_0,u_{\mathbf r}(t_0)\big)
- \partial_2 \ell_{r_n}\big( t_0,v_{\mathbf r}(t_0)\big)
\Big)\,v_{\mathbf r}'(t_0)
+ \partial_2 \ell_{r_n}\big( t_0,u_{\mathbf r}(t_0)\big)\, \big( u_{\mathbf r}'(t_0)-v_{\mathbf r}'(t_0) \big).
\end{align*}
Therefore,
\begin{align*}
\big| \Delta_{\mathbf r}'(t_0) \big|
\leq
M_{\ell,12} \big| u_{\mathbf r}(t_0)-v_{\mathbf r}(t_0) \big|
+ M_{\ell,22} \big| u_{\mathbf r}(t_0)-v_{\mathbf r}(t_0) \big|\, \big| v_{\mathbf r}'(t_0) \big|
+ M_{\ell,2} \big| u_{\mathbf r}'(t_0)-v_{\mathbf r}'(t_0) \big|.
\end{align*}

By Lemma \ref{lemma: first order derivative of pathwise difference of iterates},
\[
\big| u_{\mathbf r}(t_0)-v_{\mathbf r}(t_0) \big|
\leq
\mathrm{artanh}(\rho)\,\rho^{n-1},
\]
and
\[
\big| u_{\mathbf r}'(t_0)-v_{\mathbf r}'(t_0) \big|
\leq
\left(
\frac{M_f^{(1)}}{1-\rho^2}
+ \frac{n-1}{\rho}\,L_g\,\mathrm{artanh}(\rho)
+ \frac{\Xi\,\mathrm{artanh}(\rho)}{1-\rho^2}
\right)\rho^{n-1}.
\]
We also have $\big| v_{\mathbf r}'(t_0) \big| \leq \Xi$ by Lemma \ref{lemma: pathwise first order derivative bound}. Thus,
\begin{align}
&\big| \Delta_{\mathbf r}'(t_0) \big| \\
&\leq
\left(
\Big(
M_{\ell,12} + \Xi M_{\ell,22}
\Big) \mathrm{artanh}(\rho) + \left(
\frac{M_f^{(1)}}{1-\rho^2}
+ \frac{n}{\rho}\,L_g\,\mathrm{artanh}(\rho)
+ \frac{\Xi\,\mathrm{artanh}(\rho)}{1-\rho^2}
\right) M_{\ell,2}
\right)\rho^{n-1}. \label{eq: derivative bound for Delta sharp}
\end{align}

Therefore, using Lemma \ref{lemma: bounding the derivative of path coefficient}, equations \eqref{eq: zeroth order Delta bound at t0}, \eqref{eq: derivative bound for Delta sharp}, and $\sum_{\mathbf r\in C^{n+1}} p_{t_0}(\mathbf r) = 1$,
\begin{align*}
|a_n'(t_0)|
&\leq
\sum_{\mathbf r\in C^{n+1}}
\left(
|p_{t_0}'(\mathbf r)|\,|\Delta_{\mathbf r}(t_0)| + |p_{t_0}(\mathbf r)|\,|\Delta_{\mathbf r}'(t_0)|
\right) \\
&\leq
\Bigg\{
n \times \left( \frac{1}{\rho} L_g + M_Q^{(1)} \right) M_{\ell,2} \mathrm{artanh}(\rho)
\\
&\hspace{50pt}+
\left(
M_{\Sigma\pi}^{(1)} M_{\ell,2} + M_{\ell,12} + \Xi M_{\ell,22}
+\frac{\Xi\,M_{\ell,2}}{1-\rho^2}
\right)
\mathrm{artanh}(\rho)
+ \frac{M_f^{(1)} M_{\ell,2}}{1-\rho^2}
\Bigg\} \rho^{n-1}.
\end{align*}
The tail bound follows by summing the geometric-arithmetic series, using termwise differentiation ensured by the compact convergence of the holomorphic series $\sum a_n(z)$.
\end{proof}

\begin{lemma} \label{lemma: first order derivative of full operator}
For any $s\in C$ and $n,k\geq 1$,
\begin{equation*}
\left|
\left.\frac{d}{dt}\right|_{t=t_0}\;\sum_{\ell=0}^{n-1}\big(\mathrm T(t)^\ell \boldsymbol v(t) \big)_{s,k}
\right|
\leq
\frac{2M_\pi^{(1)} + nM_Q^{(1)}}{m_\pi}\frac{\rho^k}{k}
\;+\; \Xi(1-\rho^n) \rho^{k-1}.
\end{equation*}
\end{lemma}

\begin{proof}
By equation \eqref{eq: analytic parameter Sn-identity},
\begin{align*}
&\left|
\left.\frac{d}{dt}\right|_{t=t_0}\;\sum_{\ell=0}^{n-1}\big(\mathrm T(t)^\ell \boldsymbol v(t) \big)_{s,k}
\right|
=
\left|
\left.\frac{d}{dt}\right|_{t=t_0}\big(S_n(t)\big)_{s,k}
\right|
=
\left|
- \left.\frac{d}{dt}\right|_{t=t_0} \sum_{\mathbf r\in C^n}c_{\mathbf r,s}(t)\frac{\big(-f_{\mathbf r}^{(n)}[t](0)\big)^k}{k}
\right| \\
&\leq
\sum_{\mathbf r\in C^n} \left| c_{\mathbf r,s}'(t_0) \right| \frac{\big|-f_{\mathbf r}^{(n)}[t_0](0)\big|^k}{k}
+\sum_{\mathbf r\in C^n}c_{\mathbf r,s}(t_0) \frac1k\, \left| \left.\frac{d}{dt}\right|_{t=t_0} \left(-f_{\mathbf r}^{(n)}[t](0)\right)^k \right|
\end{align*}
Note that $f_{\mathbf r}^{(n)}[t_0](0) \in [-\rho,\rho]$. By $\sum_{\mathbf r\in C^n} c_{\mathbf r,s}(t_0) = 1$, Lemma \ref{lemma: pathwise first order derivative bound} and Lemma \ref{lemma: bounding the derivative of path coefficient},

\[
\left|
\left.\frac{d}{dt}\right|_{t=t_0}\;\sum_{\ell=0}^{n-1}\big(\mathrm T(t)^\ell \boldsymbol v(t) \big)_{s,k}
\right|
\leq
\frac{2M_\pi^{(1)} + nM_Q^{(1)}}{m_\pi}\frac{\rho^k}{k}
\;+\; \Xi(1-\rho^n) \rho^{k-1}. \qedhere
\]
\end{proof}

\subsection{First-order derivative \RN 2: integral operators and final bounds} \label{First-order derivative II: integral operators and final bounds}

Let $J_1 = \Omega \cap \mathbb R$. For $r\in C$, $m\geq 2$, and $t\in J_1$, we define
$
\pazocal{L}_r(t),\; \pazocal{R}_{r,m}(t),\; \pazocal{K}_{r,m}(t) : \;
\pazocal{H}(\mathbb D_{\bar{\rho}}) \to \pazocal{H}(\mathbb D_{\bar{\rho}})
$
in the same way as in subsection \ref{subsection: Integral operators and telescoping}:
\begin{align*}
\big(\pazocal{L}_r(t) h\big)(w)
&\;\coloneqq\;
\frac{1}{2\pi i} \oint_{\partial \mathbb D_{\bar{\rho}}} \frac{1}{1-z}\, h\left( f_r[t]\!\left( \frac{w}{z} \right) \right)\,dz, \\[4pt]
\big(\pazocal{R}_{r,m}(t) h\big)(w)
&\;\coloneqq\;
\frac{1}{2\pi i} \oint_{\partial \mathbb D_{\bar{\rho}}} \frac{z^{m-1}}{1-z}\, h\left( f_r[t]\!\left( \frac{w}{z} \right) \right)\,dz, \\[4pt]
\pazocal{K}_{r,m}(t) h
&\;\coloneqq\;
\pazocal{L}_r(t) h - \pazocal{R}_{r,m}(t) h.
\end{align*}

For $m\geq 2$, let
\[
\bar{\rho}_m = \bar{\rho}^{\,m-1} K(\bar{\rho}),
\qquad \text{where} \qquad
K(\bar{\rho}) = \frac{1}{2\pi} \oint_{\partial \mathbb D_{\bar{\rho}}} \frac{|dz|}{|1-z|}.
\]
Also,
\begin{equation} \label{eq: definition of delta the first order bound}
\delta \coloneqq \frac{2M_f^{(1)}}{\bar{\rho}-\rho}.
\end{equation}
In what follows, we use the oscillation seminorm on $\pazocal{H}(\mathbb D_{\bar{\rho}})$ defined by
\[
\| h \|_{\mathrm{osc}}
\;\coloneqq\;
\sup_{u,v \in \mathbb D_{\bar{\rho}}} |h(u)-h(v)|.
\]
\begin{lemma} \label{lemma: first order derivative of LRK}
Suppose $h\in \pazocal H(\mathbb D_{\bar{\rho}})$ satisfies $h(0) = 0$ and $\|h\|_{\mathrm{osc}} < \infty$. Then for each $r\in C$,
\[
\left.\partial_t\right|_{t=t_0} \big( \pazocal L_r(t)h(0) \big)
= \left.\partial_t\right|_{t=t_0}\big( \pazocal R_{r,m}(t)h(0) \big)
= \left.\partial_t\right|_{t=t_0} \big( \pazocal K_{r,m}(t)h(0) \big)
=0.
\]
Furthermore,
\begin{align*}
&\left\| \left.\partial_t\right|_{t=t_0} \big( \pazocal L_r(t)h \big) \right\|_{\mathrm{osc}}
\leq \delta \|h\|_{\mathrm{osc}}, 
\qquad
\left\| \left.\partial_t\right|_{t=t_0} \big( \pazocal R_{r,m}(t)h \big) \right\|_{\mathrm{osc}}
\leq \delta \bar{\rho}_m \|h\|_{\mathrm{osc}},
\\
&\left\| \left.\partial_t\right|_{t=t_0} \big( \pazocal K_{r,m}(t)h \big) \right\|_{\mathrm{osc}}
\leq \delta (1+\bar{\rho}_m)\|h\|_{\mathrm{osc}}.
\end{align*}
\end{lemma}
\begin{proof}
Suppose $h$ satisfies $h(0) = 0$ and $\|h\|_{\mathrm{osc}} < \infty$. Since
\[
\pazocal L_r(t)h(w)
\;=\;
h\big( f_r[t](w) \big) - h\big( f_r[t](0) \big),
\]
we have $\pazocal L_r(t)h(0) = 0$ for all $t\in J_1$. Thus $\left.\partial_t\right|_{t=t_0} \big( \pazocal L_r(t)h(0) \big) = 0$. Next,
\[
\partial_t \big( \pazocal L_r(t)h(w) \big)
=
h'\big( f_r[t](w) \big)\,\partial_t f_r[t](w)
-
h'\big( f_r[t](0) \big)\,\partial_t f_r[t](0).
\]
Since $f_r[t_0](\xi)\in \mathbb D_\rho$ for each $\xi\in \mathbb D_1$, the circle centered at $f_r[t_0](\xi)$ with radius $\bar{\rho} - \rho$ is contained in $\mathbb D_{\bar{\rho}}$. Thus, by Cauchy's derivative formula and $h(0)=0$, for every $\xi\in \mathbb D_1$,
\begin{equation} \label{eq: bound for derivative of h}
\left| h'\big( f_r[t_0](\xi) \big) \right|
\leq
\frac{1}{\bar{\rho}-\rho}\sup_{u\in \mathbb D_{\bar{\rho}}} |h(u)|
\leq
\frac{1}{\bar{\rho}-\rho}\|h\|_{\mathrm{osc}}.
\end{equation}
Then,
\begin{align*}
\left\|
\left.\partial_t\right|_{t=t_0} \big( \pazocal L_r(t)h \big)
\right\|_{\mathrm{osc}}
&\leq
\sup_{u,v\in \mathbb D_{\bar{\rho}}}
\left(
\Big| h'\big( f_r[t_0](u) \big)\!\left.\partial_t\right|_{t=t_0} f_r[t](u) \Big|
+
\Big| h'\big( f_r[t_0](v) \big)\!\left.\partial_t\right|_{t=t_0} f_r[t](v) \Big|
\right) \\
&\leq
2\frac{\|h\|_{\mathrm{osc}}}{\bar{\rho}-\rho}M_f^{(1)}
=\delta \|h\|_{\mathrm{osc}}.
\end{align*}

For $\pazocal R_{r,m}(t)$, since we can exchange the derivative and the integral, for any $w\in \mathbb D_{\bar{\rho}}$,
\[
\partial_t \big( \pazocal R_{r,m}(t)h(w) \big)
=
\frac{1}{2\pi i}
\oint_{\partial \mathbb D_{\bar{\rho}}}
\frac{z^{m-1}}{1-z}\,
h'\left( f_r[t]\!\left( \frac{w}{z} \right) \right)
\partial_t f_r[t]\!\left( \frac{w}{z} \right)\,dz.
\]
Taking $w=0$, we obtain $\left.\partial_t\right|_{t=t_0} \big( \pazocal R_{r,m}(t)h(0) \big)=0$ by Cauchy's integral theorem.

Now we evaluate the differentiated formula at $t=t_0$. Note that $w/z \in \mathbb D_1$ for every $w\in \mathbb D_{\bar{\rho}}$ and $z\in \partial \mathbb D_{\bar{\rho}}$. Therefore, using equation \eqref{eq: bound for derivative of h},
\begin{align*}
\left|
\left.\partial_t\right|_{t=t_0} \big( \pazocal R_{r,m}(t)h(w) \big)
\right|
&\leq
\frac{1}{2\pi} \oint_{\partial \mathbb D_{\bar{\rho}}} \frac{|z|^{m-1}}{|1-z|}\,
\left| h'\left( f_r[t_0]\!\left( \frac{w}{z} \right) \right) \right|
\left| \left.\partial_t\right|_{t=t_0} f_r[t]\!\left( \frac{w}{z} \right) \right|\,|dz| \\
&\hspace{-30pt}\leq
\frac{M_f^{(1)}}{\bar{\rho}-\rho}\, \|h\|_{\mathrm{osc}}\,
\frac{1}{2\pi} \oint_{\partial \mathbb D_{\bar{\rho}}} \frac{\bar{\rho}^{m-1}}{|1-z|}\,|dz|
=
\frac{M_f^{(1)}}{\bar{\rho}-\rho}\, \bar{\rho}^{m-1} K(\bar{\rho}) \|h\|_{\mathrm{osc}}
=
\frac{\delta}{2} \bar{\rho}_m \|h\|_{\mathrm{osc}}.
\end{align*}
This implies
$
\left\|
\left.\partial_t\right|_{t=t_0} \big( \pazocal R_{r,m}(t)h \big)
\right\|_{\mathrm{osc}}
\leq
\delta \bar{\rho}_m \|h\|_{\mathrm{osc}}
$. The assertion for $\pazocal K_{r,m}$ follows with the triangle inequality.
\end{proof}

\begin{lemma} \label{lemma: first order derivative of difference of LK compositions}
For $n\geq 1$ and a word $(r_1,\dots,r_n)\in C^n$, define
\[
\pazocal D_{r_1,\dots,r_n}^{(m)}(t)
\;=\;
\pazocal L_{r_1}(t)\cdots \pazocal L_{r_n}(t)
- \pazocal K_{r_1,m}(t)\cdots \pazocal K_{r_n,m}(t).
\]
Suppose $h\in \pazocal H(\mathbb D_{\bar{\rho}})$ satisfies $h(0) = 0$ and $\|h\|_{\mathrm{osc}} < \infty$. Then, for every $w \in \mathbb D_{\bar{\rho}}$,
\begin{equation*}
\left|
\left.\partial_t\right|_{t=t_0} \Big( \pazocal D_{r_1,\dots,r_n}^{(m)}(t)h(w) \Big)
\right|
\;\leq\;
n\,\delta\,\Big( (1+\bar{\rho}_m)^n - 1 \Big)\,\|h\|_{\mathrm{osc}}.
\end{equation*}
\end{lemma}

\begin{proof}
Fix $n\geq 1$ and a word $(r_1,\dots,r_n)\in C^n$. By Lemma \ref{lemma: telescoping of L and K}, for every $t\in J_1$,
\begin{equation*}
\pazocal D_{r_1,\dots,r_n}^{(m)}(t)
=
\sum_{\ell=1}^n \pazocal U_\ell(t),
\quad \text{where} \quad
\pazocal U_\ell(t)
=
\pazocal L_{r_1}(t)\cdots \pazocal L_{r_{\ell-1}}(t)\,
\pazocal R_{r_\ell, m}(t)\,
\pazocal K_{r_{\ell+1}, m}(t)\cdots \pazocal K_{r_n, m}(t).
\end{equation*}

It suffices to show that for every $1\leq \ell\leq n$ and $w \in \mathbb D_{\bar{\rho}}$,
\begin{equation} \label{eq: single U derivative bound}
\left|
\left.\partial_t\right|_{t=t_0} \big( \pazocal U_\ell(t) h(w) \big)
\right|
\leq
n\,\delta\,\bar{\rho}_m\,(1+\bar{\rho}_m)^{\,n-\ell}\,\|h\|_{\mathrm{osc}}.
\end{equation}
Indeed, summing the right-hand side over $\ell$ yields the desired bound.

We now fix $\ell\in\{1,\dots,n\}$. For $1\leq j\leq n$, define
\[
\pazocal B_j(t)
=
\begin{dcases}
\pazocal L_{r_j}(t) & \text{if } 1\leq j\leq \ell-1, \\
\pazocal R_{r_\ell,m}(t) & \text{if } j=\ell, \\
\pazocal K_{r_j,m}(t) & \text{if } \ell+1\leq j\leq n.
\end{dcases}
\]
Then $\pazocal U_\ell(t) = \pazocal B_1(t)\cdots \pazocal B_n(t)$. For each $j\in\{1,\dots,n\}$, set
\[
g_j
\;=\;
\pazocal B_{j+1}(t_0)\cdots \pazocal B_n(t_0)h,
\]
with the convention that $g_n=h$. Since each of the operators $\pazocal L_r(t_0)$, $\pazocal R_{r,m}(t_0)$, and $\pazocal K_{r,m}(t_0)$ maps functions vanishing at $0$ to functions vanishing at $0$, we have $g_j(0)=0$ for every $1\leq j\leq n$.

By the same telescoping argument of Lemma \ref{lemma: telescoping of L and K}, and by the definition of $g_j$,
\begin{align*}
\pazocal U_\ell(t)h - \pazocal U_\ell(t_0)h
&=
\sum_{j=1}^n
\pazocal B_1(t)\cdots \pazocal B_{j-1}(t)
\big( \pazocal B_j(t)g_j -\pazocal B_j(t_0)g_j \big)
\quad \text{for each } \; t\in J_1 \setminus \{t_0\}.
\end{align*}
We divide by $t-t_0$ and let $t\to t_0$. By the continuity and differentiability at $t_0$ of each $\pazocal B_i(t)u$ for any $u\in \pazocal H(\mathbb D_{\bar{\rho}})$, we obtain, for every $w\in \mathbb D_{\bar{\rho}}$,
\begin{align}
\left.\partial_t \right|_{t=t_0} \big( \pazocal U_\ell(t)h \big)(w)
&=
\sum_{j=1}^n
\left(
\pazocal B_1(t_0)\cdots \pazocal B_{j-1}(t_0)
	\left(
	\left.\partial_t \right|_{t=t_0} \big( \pazocal B_j(t)\,g_j \big)
	\right)
\right)(w) \nonumber \\
&=
\sum_{j=1}^n
\left(
\pazocal B_1(t_0)\cdots \pazocal B_{j-1}(t_0)
	\left(
	\left.\partial_t \right|_{t=t_0} \big( \pazocal B_j(t)\pazocal B_{j+1}(t_0)\cdots \pazocal B_n(t_0)h\big)
	\right)
\right)(w). \label{eq: derivative of U operator}
\end{align}

Now, by the same proof as in Lemma \ref{lemma: basic properties of LRK}, now on $\mathbb D_{\bar{\rho}}$, we have for $u\in \pazocal H(\mathbb D_{\bar{\rho}})$ and $t\in J_1$,
\[
\|\pazocal L_r(t)u\|_{\mathrm{osc}} \leq \|u\|_{\mathrm{osc}},
\qquad
\|\pazocal R_{r,m}(t)u\|_{\mathrm{osc}} \leq \bar{\rho}_m \|u\|_{\mathrm{osc}},
\qquad
\|\pazocal K_{r,m}(t)u\|_{\mathrm{osc}} \leq (1+\bar{\rho}_m)\|u\|_{\mathrm{osc}}.
\]
Note that every operator appearing above sends functions vanishing at $0$ to functions vanishing at $0$ by Lemma \ref{lemma: first order derivative of LRK}. Combining with the bound in Lemma \ref{lemma: first order derivative of LRK}, we see that wherever the derivative is at, it yields an extra $\delta$ compared to the usual bound above. Then, each term in the summand of equation \eqref{eq: derivative of U operator} is bounded by
\[
\delta\,\bar{\rho}_m\,(1+\bar{\rho}_m)^{\,n-\ell}\,\|h\|_{\mathrm{osc}}.
\]
Since there are $n$ such terms, we obtain equation \eqref{eq: single U derivative bound}.
\end{proof}

Let $\beta(0) = 0$, and for $n\geq 1$,
\[
\beta(n)
=
2\left(
\frac{2M_\pi^{(1)} + (n+1)M_Q^{(1)}}{m_\pi}
+
\frac{2n+1}{\bar{\rho}-\rho}M_f^{(1)}
\right).
\]

\begin{lemma} \label{lemma: first order derivative of difference of iterates}
Let $n\geq 0$, $m\geq 2$, and $s\in C$. For every $1\leq k\leq m-1$,
\begin{equation*}
\left| \left.\frac{d}{dt}\right|_{t=t_0}
\bigg(
\mathrm T(t)^n \boldsymbol v(t) - \mathrm T_{m\times m}(t)^n \boldsymbol v(t)
\bigg)_{\!\!s,k\,}
\right|
\;\leq\;
\beta(n) \,\Big( (1+\bar{\rho}_m)^n - 1 \Big)\frac{\bar{\rho}^k}{k}.
\end{equation*}
\end{lemma}

\begin{proof}
The bound is immediate for $n=0$, so let $n\geq 1$. For $k\geq 1$, define $h_k\in \pazocal H(\mathbb D_{\bar{\rho}})$ by $h_k(w) = -\frac{(-w)^k}{k}$. Then $h_k(0)=0$, and $\|h_k\|_{\mathrm{osc}} \leq \frac{2{\bar{\rho}}^k}{k}$. The same path expansion and operator formulas used in the proof of Proposition \ref{prop: evaluating the truncation} apply verbatim for each fixed $t\in J_1$, and we have
\begin{equation*}
\bigg( 
\mathrm T(t)^n \boldsymbol v(t) - \mathrm T_{m\times m}(t)^n \boldsymbol v(t)
\bigg)_{s,k}
=
\sum_{\mathbf r=(r_0,\dots,r_n)\in C^{n+1}} c_{\mathbf r,s}(t)\,
\Big( \pazocal D_{r_1,\dots,r_n}^{(m)}(t) h_k \Big) \big( f_{r_0}[t](0) \big).
\end{equation*}

We differentiate this equation at $t=t_0$. There are two contributions.

\smallskip

\emph{(i) The derivative hits the path coefficient.} First note that by the same oscillation argument as in the proof of Proposition \ref{prop: evaluating the truncation}, now with $\bar{\rho}$, we obtain
\begin{equation} \label{eq: undifferentiated telescoping bound for D}
\left| \Big( \pazocal D_{r_1,\dots,r_n}^{(m)}(t_0) h_k \Big)(w) \right|
\leq
\Big( (1+\bar{\rho}_m)^n - 1 \Big)\frac{2{\bar{\rho}}^k}{k}
\quad \text{for } \; w\in \mathbb D_{\bar{\rho}}.
\end{equation}
Using Lemma \ref{lemma: bounding the derivative of path coefficient} and equation \eqref{eq: undifferentiated telescoping bound for D} yields
\begin{equation} \label{eq: the derivative hits path coefficient in lemma: first order derivative of difference of iterates}
\sum_{\mathbf r\in C^{n+1}} |c_{\mathbf r,s}'(t_0)|
\left|
\Big( \pazocal D_{r_1,\dots,r_n}^{(m)}(t_0) h_k \Big) \big( f_{r_0}[t_0](0) \big)
\right|
\leq
\frac{2M_\pi^{(1)} + (n+1)M_Q^{(1)}}{m_\pi}
\,\Big( (1+\bar{\rho}_m)^n - 1 \Big)\frac{2{\bar{\rho}}^k}{k}.
\end{equation}

\smallskip

\emph{(ii) The derivative hits the operator word.} First, let $\pazocal F(t,w) = \Big( \pazocal D_{r_1,\dots,r_n}^{(m)}(t) h_k \Big)(w)$. Then, 
\begin{align*}
&\left.\frac{d}{dt}\right|_{t=t_0} \Big( \pazocal D_{r_1,\dots,r_n}^{(m)}(t) h_k \Big) \big( f_{r_0}[t](0) \big)
\;=\; \left.\frac{d}{dt}\right|_{t=t_0}\pazocal F\big( t,f_{r_0}[t](0) \big) \\
%\;=\; \partial_1\pazocal F\big(t_0,f_{r_0}[t_0](0)\big) \;+\; \partial_2\pazocal F\big(t_0,f_{r_0}[t_0](0)\big) \times \left.\frac{d}{ds}\right|_{s=t_0} f_{r_0}[s](0)\\
&=\; \left.\frac{d}{dt}\right|_{t=t_0} \!\!\Big( \pazocal D_{r_1,\dots,r_n}^{(m)}(t) h_k \Big) \big( f_{r_0}[t_0](0) \big)
+
\left.\frac{d}{dw}\right|_{w=f_{r_0}[t_0](0)} \!\!\!\Big( \pazocal D_{r_1,\dots,r_n}^{(m)}(t_0) h_k \Big)(w)\left.\frac{d}{dt}\right|_{t=t_0} f_{r_0}[t](0).
\end{align*}

Since $f_{r_0}[t_0](0)\in \mathbb D_\rho$, the circle centered at $f_{r_0}[t_0](0)$ with radius $\bar{\rho} - \rho$ is contained in $\mathbb D_{\bar{\rho}}$. Thus, by Cauchy's derivative formula,
\[
\left|
\left.\frac{d}{dw}\right|_{w=f_{r_0}[t_0](0)} \!\!\!\Big( \pazocal D_{r_1,\dots,r_n}^{(m)}(t_0) h_k \Big)(w)
\right|
\leq
\frac{1}{\bar{\rho}-\rho} \sup_{w\in \mathbb D_{\bar{\rho}}}
\left| \Big( \pazocal D_{r_1,\dots,r_n}^{(m)}(t_0) h_k \Big)(w) \right|
\]
Combining with equation \eqref{eq: undifferentiated telescoping bound for D} and Lemma \ref{lemma: first order derivative of difference of LK compositions}, we obtain
\begin{align*}
\left|
\left.\frac{d}{dt}\right|_{t=t_0} \Big( \pazocal D_{r_1,\dots,r_n}^{(m)}(t) h_k \Big) \big( f_{r_0}[t](0) \big)
\right|
\leq
\left( n\,\delta + \frac{M_f^{(1)}}{\bar{\rho}-\rho} \right) \Big( (1+\bar{\rho}_m)^n - 1 \Big) \frac{2{\bar{\rho}}^k}{k}.
\end{align*}
This and equation \eqref{eq: the derivative hits path coefficient in lemma: first order derivative of difference of iterates} yield the desired inequality, using the definition of $\delta$ in equation \eqref{eq: definition of delta the first order bound} and $\sum c_{\mathbf r, s}(t_0) = 1$ by equation \eqref{eq: path coefficient sums to 1}.
\end{proof}

\bigskip

We are now ready to prove Theorem \ref{theorem: polynomial time calculation of every derivative} for the first derivative while obtaining a bound that does not depend on the radius of the complex disk centered at $t_0$ and contained in $\Omega$.

\begin{proof}[Proof of Theorem \ref{theorem: polynomial time calculation of every derivative} for $q=1$]
Let $D = \max_{r \in C} \left| f_r^\top[t_0](0) \right|$ and $\rho_m = \rho^{m-1} K(\rho)$. Let $n,m \geq 2$. First, let us bound $\Lambda_n'(t_0) - \Lambda_{n,m}'(t_0)$.
\begin{flalign*}
\Lambda_n'(t_0) - \Lambda_{n,m}'(t_0)
&=
\frac1d \left.\frac{d}{dt}\right|_{t=t_0} \;\sum_{\ell=0}^{n-1}
\left\{ \sum_{s\in C} \pi_s(t) \left(\mathrm T(t)^\ell \boldsymbol v(t)-\mathrm T_{m\times m}(t)^\ell \boldsymbol v(t)\right)_{\!\!s,0} \right\} &\\
&=
\frac1d \Bigg( \;\; \underbrace{
\sum_{\ell=0}^{n-1}
\left\{
\sum_{s\in C} \pi_s'(t_0)
\left(\mathrm T(t_0)^\ell \boldsymbol v(t_0)-\mathrm T_{m\times m}(t_0)^\ell \boldsymbol v(t_0)\right)_{\!\!s,0}
\right\}
}_{=(\RN 1)}\\
&\hspace{80pt}+
\underbrace{
\sum_{s\in C} \pi_s(t_0) \left.\frac{d}{dt}\right|_{t=t_0} \sum_{\ell=0}^{n-1} \left(\mathrm T(t)^\ell \boldsymbol v(t)-\mathrm T_{m\times m}(t)^\ell \boldsymbol v(t)\right)_{\!\!s,0}
}_{=(\RN 2)} \;\; \Bigg).
\end{flalign*}

Let us start with the term $(\RN 1)$. For each $s\in C$ and $1\leq \ell \leq n-1$, the following bound is obtained with the same argument as Proposition \ref{prop: evaluating the truncation}.

\begin{align*}
&\left| \left(\mathrm T(t_0)^\ell \boldsymbol v(t_0)-\mathrm T_{m\times m}(t_0)^\ell \boldsymbol v(t_0)\right)_{\!\!s,0} \right| \\
&\leq
\sum_{r\in C} c_{r,s}(t_0)
\Bigg\{
\sum_{k=1}^{m-1}\left|f_r^\top[t_0](0)\right|^k
\left| \left(\mathrm T(t_0)^{\ell-1} \boldsymbol v(t_0)-\mathrm T_{m\times m}(t_0)^{\ell-1} \boldsymbol v(t_0)\right)_{\!\!r,\,k} \right|
\\
&\hspace{220pt}+
\sum_{k=m}^\infty \left|f_r^\top[t_0](0)\right|^k
\left| \left(\mathrm T(t_0)^{\ell-1} \boldsymbol v(t_0)\right)_{\!\!r,\,k} \right|
\Bigg\}
\\
&\leq
\sum_{k=1}^{m-1}D^k \frac{2\bar{\rho}^k}{k}\,\Big( (1+\bar{\rho}_m)^{\ell-1} - 1 \Big) + \sum_{k=m}^\infty D^k\frac{2\rho^k}{k} \\
&\leq
2\left(\log\frac{1}{1-\bar{\rho} D}\right)\,\Big( (1+\bar{\rho}_m)^{\ell-1} - 1 \Big) + \frac{2(\rho D)^m}{m(1-\rho D)}.
\end{align*}

Now, note that we have $\sum_{s\in C}\left|\pi_s'(t_0)\right| \leq M_{\Sigma\pi}^{(1)}$. Thus, we have
\begin{align*}
|(\RN 1)|
\leq
M_{\Sigma\pi}^{(1)} \left\{ 2\left(\log\frac{1}{1-\bar{\rho} D}\right)\bar{\alpha}_{n,m} + \frac{2(n-1)(\rho D)^m}{m(1-\rho D)} \right\},
\end{align*}
where
\[
\bar{\alpha}_{n,m}
=
\sum_{\ell=1}^{n-1}\Big( (1+\bar{\rho}_m)^{\ell-1} - 1 \Big)
=
\frac{(1+\bar{\rho}_m)^{n-1} - 1 - (n-1)\bar{\rho}_m}{\bar{\rho}_m}.
\]

\medskip

We move on to the term $(\RN 2)$. First,
\begin{align*}
&\left| \left.\frac{d}{dt}\right|_{t=t_0} \sum_{\ell=0}^{n-1} \left(\mathrm T(t)^\ell \boldsymbol v(t)-\mathrm T_{m\times m}(t)^\ell \boldsymbol v(t)\right)_{\!\!s,0} \right| \\
&\leq
\underbrace{
\sum_{\ell=1}^{n-1}
\left|
\left.\frac{d}{dt}\right|_{t=t_0}
\left\{
\sum_{r\in C} c_{r,s}(t)
\sum_{k=1}^{m-1} \big(f_r^\top[t](0)\big)^k
\left(\mathrm T(t)^{\ell-1} \boldsymbol v(t)-\mathrm T_{m\times m}(t)^{\ell-1} \boldsymbol v(t)\right)_{\!\!r,\,k}
\right\} \right|
}_{(\RN 3)}\\
&\hspace{120pt}+
\underbrace{
\left|
\left.\frac{d}{dt}\right|_{t=t_0}
\sum_{r\in C} c_{r,s}(t)
\sum_{k=m}^\infty \big(f_r^\top[t](0)\big)^k
\sum_{\ell=1}^{n-1}
\left(\mathrm T(t)^{\ell-1} \boldsymbol v(t)\right)_{\!\!r,\,k}
\right|
}_{(\RN 4)}.
\end{align*}
For the term $(\RN 3)$, there are three kinds of terms.

\smallskip

\emph{(i) The derivative hits the path coefficient $c_{r,s}(t)$.} Note that $\rho_m \leq \bar{\rho}_m$. Using Lemma \ref{lemma: bounding the derivative of path coefficient}, such terms are bounded from above by
\[
\sum_{\ell=1}^{n-1}
\frac{2M_\pi^{(1)} + M_Q^{(1)}}{m_\pi}
\sum_{k=1}^{m-1} D^k \frac{2\rho^k}{k}\,\Big( (1+\rho_m)^{\ell-1} - 1 \Big)
\leq
\frac{2M_\pi^{(1)} + M_Q^{(1)}}{m_\pi} \times 2\left(\log\frac{1}{1-\rho D}\right) \bar{\alpha}_{n,m}.
\]

\emph{(ii) The derivative hits $\big(f_r^\top[t](0)\big)^k$.} Such terms are bounded by
\[
\sum_{\ell=1}^{n-1} \sum_{k=1}^{m-1} kD^{k-1} M_\top^{(1)} \frac{2\rho^k}{k}\,\Big( (1+\rho_m)^{\ell-1} - 1 \Big)
\leq
\frac{2\rho M_\top^{(1)}\bar{\alpha}_{n,m}}{1-\rho D}.
\]

\emph{(iii) The derivative hits the difference of full and truncated operators.} Note that $\beta(\ell)$ is increasing in $\ell$. By Lemma \ref{lemma: first order derivative of difference of iterates}, such terms are bounded by
\[
\sum_{\ell=1}^{n-1} \sum_{k=1}^{m-1} D^k \beta(\ell-1) \,\Big( (1+\bar{\rho}_m)^{\ell-1} - 1 \Big)\frac{\bar{\rho}^k}{k}
\leq
\left(\log\frac{1}{1-\bar{\rho} D}\right) \beta(n-2) \bar{\alpha}_{n,m}.
\]

Therefore, the term $(\RN 3)$ is bounded as
\begin{align*}
\big| (\RN 3) \big|
\leq
\left\{
\left(
\beta(n-2) + \frac{4M_\pi^{(1)} + 2M_Q^{(1)}}{m_\pi}
\right)
\left(\log\frac{1}{1-\bar{\rho} D}\right)
+ \frac{2\rho M_\top^{(1)}}{1-\rho D}
\right\} \bar{\alpha}_{n,m}.
\end{align*}

\bigskip

Let us next bound the term $(\RN 4)$. First note that, by equation \eqref{eq: analytic parameter Sn-identity},
\[
\sum_{\ell=1}^{n-1}
\left(\mathrm T(t)^{\ell-1} \boldsymbol v(t)\right)_{\!\!r,\,k}
=
\left( S_{n-1}(t) \right)_{r,\,k}.
\]
By the definition of $S_n(t)$ in equation \eqref{eq: analytic parameter definition of Sn},
\[
\left| (S_{n-1}(t_0))_{r,\,k}\right| \leq \frac{\rho^k}{k}.
\]

\emph{(i) The derivative hits the path coefficient $c_{r,s}(t)$.} Using Lemma \ref{lemma: bounding the derivative of path coefficient}, such terms are bounded from above by
\[
\frac{2M_\pi^{(1)} + M_Q^{(1)}}{m_\pi} \sum_{k=m}^\infty D^k \frac{\rho^k}{k}
\;\leq\;
\frac{2M_\pi^{(1)} + M_Q^{(1)}}{m_\pi} \frac{(\rho D)^m}{m(1-\rho D)}.
\]

\emph{(ii) The derivative hits $\big(f_r^\top[t](0)\big)^k$.} Such terms are bounded by
\[
\sum_{k=m}^\infty kD^{k-1} M_\top^{(1)} \frac{\rho^k}{k}
\;\leq\;
\frac{\rho M_\top^{(1)}(\rho D)^{m-1}}{1-\rho D}.
\]

\emph{(iii) The derivative hits the partial sum of the genuine operator $\mathrm T(t)$.} By Lemma \ref{lemma: first order derivative of full operator}, such terms are bounded by
\begin{align*}
&\sum_{k=m}^\infty D^k
\left(
\frac{2M_\pi^{(1)} + (n-1)M_Q^{(1)}}{m_\pi} \frac{\rho^k}{k}
+ \rho^{k-1} \frac{M_f^{(1)}}{1-\rho^2}\, \frac{1-\rho^{n-1}}{1-\rho}
\right) \\
&\leq\;
\left\{
\frac{2M_\pi^{(1)} + (n-1)M_Q^{(1)}}{m_\pi\cdot m}
\;+\; \frac{M_f^{(1)}(1-\rho^{n-1})}{\rho(1-\rho^2)(1-\rho)}
\right\} \frac{(\rho D)^m}{1-\rho D}.
\end{align*}

We can also see that the derivative series is uniformly summable around $t_0$ by the Weierstrass M-test, so termwise differentiation is allowed. Combining these bounds, the term $(\RN 4)$ is bounded as
\[
\big|(\RN 4)\big|
\leq
\left(
\frac{\rho D}{m} \cdot\frac{4M_\pi^{(1)} + n\,M_Q^{(1)}}{m_\pi} + \rho M_\top^{(1)} + \frac{M_f^{(1)}D(1-\rho^{n-1})}{(1-\rho^2)(1-\rho)} 
\right)
\frac{(\rho D)^{m-1}}{1-\rho D}.
\]

Therefore,
\begin{align*}
&\left| \Lambda_n'(t_0) - \Lambda_{n,m}'(t_0) \right|
\;\leq\;
\frac1d \Big( \big|(\RN 1)\big| + \big|(\RN 2)\big| \Big)\\
&\leq
\frac1d \left\{
\left(
\beta(n-2) + 2M_{\Sigma\pi}^{(1)} + \frac{4M_\pi^{(1)} + 2M_Q^{(1)}}{m_\pi}
\right)
\left(\log\frac{1}{1-\bar{\rho}D}\right) + \frac{2\rho M_\top^{(1)}}{1-\rho D}
\right\} \bar{\alpha}_{n,m} \\
&\hspace{20pt}+
\frac1d \left\{
\frac{\rho D}{m}
\left(
2(n-1)M_{\Sigma\pi}^{(1)} + \frac{4M_\pi^{(1)} + n\,M_Q^{(1)}}{m_\pi}
\right)
+\rho M_\top^{(1)} + \frac{M_f^{(1)}D(1-\rho^{n-1})}{(1-\rho^2)(1-\rho)} 
\right\}
\frac{(\rho D)^{m-1}}{1-\rho D}.
\end{align*}

Thus, by the triangle inequality and Lemma \ref{lemma: derivative tail bound},

\begin{align*}
&\left|
\lambda'(t_0) - \Lambda_{n,m}'(t_0)
\right|
\;\leq\;
\left| \lambda'(t_0) - \Lambda_n'(t_0) \right|
+
\left| \Lambda_n'(t_0) - \Lambda_{n,m}'(t_0) \right| \\
&\leq\;
\frac1d \left\{
\left( \frac{n}{1-\rho}+\frac{\rho}{(1-\rho)^2} \right)A_{\mathrm{tail}}
+ \frac{B_{\mathrm{tail}}}{1-\rho}
\right\} \rho^{n-1} \\
&\hspace{19pt}+
\frac1d \left\{
\left(
\beta(n-2) + 2M_{\Sigma\pi}^{(1)} + \frac{4M_\pi^{(1)} + 2M_Q^{(1)}}{m_\pi}
\right)
\left(\log\frac{1}{1-\bar{\rho}D}\right) + \frac{2\rho M_\top^{(1)}}{1-\rho D}
\right\} \bar{\alpha}_{n,m} \\
&\hspace{19pt}+
\frac1d \left\{
\frac{\rho D}{m}
\left(
2(n-1)M_{\Sigma\pi}^{(1)} + \frac{4M_\pi^{(1)} + n\,M_Q^{(1)}}{m_\pi}
\right)
+\rho M_\top^{(1)} + \frac{M_f^{(1)}D(1-\rho^{n-1})}{(1-\rho^2)(1-\rho)} 
\right\}
\frac{(\rho D)^{m-1}}{1-\rho D}.
\end{align*}

Notice that every term in the braces $\{ \}$ is at most polynomial in $n$. Then, by the same argument as in the proof of Theorem \ref{thm: complexity statement}, the above bound is less than arbitrary $\vep>0$ for some choices of $n=O\big(\log(1/\vep)\big)$ and $m=O\big(\log(n/\vep)\big)$. Since $\Lambda_{n,m}'(t_0)$ can be approximated with $O(m^3)+O(n\,m^2)$ arithmetic operations using equation \eqref{eq: inductive identity of finite approximation of first derivative of lambda}, we conclude that, for a given $\vep>0$, one can compute $\lambda'(t_0)$ to error $\vep$ in $O\big( (\log(1/\vep))^3 \big)$ arithmetic operations.
\end{proof}

\subsection{Higher-order derivatives} \label{subsection: Higher-order derivatives}

Recall that we are working under the hypotheses and notation of Theorem \ref{thm: analytic dependence on parameter} and Remark \ref{rmk: condition on domain of holomorphic extension}. Fix $c>0$ such that $U = \{z\in \mathbb C \,:\, |z-t_0|\leq c\} \subset \Omega$. Let
\begin{gather*}
\overline Q
\coloneqq
\max \Big\{ 1, \; 
\max_{r\in C} \max_{z\in U} \sum_{r'\in C} \big| Q_{r,r'}(z) \big| \Big\},
\qquad
\overline D
\coloneqq
\max_{r\in C} \max_{z\in U} \big| f_r^\top[z](0) \big|,
\\
\overline M_\ell
\;\coloneqq\;\;
\max_{r\in C}\!\!\max_{(z,w)\in U\times \closure{\mathbb D_{\bar{\rho}}}}
\big| \partial_2 \ell_r(z,w) \big|,
\qquad
\overline M_{\Sigma\pi}
\coloneqq
\max_{z\in U} \sum_{r\in C} \big| \pi_r(z) \big|,
\qquad
\overline m_\pi
\coloneqq
\min_{r\in C} \min_{z\in U}
|\pi_r(z)|.
\end{gather*}
We have $\overline D < 1$ by the assumption that $f_r[z]\big( \overline{\mathbb D_1} \big)\subset \mathbb D_{\bar \rho}$ for every $r\in C$ and $z\in \Omega$ and Lemma \ref{lemma: Mobius contraction implies transpose contraction}.

\begin{proposition} \label{prop: precise bounds for higher order derivatives}
Let $q\geq 1$. Under the setting above, for any natural numbers $n,m \geq 2$,
\begin{align*}
&\Big| \lambda^{\!\!(q)}(t_0) - \Lambda^{(q)}_{n,m}(t_0) \Big| \\
&\hspace{20pt}\leq
\frac{q!\overline Q^n}{c^q d} 
\Bigg\{
\frac{\overline M_{\Sigma\pi} \overline M_\ell \; \mathrm{artanh}(\bar{\rho})}{\bar{\rho}\big(1-\bar{\rho}\overline Q\big)} \bar{\rho}^n
+
\frac{2\overline M_{\Sigma\pi}^2}{\overline m_\pi} \left(
\bar{\alpha}_{n,m} \log \frac{1}{1-\bar{\rho} \overline D}
+ \frac{(n-1)\bar{\rho}^m \overline D^m}{m(1-\bar{\rho}\overline D)}
\right)
\Bigg\}.
\end{align*}
Here,
\[
\bar{\alpha}_{n,m}
=
\frac{(1+\bar{\rho}_m)^{n-1} - 1 - (n-1)\bar{\rho}_m}{\bar{\rho}_m},
\qquad
\bar{\rho}_m = \bar{\rho}^{m-1} K(\bar{\rho}).
\]
In particular, one can compute $\lambda^{\!\!(q)}(t_0)$ to error $\vep > 0$ with $O\big( (\log(1/\vep))^3 \big)$ arithmetic operations for each fixed $q$.
\end{proposition}

\begin{proof}
Since $\lambda(z) - \Lambda_{n,m}(z)$ is holomorphic on $\Omega$, by Cauchy's integral formula,
\begin{align*}
\lambda^{\!\!(q)}(t_0) - \Lambda^{(q)}_{n,m}(t_0)
=
\frac{q!}{2\pi i} \oint_{|z-t_0|=c}
\frac{1}{(z-t_0)^{q+1}} \Big( \lambda(z) - \Lambda_{n,m}(z) \Big) dz.
\end{align*}
Then,
\begin{equation} \label{eq: derivative lyapunov exponent is bounded by the original by analyticity}
\Big| \lambda^{\!\!(q)}(t_0) - \Lambda^{(q)}_{n,m}(t_0) \Big|
\leq
\frac{q!}{c^q} \max_{z\in U} \Big| \lambda(z) - \Lambda_{n,m}(z) \Big|.
\end{equation}

Now, by equation \eqref{eq: bound on anz}, for any $z\in U$ and $\ell\geq1$,
\[
|a_\ell(z)|
=
\left|
\Big[ \mathrm T(z)^\ell \boldsymbol v(z) \Big]_{\boldsymbol 0, \,z}
\right|
\leq
\frac{\overline M_{\Sigma\pi} \overline M_\ell \; \mathrm{artanh}(\bar{\rho})}{\bar{\rho}}\,
\big(\bar{\rho}\overline Q\big)^\ell.
\]
Letting $\Lambda_n(z) = \frac1d \sum_{\ell=0}^{n-1} a_\ell(z)$,
\begin{equation} \label{eq: crude tail estimate}
\Big| \lambda(z) - \Lambda_n(z) \Big|
\leq
\frac1d \cdot \frac{\overline M_{\Sigma\pi} \overline M_\ell \; \mathrm{artanh}(\bar{\rho})}{\bar{\rho}\big(1-\bar{\rho}\overline Q\big)} \big( \bar{\rho}\overline Q \big)^n.
\end{equation}

Fix $z\in U$. For each $r\in C$, we define
$
\pazocal{L}_r(z),\; \pazocal{R}_{r,m}(z),\; \pazocal{K}_{r,m}(z) : \;
\pazocal{H}(\mathbb D_{\bar{\rho}}) \to \pazocal{H}(\mathbb D_{\bar{\rho}})
$
in the same way as in subsection \ref{subsection: Integral operators and telescoping}:
\begin{gather*}
\big(\pazocal{L}_r(z) h\big)(w)
=
\frac{1}{2\pi i} \oint_{\partial \mathbb D_{\bar{\rho}}} \frac{1}{1-\zeta}\, h\left( f_r[z]\!\left( \frac{w}{\zeta} \right) \right)\,d\zeta, \\[4pt]
\big(\pazocal{R}_{r,m}(z) h\big)(w)
=
\frac{1}{2\pi i} \oint_{\partial \mathbb D_{\bar{\rho}}} \frac{\zeta^{m-1}}{1-\zeta}\, h\left( f_r[z]\!\left( \frac{w}{\zeta} \right) \right)\,d\zeta, \\[4pt]
\pazocal{K}_{r,m}(z) h
=
\pazocal{L}_r(z) h - \pazocal{R}_{r,m}(z) h.
\end{gather*}

The same argument as in Proposition \ref{prop: evaluating the truncation} applies: for any $h\in \pazocal H(\mathbb D_{\bar{\rho}})$ with $h(0) = 0$ and $\|h\|_{\mathrm{osc}} < \infty$, and $w\in \mathbb D_{\bar{\rho}}$,
\[
\big| \pazocal L_{r_1}(z)\cdots \pazocal L_{r_n}(z)h(w)
- \pazocal K_{r_1,m}(z)\cdots \pazocal K_{r_n,m}(z)h(w) \big|
\;\leq\;
\Big( (1+\bar{\rho}_m)^n - 1 \Big)\|h\|_{\mathrm{osc}}.
\]

For $k\geq 1$, define $h_k\in \pazocal H(\mathbb D_{\bar{\rho}})$ by $h_k(w) = -\frac{(-w)^k}{k}$. Then $h_k(0)=0$, and $\|h_k\|_{\mathrm{osc}} \leq \frac{2{\bar{\rho}}^k}{k}$. Therefore, for any $s\in C$ and $1\leq k\leq m-1$,
\begin{align*}
&\left| \Big( \mathrm T(z)^n \boldsymbol v(z) \Big)_{s,k} - \Big( \mathrm T_{m\times m}(z)^n \boldsymbol v(z) \Big)_{s,k} \right| \\
&=
\left|
\sum_{\mathbf r = (r_0,\ldots,r_n) \in C^{n+1}} c_{\mathbf r, s}(z)
\bigg(
\pazocal L_{r_1}(z)\cdots \pazocal L_{r_n}(z)h_k\big( f_{r_0}[z](0) \big)
- \pazocal K_{r_1,m}(z)\cdots \pazocal K_{r_n,m}(z)h_k\big( f_{r_0}[z](0) \big)
\bigg)
\right| \\
&\leq\;\;
\frac{\overline M_{\Sigma\pi}\,\overline Q^{n+1}}{\overline m_\pi} \Big( (1+\bar{\rho}_m)^n - 1 \Big) \frac{2\bar{\rho}^k}{k}.
\end{align*}
Also, by the same argument as in Proposition \ref{prop: evaluating the truncation}, we have, for every $s\in C$, $\ell\geq 0$ and $k\geq 1$,
\begin{align*}
\left| \Big(
\mathrm T(z)^\ell \boldsymbol v(z)
\Big)_{\!\!s,\,k} \right|
&\leq
\sum_{\mathbf r = (r_0,\ldots,r_\ell) \in C^{\ell+1}} \big| c_{\mathbf r,s}(z) \big| \, \left| \Big( T_{r_\ell}(z)\cdots T_{r_1}(z) v\big( f_{r_0}[z](0)\,;\,\ell_{r_0}(z, 0) \big) \Big)_k \right| \\
&\leq
\sum_{\mathbf r = (r_0,\ldots,r_\ell) \in C^{\ell+1}} \big| c_{\mathbf r,s}(z) \big| \,
\left| \big( \pazocal{L}_{r_1}(z)\cdots \pazocal{L}_{r_\ell}(z) h_k \big)\big( f_{r_0}[z](0) \big) \right|
\leq
\frac{\overline M_{\Sigma \pi} \overline Q^{\ell+1}}{\overline m_\pi}
\frac{2\bar{\rho}^k}{k}.
\end{align*}

Therefore,
\begin{align}
&\big| \Lambda_n(z) - \Lambda_{n,m}(z) \big|
\leq
\frac1d \sum_{\ell=1}^{n-1} \sum_{s\in C} | \pi_s(z) |
\left| \Big(
\mathrm T(z)^\ell \boldsymbol v(z) - \mathrm T_{m\times m}(z)^\ell \boldsymbol v(z)
\Big)_{\!\!s,\,0} \right| \nonumber \\
&\hspace{-5pt}\leq
\frac1d\sum_{\ell=1}^{n-1} \sum_{r,s\in C}
\big| \pi_r(z) Q_{r,s}(z) \big| \nonumber \\
&\hspace{30pt}
\left\{ \sum_{k=1}^{m-1}
\overline D^k \left| \Big(
\mathrm T(z)^{\ell-1} \boldsymbol v(z) - \mathrm T_{m\times m}(z)^{\ell-1} \boldsymbol v(z)
\Big)_{\!\!r,\,k} \right|
+
\sum_{k=m}^{\infty}
\overline D^k \left| \Big( \mathrm T(z)^{\ell-1}\boldsymbol v(z) \Big)_{\!\!r,\,k} \right|
\right\} \nonumber \\
&\hspace{-5pt}\leq
\frac1d \sum_{\ell=1}^{n-1} \overline M_{\Sigma\pi}\,\overline Q
\left\{
\sum_{k=1}^{m-1}
\overline D^k \frac{\overline M_{\Sigma\pi}\,\overline Q^{n-1}}{\overline m_\pi} \Big( (1+\bar{\rho}_m)^{\ell-1} - 1 \Big) \frac{2\bar{\rho}^k}{k}
+ \frac{\overline M_{\Sigma \pi} \overline Q^\ell}{\overline m_\pi}\sum_{k=m}^\infty \overline D^k\frac{2\bar{\rho}^k}{k}
\right\} \nonumber\\
&\hspace{-5pt}\leq
\frac1d \cdot \frac{2\overline M_{\Sigma\pi}^2\overline Q^n}{\overline m_\pi} \left(
\left( \log \frac{1}{1-\bar{\rho} \overline D} \right) \bar{\alpha}_{n,m}
+ \frac{(n-1)\bar{\rho}^m \overline D^m}{m(1-\bar{\rho}\overline D)}
\right). \label{eq: second error in crude derivative estimate}
\end{align}

The desired bound follows by combining this with equations \eqref{eq: derivative lyapunov exponent is bounded by the original by analyticity} and \eqref{eq: crude tail estimate} and using the triangle inequality.

\medskip

Next, take any $\vep>0$, and take $n$ large enough so that the right-hand side of \eqref{eq: crude tail estimate} is less than $\vep/2$. As we have seen in the proof of Theorem \ref{thm: complexity statement}, if $m$ is large enough so that $(n-1)\bar{\rho}_m \leq 1$, then $\bar{\alpha}_{n,m} \leq \frac{e}{2}(n-1)^2\bar{\rho}_m$. Then, for $m$ satisfying $(n-1)\bar{\rho}_m \leq 1$ and $m\geq n$, the bound in equation \eqref{eq: second error in crude derivative estimate} is exponentially small in $m$ due to $\bar{\rho} \overline Q < 1$ and $\overline D<1$. Take $m$ so that the bound is less than $\vep/2$. For these choices of $n,m$,
\[
\Big| \lambda^{\!\!(q)}(t_0) - \Lambda^{(q)}_{n,m}(t_0) \Big|
<\frac{q!}{c^q} \vep.
\]
Finally, only $O\big( (\log(1/\vep))^3 \big)$ arithmetic operations are needed to calculate $\Lambda^{(q)}_{n,m}(t_0)$ using equation \eqref{eq: inductive identity of finite approximation of general derivatives of lambda}. This completes the proof.
\end{proof}

This also completes the proof of Theorem \ref{thm: introduction: every derivative can be computed in polynomial time} from the introduction.

\appendix

\section{Appendix} \label{section: appendix}

\begin{proof}[Proof of Proposition \ref{prop: aperiodic reduction}]
Let $P$ be the transition matrix of $\Sigma$. Fix $i_\ast \in X$. For each $u \in \{0,\dots,d-1\}$, define
\[
X_u
\;=\;
\left\{ j \in X \;\setcond\; \text{There is } n\geq 0 \text{ with } (P^n)_{i_\ast\, j} > 0 \text{ and } n \equiv u \;\; (\mathrm{mod} \;\;d) \right\}.
\]

We first show that
\begin{equation} \label{eq: cyclic decomposition of X}
X
\;=\;
\bigsqcup_{u=0}^{d-1} X_u.
\end{equation}
For every $j\in X$, by the irreducibility of $P$, there is $n\geq 0$ such that $(P^n)_{i_\ast j}>0$, so $j\in X_u$ for some $u$. To prove that the union is disjoint, suppose that $j \in X_u \cap X_v$. Then, there are $m,n\geq 0$ such that
\[
(P^m)_{i_\ast j} > 0,
\qquad
(P^n)_{i_\ast j} > 0,
\qquad
m \equiv u \pmod d,
\qquad
n \equiv v \pmod d.
\]
Since $P$ is irreducible, there is $\ell \geq 1$ such that $(P^\ell)_{j i_\ast} > 0$. Then, $(P^{m+\ell})_{i_\ast i_\ast} > 0$ and $(P^{n+\ell})_{i_\ast i_\ast} > 0$. By the definition of the period, every return time to $i_\ast$ is divisible by $d$. Therefore $d \mid (m+\ell)$ and $d \mid (n+\ell)$. Thus $d \mid (m-n)$. This implies $u=v$, and the union is disjoint.

Next, we show that transitions move cyclically through the sets $X_u$. Let $i\in X_u$. By the definition of $X_u$, there is $n\geq 0$ with $(P^n)_{i_\ast i} > 0$ and $n \equiv u \pmod d$. Then if $P_{ij}>0$, we have $(P^{n+1})_{i_\ast j} > 0$, so $j \in X_{u+1}$, where the index is taken modulo $d$. Therefore,
\begin{equation} \label{eq: cyclic decomposition of period classes}
\Big( i\in X_u , \quad P_{ij} > 0 \Big)
\quad\Longrightarrow\quad
j\in X_{u+1}.
\end{equation}

We now consider the $d$-step dynamics on $X_0$. Let $i,j\in X_0$. Since $P$ is irreducible, there is $n\geq 1$ such that $(P^n)_{ij}>0$. By equations \eqref{eq: cyclic decomposition of X} and \eqref{eq: cyclic decomposition of period classes}, every path from $X_0$ to $X_0$ has length divisible by $d$. Thus $n=md$ for some $m\geq 1$, and therefore $\big( (P^d)^m \big)_{ij} = (P^n)_{ij} > 0$. This proves that $P^d|_{X_0}$ is irreducible, since we have $(P^d|_{X_0})^m = (P^{dm})|_{X_0}$.

We claim that $P^d|_{X_0}$ is also aperiodic. Fix $i\in X_0$. Since every return time from $i$ to itself is divisible by $d$,
\[
\left\{ n\geq 1 \setcond (P^n)_{ii}>0 \right\}
=
\left\{ md \setcond m\geq 1,\ (P^{md})_{ii}>0 \right\}.
\]
Since the gcd of the set on the left-hand side is $d$, we obtain
\[
\gcd \left\{ m\geq 1 \setcond (P^{md})_{ii}>0 \right\}
=
1.
\]
Therefore $P^d|_{X_0}$ is aperiodic.

Since $P^d|_{X_0}$ is irreducible and aperiodic, it is primitive. Therefore there exists $N\in \mathbb N$ such that
\begin{equation} \label{eq: P to the power of period is primitive}
(P^{dN})_{ab} > 0
\qquad
\text{for every } a,b\in X_0.
\end{equation}

Define
\[
\widetilde X
=
\left\{
(i_0,\dots,i_{d-1}) \in X^d \;\setcond\; \text{$i_u \in X_u$ for $0\leq u\leq d-1$, and $P_{i_u i_{u+1}} > 0$ for $0\leq u\leq d-2$} \right\}.
\]
For $\widetilde i = (i_0,\dots,i_{d-1}) \in \widetilde X$ and $\widetilde j = (j_0,\dots,j_{d-1}) \in \widetilde X$, define
\[
\widetilde P_{\; \widetilde i,\,\widetilde j}
\;\coloneqq\;
P_{i_{d-1} j_0}\,
P_{j_0 j_1}\cdots P_{j_{d-2} j_{d-1}}.
\]
Also, let
\[
\widetilde \Sigma
=
\left\{
(\widetilde x_n)_{n\geq 0}\in \widetilde X^{\mathbb N_0}
\;\middle|\;
\widetilde P_{\; \widetilde x_n,\widetilde x_{n+1}} > 0
\text{ for every } n\geq 0
\right\},
\]
and let $\widetilde \sigma: \widetilde \Sigma\to\widetilde \Sigma$ denote the left shift map. Finally, define $\widetilde A : \widetilde \Sigma \to \GL_2(\mathbb R)$ by
\[
\widetilde A(\widetilde x)
=
A_{i_{d-1}} \cdots A_{i_0}
\qquad
\text{whenever } \;
\widetilde x_0 = (i_0,\dots,i_{d-1}).
\]

The matrix $\widetilde P$ is row-stochastic. Indeed, for any $\widetilde i = (i_0,\dots,i_{d-1}) \in \widetilde X$, by equation \eqref{eq: cyclic decomposition of period classes} and row-stochasticity of $P$,
\[
\sum_{\widetilde j\in \widetilde X} \widetilde P_{\; \widetilde i,\widetilde j}
\;=\;
\sum_{j_0\in X_0} P_{i_{d-1} j_0}
\sum_{j_1\in X_1} P_{j_0 j_1}
\cdots
\sum_{j_{d-1}\in X_{d-1}} P_{j_{d-2} j_{d-1}}
\;=\; 1.
\]

Next, we show that $\widetilde P$ is primitive, and hence irreducible and aperiodic. Take arbitrary $\widetilde i = (i_0,\dots,i_{d-1}) \in \widetilde X$ and $\widetilde j = (j_0,\dots,j_{d-1}) \in \widetilde X$. Take any $a\in X_0$ with $P_{i_{d-1} a} > 0$. By equation \eqref{eq: P to the power of period is primitive}, we have $(P^{Nd})_{a j_0} > 0$. Therefore, choose any $P$-admissible path $(u_0,\dots,u_{Nd})\in X^{Nd+1}$ connecting $a = u_0$ and $j_0 = u_{Nd}$. For each $\ell = 1,\dots,N$, define
$
\widetilde u^{(\ell)}
=
\big(
u_{(\ell-1)d},
u_{(\ell-1)d+1},
\dots,
u_{\ell d -1}
\big)$. Then, each $\widetilde u^{(\ell)}$ belongs to $\widetilde X$, and
\[
\widetilde P_{\;\widetilde i,\,\widetilde u^{(1)}} > 0,
\qquad
\widetilde P_{\;\widetilde u^{(\ell)},\,\widetilde u^{(\ell+1)}} > 0
\quad\text{for } 1\leq \ell \leq N-1,
\quad \text{and} \quad
\widetilde P_{\;\widetilde u^{(N)},\,\widetilde j} > 0.
\]
Thus $(\widetilde P^{N+1})_{\widetilde i,\,\widetilde j} > 0$. Since $\widetilde i,\widetilde j\in \widetilde X$ were arbitrary, $\widetilde P$ is primitive. This proves item \textnormal{(1)}.

We now relate the accelerated system to the original one. Define
\[
E_0
=
\left\{ x=(x_n)_{n\geq 0}\in \Sigma \setcond x_0 \in X_0 \right\}.
\]
Since $p_i>0$ for every $i\in X$, and $X_0 \neq \varnothing$, we have
$
\mu(E_0)
=
\sum_{i\in X_0} p_i > 0$.
Let $\nu = \mu(\,\cdot\,|E_0)$. Because $p$ is stationary, the Markov measure $\mu$ is $\sigma$-invariant, hence $\sigma^d$-invariant. Also, by \eqref{eq: cyclic decomposition of period classes}, the set $E_0$ is $\sigma^d$-invariant. Therefore $\nu$ is $\sigma^d$-invariant.

Define
\[
\Phi : E_0 \to \widetilde \Sigma,
\qquad
\Phi(x)_n
=
(x_{nd},x_{nd+1},\dots,x_{nd+d-1}).
\]
By construction, $\Phi$ is a bijection from $E_0$ onto $\widetilde \Sigma$, and $\Phi \circ \sigma^d = \widetilde \sigma \circ \Phi$. Let
\[
\widetilde \mu \coloneqq \Phi_\ast \nu.
\]
Since $\nu$ is $\sigma^d$-invariant, $\widetilde \mu$ is $\widetilde \sigma$-invariant. We claim that $\widetilde \mu$ is the Markov measure associated to $\widetilde P$. First, let
\[
\widetilde \pi_{(i_0,\dots,i_{d-1})}
=
\frac{1}{\mu(E_0)}
\,
p_{i_0} P_{i_0 i_1}\cdots P_{i_{d-2} i_{d-1}}.
\]
Now, take any $\widetilde i^{(m)} = \big(i^{(m)}_0,\dots,i^{(m)}_{d-1}\big) \in \widetilde X$ for each $0\leq m\leq n$. Then
\begin{align*}
& \widetilde \mu \big( [\widetilde i^{(0)},\dots,\widetilde i^{(n)}] \big)
=
\nu\Big(
\big[
i^{(0)}_0,\dots,i^{(0)}_{d-1},
i^{(1)}_0,\dots,i^{(1)}_{d-1},
\dots,
i^{(n)}_0,\dots,i^{(n)}_{d-1}
\big]
\Big)
\\[4pt]
&\qquad =
\frac{1}{\mu(E_0)}
\,
p_{i^{(0)}_0}
\Big(
P_{i^{(0)}_0 i^{(0)}_1}
\cdots
P_{i^{(0)}_{d-2} i^{(0)}_{d-1}}
\Big)
\Big(
P_{i^{(0)}_{d-1} i^{(1)}_0}
\cdots
P_{i^{(1)}_{d-2} i^{(1)}_{d-1}}
\Big)
\cdots
\Big(
P_{i^{(n-1)}_{d-1} i^{(n)}_0}
\cdots
P_{i^{(n)}_{d-2} i^{(n)}_{d-1}}
\Big)
\\[4pt]
&\qquad =
\widetilde \pi_{\; \widetilde i^{(0)}}
\,
\widetilde P_{\; \widetilde i^{(0)},\,\widetilde i^{(1)}}
\cdots
\widetilde P_{\; \widetilde i^{(n-1)},\,\widetilde i^{(n)}}.
\end{align*}
In particular, for any $\widetilde j\in \widetilde X$,
\[
\widetilde \pi_{\;\widetilde j}
=
\widetilde \mu\big( \big[ \widetilde j\; \big] \big)
=
\widetilde \mu\big( \widetilde \sigma^{-1} \big[ \widetilde j\; \big] \big)
=
\sum_{\widetilde i\in \widetilde X} \widetilde \mu\big( \big[\widetilde i,\;\widetilde j\;\big] \big)
=
\sum_{\widetilde i\in \widetilde X} \widetilde \pi_{\;\widetilde i} \;\widetilde P_{\;\widetilde i,\, \widetilde j}.
\]
Thus $\widetilde \pi$ is the stationary probability vector of $\widetilde P$, and $\widetilde \mu$ is the Markov measure associated to $\widetilde P$.

Finally, let $x\in E_0$ and write $\widetilde x=\Phi(x)$. Then for every $n\geq 1$,
\[
\widetilde A^n(\widetilde x)
=
A_{x_{nd-1}} \cdots A_{x_0}
=
A^{nd}(x).
\]
Let
\[
F
=
\left\{
x\in \Sigma
\;\setcond\;
\lim_{n\to\infty}
\frac{1}{n}\log \|A^n(x)\|
=
\lambda_+(A,P)
\right\}.
\]
Then $\mu(F)=1$. Since $\nu(F\cap E_0)=1$, for $\nu$-a.e.\ $x\in E_0$,
\[
\lim_{n\to\infty}
\frac{1}{n}
\log \big\| \widetilde A^n(\Phi(x)) \big\|
=
\lim_{n\to\infty}
\frac{1}{n}
\log \|A^{nd}(x)\|
=
d\,\lambda_+(A,P).
\]
Pushing this forward by $\Phi$, we obtain \textnormal{(3)}.
\end{proof}

\section*{Acknowledgements}
I am deeply grateful to my supervisor, Masaki Tsukamoto, for his steady guidance, encouragement, and support throughout this work. I would also like to thank Ryokichi Tanaka for his expert advice, especially for pointing me to \cite{Bochi--Potrie--Sambarino}, from which Example 2 was obtained. I am also sincerely thankful to my family for their kindness and support.

ChatGPT (OpenAI) was used to check the proofs and carry out numerical experiments. In particular, it discovered the matrices, multicones, and coordinate changes in Example 1, which the author verified and incorporated.

\vspace{0.5cm}

\address{
Department of Mathematics, Kyoto University, Kyoto 606-8501, Japan}

\textit{E-mail}: \texttt{alibabaei.nima.28c@st.kyoto-u.ac.jp}

This work was supported by JSPS KAKENHI Grant Number 25KJ1473.


\begin{thebibliography}{99}



%\bibitem[Ali24]{Alibabaei}
%N.~Alibabaei,
%Weighted topological pressure revisited,
%{\it{Ergod. Th. \& Dynam. Sys.}}, \textbf{45} (2025), 34-70.

\bibitem[A26]{Alibabaei26}
N.~Alibabaei,
On the intersection of Cantor sets and products of random matrices,
arXiv:2512.02675 (2026).


\bibitem[A26-2]{Alibabaei26-2}
N.~Alibabaei,
Lyapunov exponents for random products of non-negative matrices,
arXiv:2602.23317 (2026).



\bibitem[ABY10]{Avila--Bochi--Yoccoz}
A.~Avila, J.~Bochi, J.~Yoccoz,
Uniformly hyperbolic finite-valued SL(2,R)-cocycles,
{\it{Comment. Math. Helv.}}, \textbf{85} (2010), no. 4, pp. 813-884.


%\bibitem[Bed84]{Bedford}
%T.~Bedford,
%Crinkly curves, Markov partitions and box dimension in self-similar sets. {\it{Ph.D. Thesis}}, University of Warwick, 1984.


%\bibitem[BF12]{Barral--Feng}
%J.~Barral and D.~J.~Feng,
%Weighted thermodynamic formalism on subshifts and applications,
%{\it{Asian J. Math.}} \textbf{16} (2012), 319–352.


\bibitem[BPS19]{Bochi--Potrie--Sambarino}
J.~Bochi, R.~Potrie, A.~Sambarino,
Anosov representations and dominated splittings,
{\it{J. Eur. Math. Soc.}}, \textbf{21} (2019), no. 11, pp. 3343–3414.


%\bibitem[BKM85]{Boyle--Kitchens--Marcus}
%M.~Boyle, B.~Kitchens and B.~Marcus,
%A note on minimal covers for sofic systems,
%{\it{Proc. Amer. Math. Soc.}} \textbf{95} (1985), 403-411.


%\bibitem[BL]{Bougerol--Lacroix}
%P.~Bougerol, J.~Lacroix,
%{\it{Products of random matrices with applications to Schr\"odinger operators}},
%Boston: Birkh\"auser (1985).


%\bibitem[Din70]{Dinaburg}
%E.~I.~Dinaburg,
%A correlation between topological entropy and metric entropy,
%{\it{Dokl. Akad. Nauk SSSR}} \textbf{190} (1970) 19-22.


%\bibitem[Dow11]{Downarowicz}
%T.~Downarowicz,
%{\it{Entropy in dynamical systems}},
%Cambridge University Press, Cambridge, \textbf{MA} 2011.



\bibitem[DDGK25]{DDGK25}
P.~Duarte, M.~Durães, T.~Graxinha, S.~Klein,
Random 2D Linear Cocycles \RN{1}: Dichotomic Behavior,
arXiv:2503.21050 (2025).


%\bibitem[Fe11]{Feng}
%D.~J.~Feng,
%Equilibrium states for factor maps between subshifts,
%{\it{Adv. Math.}} \textbf{226} (2011), 2470–2502.


%\bibitem[Fe24]{Z.Feng}
%Z.~Feng,
%On the coincidence of the Hausdorff and box dimensions for some affine-invariant sets,
%arXiv:2405.03213

%\bibitem[FH16]{Feng--Huang}
%D.~J.~Feng, W.~Huang, 
%Variational principle for weighted topological pressure,
%{\it{J. Math. Pures Appl.}} \textbf{106} (2016), 411-452.

\bibitem[FKe60]{Furstenberg--Kesten}
H.~Furstenberg, H.~Kesten,
Products of Random Matrices,
{\it{Ann. Math. Stat.}} \textbf{31} (1960), 457-469.

\bibitem[FKi83]{Furstenberg--Kifer}
H.~Furstenberg, Y.~Kifer,
Random matrix products and measures on projective spaces,
{\it{Israel J. Math.}} \textbf{46} (1983) 12-32.


%\bibitem[FV25]{Fan--Vebitskiy}
%A.~H.~Fan, E.~Verbitskiy,
%Computation of Lyapunov exponents of matrix products,
%arXiv:2501.11941 (2025).

%\bibitem[GR85]{Guivarch--Raugi}
%Y.~Guivarc'h, A.~Raugi,
%Frontière de Furstenberg, propriétés de contraction et théorèmes de convergence,
%{\it{Z. Wahrscheinlichkeitstheorie verw. Gebiete}} \textbf{69} (1985), 187–242.


%\bibitem[Gm71]{Goodman}
%T.~N.~T.~Goodman,
%Relating topological entropy and measure entropy, 
%{\it{Bull. Lond. Math. Soc.}} \textbf{3} (1971) 176-180.


%\bibitem[Gw69]{Goodwyn}
%L.~W.~Goodwyn,
%Topological entropy bounds measure-theoretic entropy,
%{\it{Proc. Amer. Math. Soc.}} \textbf{23} (1969) 679-688.


%\bibitem[H75]{Hawkes}
%J.~Hawkes,
%Some algebraic properties of small sets,
%{\it{Q. J. Math. Oxf.}} \textbf{26} (1975), 195-201.

\bibitem[JM19]{Jurga-Morris}
N.~Jurga, I.~Morris,
Effective estimates on the top Lyapunov exponents for random matrix products,
{\it Nonlinearity} \textbf{32} (2019), no.~11, 4117-4146.




%\bibitem[KP91]{Kenyon--Peres: translated Cantor sets}
%R.~Kenyon, Y.~Peres, 
%Intersecting random translates of invariant Cantor sets,
%{\it{Invent. math.}} \textbf{104} (1991), 601-629.



%\bibitem[KP96]{Kenyon--Peres}
%R.~Kenyon, Y.~Peres, 
%Measures of full dimension on affine-invariant sets,
%{\it{Ergod. Th. \& Dynam. Sys.}} \textbf{16} (1996), 307-323.


%\bibitem[KP96-2]{Kenyon--Peres: sofic}
%R.~Kenyon, Y.~Peres, 
%Hausdorff dimensions of sofic affine-invariant sets,
%{\it{Israel J. Math.}} \textbf{94} (1996) 157-178.



\bibitem[K73]{Kingman}
J.~Kingman,
Subadditive ergodic theory,
{\it{Ann. Probab.}} \textbf{1} (1973), 883-909.




%\bibitem[LN]{Lemmens--Nussbaum}
%B.~Lemmens, R.~Nussbaum,
%{\it{Nonlinear Perron–Frobenius Theory}},
%Cambridge: Cambridge University Press (2012)




\bibitem[MV15]{Malheiro--Viana}
E.~C.~Malheiro, M.~Viana,
Lyapunov exponents of linear cocycles over Markov shifts,
{\it Stoch. Dyn.} \textbf{15} (2015), no.~3, 1550020.



%\bibitem[McM84]{McMullen}
%C.~McMullen,
%The Hausdorff dimension of general Sierpinski carpets,
%{\it{Nagoya Math. J.}} 96 (1984), 1–9.



\bibitem[M05]{Mairesse}
J.~Mairesse,
Random Walks on Groups and Monoids with a Markovian Harmonic Measure,
{\it{Electron. J. Probab.}} \textbf{10} (2005), 1417 - 1441.



\bibitem[Pe92]{Peres92}
Y.~Peres,
Domains of analytic continuation for the top Lyapunov exponent,
{\it{Ann. Inst. H. Poincar\'e Probab. Statist.}} \textbf{28} (1992), 131–148.




\bibitem[Po10]{Pollicott10}
M.~Pollicott,
Maximal Lyapunov exponents for random matrix products,
{\it{ Invent. math. }} \textbf{181} (2010), 209–226.



\bibitem[Po21]{Pollicott20}
M.~Pollicott,
Effective estimates of Lyapunov exponents for random products of positive matrices,
{\it{Nonlinearity}} \textbf{34} (2021), 6705-6718.




%\bibitem[Ru73]{Ruelle73}
%D.~Ruelle,
%Statistical mechanics on a compact set with $Z^v$ action satisfying expansiveness and specification,
%{\it{Trans. Amer. Math. Soc.}} \textbf{187} (1973), 237–251.


%\bibitem[Ru76]{Ruelle76}
%D.~Ruelle,
%Zeta-functions for expanding maps and Anosov flows,
%{\it{Invent. Math.}} \textbf{34} (1976), 231–242.


\bibitem[Ru79]{Ruelle79}
D.~Ruelle,
Analycity Properties of the Characteristic Exponents of Random Matrix Products,
{\it{Adv. Math.}} \textbf{32} (1979), 68-80.




%\bibitem[S00]{Strichartz}
%R. S. Strichartz,
%Evaluating integrals using self-similarity,
%{\it{Amer. Math. Monthly}} \textbf{107(4)} (2000), 316–326


%\bibitem[Tsu22]{Tsukamoto}
%M.~Tsukamoto,
%New approach to weighted topological entropy and pressure,
%{\it{Ergod. Th. \& Dynam. Sys.}} \textbf{43} (2023) 1004-1034.


\bibitem[Vil]{Villani}
C.~Villani,
{\it{Optimal Transport, Old and New}},
Springer Berlin, Heidelberg, 2009.




%\bibitem[Vis00]{Viswanath}
%D.~Viswanath,
%Random Fibonacci Sequences and the Number 1.13198824...,
%{\it{Math. Comp.}} \textbf{69}, no. 231 (2000), 1131–55. 



%\bibitem[Wal75]{Walters75}
%P.~Walters,
%A variational principle for the pressure of continuous transformations, {\it{Amer. J. Math.}} \textbf{97} (1975), 937–971.


%\bibitem[Wal82]{Walters}
%P.~Walters, 
%{\it{An introduction to ergodic theory}},
%Springer-Verlag, New York, 1982.


%\bibitem[We82]{Weiss}
%B.~Weiss,
%Subshifts of finite type and sofic systems,
%{\it{Monatsh. Math.}} \textbf{77} (1973), 462-474


%\bibitem[Ya11]{Yayama}
%Y.~Yayama,
%Applications of a relative variational principle to dimensions of nonconformal expanding maps,
%{\it{Stoch. Dyn.}} \textbf{11} (2011), 643-679.



\end{thebibliography}
\end{document}